%% file: SU2-EW-P-arXiv.tex
\documentclass[12pt,reqno]{amsart}
\usepackage{times,amsfonts,amsmath,amstext,amsbsy,amssymb,amsopn,amsthm,upref,amscd,url}
\usepackage[mathscr]{euscript}
\usepackage[T1]{fontenc}
\usepackage[all]{xy}
\usepackage[pdftex]{graphicx}
\usepackage{color}
\usepackage{stmaryrd}
\usepackage{mathtools} 
\usepackage{tikz-cd}
\usepackage[backref,bookmarks=true,colorlinks=true,pdfstartview=FitV, linkcolor=blue, citecolor=red, urlcolor=green]{hyperref} 

\usepackage{mathrsfs}
\usepackage[scr=rsfs,cal=boondox]{mathalfa} 

\newtheorem{theorem}{Theorem}[section]
\newtheorem{lemma}[theorem]{Lemma}

\newtheorem{proposition}[theorem]{Proposition}

\numberwithin{equation}{section}

\newtheorem*{thm*}{Theorem}
\newtheorem*{remark*}{Remark}

\theoremstyle{definition}

\newtheorem{remark}[theorem]{Remark}

\newcommand\R{\mathbb{R}}
\newcommand\Z{\mathbb{Z}}

\newcommand{\C}{\mathbb{C}}  

\renewcommand{\o}{\mathsf}    

\renewcommand{\P}{\o{P}}		
\renewcommand{\H}{\o{H}}   
\newcommand\tr{\o{tr}}           
\newcommand{\Hom}{\o{Hom}}
\newcommand{\Inn}{\o{Inn}}
\newcommand{\Out}{\o{Out}}
\newcommand{\Aff}{\o{Aff}}
\newcommand{\Ker}{\o{Ker}}
\newcommand{\Homeo}{\o{Homeo}}
\newcommand{\Ad}{\o{Ad}}
\newcommand{\bb}{\mathfrak{B}}        
\newcommand{\B}{\o{B}}       

\newcommand{\TTT}{\Phi}

\newcommand\SL{\o{SL}}
\newcommand\SU{\o{SU}}
\newcommand{\T}{\o{T}} 
\newcommand{\Aut}{\o{Aut}} 

\newcommand{\XX}{\mathfrak{X}}
\newcommand\g{\mathfrak g}
\newcommand\gC{\g^\C}
\newcommand{\Ll}{\mathfrak{L}}
\newcommand{\UUU}{\mathfrak{U}}
\newcommand{\IndexSet}{\mathfrak{I}}

\newcommand{\Ca}{\mathscr{C}} 
\newcommand{\Cp}{\overline{\mathscr{C}}} 
\newcommand{\mB}{\mathscr{B}} 

\newcommand{\vv}{\mathbf{v}}
\newcommand{\Asf}{\o{A}}   
  
\newcommand{\ty}{\widetilde{y}}

\newcommand{\ts}{\widetilde{s}}
\newcommand{\tS}{\widetilde{\surf}} 

\newcommand{\Aaa}{\mathscr{A}} 
\newcommand{\Aa}{\mathcal{A}} 

\newcommand{\Dsf}{\o{D}} 
\newcommand{\inv}{^{-1}}
\newcommand{\GL}{\o{GL}}

\newcommand{\APatch[1]}{\Asf^{(#1)}}    
\newcommand{\AChart[1]}{\Aa^{(#1)}}   
\newcommand{\pI}{p_\infty}                    
\newcommand{\qI}{q_\infty}   
\newcommand{\ldiv}{(\!\!\mathbf(}							
\newcommand{\rdiv}{)\!\!\mathbf)}
\newcommand{\Pdd}{\ddot{\P}}			
\newcommand{\Pddd}{\dddot{\P}}			
\newcommand{\tPddd}{(\Pddd)^\sim}			

\newcommand{\divisor[1]}{\ldiv #1\rdiv} 
\newcommand{\pointdivisor[1]}{\mathbf{[}#1\mathbf{]}} 

\newcommand{\Sphere}{\mathbb{S}^2} 
\newcommand{\Torus}{\mathbb{T}^2} 

\newcommand{\group}{<} 
\newcommand{\cho}{\chi^{\o{orb}}} 
\newcommand{\pio}{\pi_1^{\o{orb}}} 
\newcommand{\EW}{\o{EW}}
\newcommand{\Plat}{\o{Plat}}

\newcommand{\ZdX}{\big( Z\, dX - X\, dZ\big)} 

\newcommand{\SUtwo}{\o{SU}(2)}
\newcommand{\SOtwo}{\o{SO}(2)}

\newcommand{\SLtwoR}{\o{SL}(2,\R)}
\newcommand{\sutwo}{\mathfrak{su}(2)}
\newcommand{\sltwoC}{\mathfrak{sl}(2,\C)}

\newcommand{\Qu}{\o{Q}}
\newcommand{\Ham}{\mathbb{H}} 
\newcommand{\biTet}{\o{BTet}}
\newcommand{\biOct}{\o{BOct}}
\newcommand{\bi}{\mathbf{i}}
\newcommand{\bj}{\mathbf{j}}
\newcommand{\bk}{\mathbf{k}}
\newcommand{\bu}{\mathbf{u}}
\newcommand{\bOne}{\mathbf{1}}

\newcommand{\V}{\o{V}} 

\newcommand{\Cech}{\v{C}ech~}

\newcommand{\Uone}{\o{U}(1)}
\newcommand{\Ur}{\o{U}(r)}
\newcommand{\GLrC}{\GL(r,\C)}
\newcommand{\UU}{\o{U}}
\newcommand{\Rr}{\mathcal{R}} 
\newcommand{\rr}{\mathcal{r}} 

\newenvironment{bbmatrix}{%
    \left\llbracket \begin{matrix}}{%
    \end{matrix}\right\rrbracket}

\newcommand{\suAr}{\sutwo_{\Adrho}}
\newcommand{\Fermat}{\mathfrak{F}}

\newcommand{\surf}{\Sigma} 
\newcommand{\MCG}{\Mod(\surf_g)} 
\newcommand{\mm}{\mathcal{M}}
\newcommand{\RMS}{\mm_g} 
\newcommand{\TT}{\mathfrak{T}} 
\newcommand{\TTg}{\TT_g} 

\newcommand{\RepG}{\XX\big(\pi_1(\surf_g),G\big)}     
 
\newcommand{\RepGt}{\XX_\tau(\pi_1(\surf_g),G\big)}   

\renewcommand{\top}{\o{top}} 
\newcommand{\EGS}{\mathfrak{E}_G(\surf)} 
\newcommand{\U}{\mathbb{U}}
\newcommand{\Uu}{\mathscr{U}} 
\newcommand{\URk}{\U^\top\RMS}
\newcommand{\UE}{\U^\top_\tau\EGS}
\newcommand{\bo}{\overline{\omega}} 
\newcommand{\zbar}{\bar{z}}
\newcommand{\ddz}{\frac{\partial}{\partial z} }
\newcommand\ddt{\frac{d}{d t}}
\newcommand{\kappab}{\bar{\kappa}}
\newcommand{\kappainv}{\kappa\inv}
\newcommand{\dbar}{\overline{\partial}}

\newcommand{\DD}{\Delta} 

\newcommand{\Mod}{\o{Mod}}

\newcommand{\Mg}{\mathcal{M}_g} 
\newcommand{\Tg}{\mathcal{H}_g} 

\newcommand{\Alt}{\o{Alt}}
\newcommand{\Teich}{Teichm\"uller}
\newcommand{\TS}{\TT(\surf)}
\newcommand{\Hodge}{\H^{(1,0)}(M)}

\newcommand{\Beltrami}{L^\infty (M,\kappainv \otimes \kappab)}

\newcommand{\betab}{\overline{\beta}}

\newcommand{\Vv}{\mathcal{V}}
\newcommand{\Jj}{\mathbb{J}}

\newcommand{\hh}{\mathsf{h}} 

\newcommand{\Up}{\Upsilon_t^*}
\renewcommand{\Re}{\o{Re}}

\newcommand{\Adrho}{\Ad\,\rho}
\newcommand{\hrho}{\widehat\rho}
\newcommand{\trho}{\widetilde\rho}

\newcommand{\Lu}{\Ll^\C_\bu}

\begin{document}

\title[Veech groups and character varieties]
{Infinitesimal random dynamics of certain Veech groups on $\mathsf{SU}(2)$- character varieties}

\author[Forni]{Giovanni Forni}
\address{Department of Mathematics, University of Maryland, College Park, MD 20742, USA
and Laboratoire AGM - Analyse G\'eom\'etrie Mod\'elisation at CY Cergy Paris Universit\'e, 
Adolphe Chauvin, 95302 Cergy-Pontoise, France.}
\email{gforni@umd.edu}

\author[Goldman]{William M. Goldman}
\address{Department of Mathematics, University of Maryland, College Park, MD 20742, USA
and Simons-Laufer Mathematical Sciences Institute, Berkeley CA 94720 USA}
\email{wmg@umd.edu}

\author[Lawton]{Sean Lawton}
\address{Department of Mathematical Sciences, George Mason University, 4400 University Drive, Fairfax, Virginia 22030, USA}
\email{slawton3@gmu.edu}

\author[Matheus]{Carlos Matheus}
\address{CMLS, CNRS, \'Ecole polytechnique, Institut Polytechnique de Paris, 91128 Palaiseau Cedex, France.}
\email{carlos.matheus@math.cnrs.fr.}



\date{\today}

\begin{abstract}
Almost 20 years ago, the first and fourth authors found examples of $\SLtwoR$-invariant subbundles of Hodge bundles over Teichm\"uller curves having maximally degenerate Lyapunov spectrum. For these same surfaces, we show that a natural non-Abelian analogue has no zero Lyapunov exponents.
\end{abstract}

\maketitle

\tableofcontents

\section*{Introduction}

The mapping class group $\MCG$ of a closed orientable surface $\surf_g$ of genus $g>1$ is a fundamental mathematical object, which acts on various spaces associated with $\surf$, such as the $G$-character variety
\[\RepG :=\Hom(\pi_1(\surf_g),G)/\Inn(G), \] 
for any compact connected Lie group $G$.
By combined results of Goldman~\cite{Gold2} and Huebschmann~\cite{Hueb} this action preserves a finite smooth probability measure,
which, by Goldman-Xia~\cite{Gold6,GoldmanXia11} and Pickrell-Xia~\cite{PickrellXia02}, is $\MCG$-ergodic on each connected component of $\RepG$.

This paper is the third in a series of papers applying Teichm\"uller dynamics to this setting. The program was initiated in Forni-Goldman~\cite{FG17}, where the $\MCG$-dynamics on $\RepG$ is replaced by its {\em suspension,\/} a cocycle $\Phi$ over the \Teich~ flow  on $\U\RMS$, where $\RMS$ denotes the Riemann moduli space and $\U\RMS \to \RMS$ its \Teich~ unit  sphere bundle.

Here is the construction of $\Phi$. Let $\RepGt$ denote a connected component of $\RepG$ and $\U^\top\TTg$ the top stratum of the \Teich~ unit sphere bundle over $\TTg$, identified with holomorphic quadratic differentials with $4g-4$ simple zeroes. The mapping class group $\MCG$ acts on $\U^\top\TTg$ with quotient the \Teich~ unit sphere bundle of $\RMS$. Furthermore $\MCG$ acts on $\RepGt$.
The quotient of 
\[
\U^\top\TTg \times \RepGt \]
 by the diagonal action of $\MCG$ is the total space $\EGS$ of a $\RepGt$-fibration
\[
\EGS \longrightarrow \TTg/\MCG = \RMS. \]
Let $\UE$ denote the pullback of this bundle over $\Mg$ under the fibration 
\[
\URk \longrightarrow \Mg \]
and $\Phi$ the horizontal lift of the $\SLtwoR$-action on $\URk$.

The main result of ~\cite{FG17} is the following:
\begin{thm*}
Let $\URk$ denote the top stratum of $\U\RMS$ and let $\RepGt$ be a connected component of $\RepG$. The flow $\Phi^{\top}_\tau$ on $\UE$  is strongly mixing with respect to a smooth invariant probability measure on $\UE$, 
constructed as the local product of the Masur-Veech measure on  $U\RMS$ and the symplectic measure on $\RepGt$.
\end{thm*}

In this paper we restrict this cocycle to \Teich~discs defined by translation surfaces.
(Compare Saadi~\cite{Saadi1,Saadi2} for related results.)
We prove the nondegeneracy of the Lyapunov spectrum for a pair of particularly interesting Veech surfaces in $\RMS$ for $g=3,4$ 
(for which the Kontsevich-Zorich cocycle degenerates),
and some special representations $\pi_1(\Sigma_g)\to G$ for the Lie group $G = \SUtwo$. The genus three Riemann surface is the {\em  Eierlegende Wollmilchsau\/} ($\EW$), which is also the Fermat quartic plane curve,
and its sister genus four Riemann surface is the {\em Platypus} ($\Plat$). 

Let $\mu_0$  denote either of the  $\SL(2, \R)$-invariant measures obtained from the canonical measures on the $\SL(2, \R)$-orbit of $\EW$ or $\Plat$ and from the $\Aff({\EW})$-invariant or
$\Aff({\Plat})$-invariant representations, respectively denoted $\rho$ and $\rho_0$, by the suspension construction (see Sections \ref{sec-2}, \ref{sec:EW} and \ref{sec-3} for definitions of $\rho$ and $\rho_0$
for $\EW$ and $\Plat$). 

Here are our main theorems.

\begin{theorem} 
\label{main_thm:EW} 
In the case of the $\EW$, 
the Lyapunov spectrum of the lift of the Teichm\"uller flow 
with respect to $\mu_0$ has $6$ strictly positive $($and $6$ strictly negative$)$ fiber Lyapunov exponents. 
It follows that $\mu_0$ is non-uniformly hyperbolic for the lift of the Teichm\"uller flow to the moduli space of flat bundles.
\end{theorem} 
In fact, 
we prove the rank of the second fundamental form of the relevant variation of Hodge structure  is maximal 
(equal to $6$), 
hence the non-vanishing of all Lyapunov exponents
follows from a result of S.~Filip \cite{Fil17} after passing to a stratum of finite covers (to untwist the representation).

\smallskip
We obtain similar results for $\Plat$.

\begin{theorem} 
\label{main_thm:P} In the case of $\Plat$, 
the Lyapunov spectrum of the lift of the Teichm\"uller flow with respect to  $\mu_0$ has  
$9$ strictly positive $($and $9$ strictly negative$)$ fiber Lyapunov exponents. 
It follows that $\mu_0$ is non-uniformly hyperbolic for the lift of the Teichm\"uller flow to the moduli space of flat bundles.
\end{theorem} 
Specifically we prove that the rank of the second fundamental form of the relevant variation of Hodge structure  is not maximal, 
and equals $6$ with a kernel of dimension $3$. 
Filip's theorem~\cite{Fil17} implies that there are {\it at least} a number of strictly positive exponents equal to the rank. 
On the other hand, this bound (via rank) is not enough to recover the statement above and, for this reason, we complete the proof of the previous theorem thanks to a Zariski closure calculation. Therefore, we have established the first example (known to us) of strict inequality in Filip's lower bound for the number of strictly positive Lyapunov exponents in terms of the rank of the second fundamental form. 

These results are notable because Forni and Forni-Matheus showed total degeneracy for these translation surfaces,
answering a famous question of Veech.

\section*{Acknowledgments}

Lawton thanks the Simons Foundation for support.  
Forni was supported by the NSF grant DMS 2154208.
Goldman was supported by NSF grants DMS 1709791 and DMS 2203493,
and wishes to acknowledge helpful conversations
with  Francisco Arana Herrera, 
David Eisenbud,
Patrick Brosnan,
Simion Filip,
Chris Judge,
Richard Wentworth,
and Scott Wolpert. 
Goldman also gratefully acknowledges support and hospitality from the Simons-Laufer
Mathematics Institute in Spring 2026.

\section{Background} 

\subsection{Translation surfaces}
A compact connected surface (no boundary) $\Sigma$ is a {\it Euclidean cone surface} if every point $p\in \Sigma$ has a neighborhood isometric to a Euclidean cone of angle $\theta(p)$ 
and all but finitely many $\theta(p)$'s are $2\pi$.
A point $p$ for which $\theta(p)\neq 2\pi$ is called a {\em cone point.\/}
This structure is equivalent to a singular Riemannian metric whose curvature distribution $k$ is
concentrated at the cone points: 
\[ 
k = \sum_{p}  \big(2\pi - \theta(p)\big) \delta_p. \]
Suppose that these angles $\theta(p)$'s are integer multiples of $2\pi$.
A Euclidean cone surface is a {\em translation surface\/} if and only if 
the local coordinate changes in overlapping coordinate patches are required to be local isometries
which preserve the coordinate directions;
compare Athreya-Masur~\cite{MR4783430}.

Translation surfaces admit an atlas of charts whose transition functions (away from cone points) are Euclidean translations.  
Such surfaces are also equivalent to pairs $(M,\omega)$ where $M$ is a Riemann surface (irreducible smooth projective curve over $\C$ or, equivalently, a connected compact holomorphic 2-manifold) and $\omega$ is a nonzero Abelian differential on $M$, 
that is, a nonzero holomorphic $1$-form.
In other words, 
$\omega$ is a global section of the {\em canonical bundle\/} $\kappa$ of the 
Riemann surface $M$, 
the  (holomorphic) cotangent bundle  $\o{T}^*M$. 
We can construct translation surfaces by identifying pairs of sides of a collection of plane polygons by translations.
The vectors defining the translations are the periods of $\omega$ along curves effecting
the identifications. 
Equivalently, a covering space $\widetilde{M}\to M$ exists, 
where the identifications generate the group of deck transformations,
and the embeddings of the polygons in the plane define a developing map 
$\widetilde{M}\to \C$. 

Closely related are {\em half-translation surfaces,\/} 
where the cone angles are only required to be integer multiples of $\pi$.
Equivalently, 
the local isometries are allowed to include {\em point-symmetries\/} 
as well as translations.
A point-symmetry is a reflection (or inversion) about points, 
which in local coordinates look like 
\[ 
z\longmapsto 2z_0 - z \]
where $z_0$ is the coordinate of the fixed point.
These correspond to nonzero {\em holomorphic quadratic differentials,\/} 
sections of the tensor square of (holomorphic) cotangent bundle  
\[
\kappa^2 = \otimes^2\o{T}^*M =  \T^*M\otimes_\C \T^*M.\]
If $\omega$ is a nonzero Abelian differential, then 
$\omega\otimes\omega = \omega^2$ is a a holomorphic quadratic
differential (but not every quadratic differential is the square of an Abelian differential).

The moduli space $\Mg$ of genus $g$ surfaces and
the moduli space $\Tg$  of genus $g$ translation surfaces 
are each quasiprojective varieties over $\C$. 
Moreover they are connected K\"ahler orbifolds. 

Forgetting the Abelian differential defines a natural mapping 
\[
\Tg\to \Mg \] 
expressing  $\Tg$ as the complement of the zero-section in
the total space of a holomorphic rank $g$ vector bundle (the zero 1-form does not give a translation structure). 
The fiber of this vector bundle over a point in $\Mg$ identifies with the vector space $\H^{1,0}(M )$ of holomorphic $1$-forms 
(Abelian differentials) on a Riemann surface $M$ corresponding to the point in $\Mg$.  

The group  $\SL(2,\R)$  
acts on $\Tg$ by mapping the collection of planar polygons $\{P_i\}_{i\in\mathcal{I}}$ 
associated to a translation surface to the collection $\{g(P_i)\}_{i\in\mathcal{I}}$ for 
$g\in \SL(2,\R)$. 
The flows (orbits) corresponding to the subgroup of diagonal matrices
\[
\left(\begin{array}{cc}e^t&0\\0&e^{-t}\end{array}\right)\] are called the {\it Teichm\"uller flows} since they projects to geodesics in $\Mg$ with respect to the Teichm\"uller metric.
It follows from Veech~\cite{Veech1986} that 
the Teichm\"uller flow is non-uniformly hyperbolic.
This implies strong chaotic dynamical behavior for the full action of the mapping class group.

A closed orbit of this $\SL(2,\R)$-action projects to an 
holomorphic curve in $\Mg$ which is isometrically immersed with respect to the Teichm\"uller metric; 
such curves are called {\it Teichm\"uller curves}.  
A translation surface defining the orbit is called a {\it Veech surface}.

Pushing forward the bi-invariant Haar measure on $\SLtwoR$ defines an $\SLtwoR$-invariant measure on
$\U\TT$ which for a Veech surface defines an $\SLtwoR$-invariant probability measure.
\subsection{Lyapunov exponents}
We recall the notion of Lyapunov exponents for a cocycle over 
an ergodic measure-preserving flow.
By a {\em flow\/} on a 
locally compact Hausdorff space $X$ we mean a homomorphism
\begin{align*}
\R & \longrightarrow   \Homeo(X) \\
t &\longmapsto F_t 
\end{align*}
(where $\Homeo(X)$ is given the compact-open topology)
which we denote
\[ 
F := \{F_t\}_{t\in\R}.\]
That is,
$F_0= \o{Id}$ and $F_s \circ F_t = F_{s+t}$.
Furthermore we require that each $F_t$ 
preserves a Borel ergodic probability measure $\mu$ on $B$.

Now we define {\em cocycles\/} over $F$.
Suppose that 
\[
E \xrightarrow{~\pi~} B \] is a vector bundle of rank $d$.
A {\em cocycle over a flow $F$\/} 
is a one-parameter family $\{Z_t\}_{t\in\R}$ where $\pi\circ Z_t=F_t$ and 
\[
E \xrightarrow{~Z_t~} E \]
restricts to a linear map on each fiber of $\pi$
with $Z_0= \o{Id}$ and $Z_s \circ Z_t = Z_{s+t}$.

Let $\Vert\ \cdot \Vert_{b\in B}$ be a family of norms on the fibers of $\pi$.  
If there exists a non-zero vector $v\in \pi^{-1}(x)$ such that 
\[ 
\lambda:=\lim_{t\to \infty}\frac{1}{t}\log\Vert Z_t(v)\Vert_{F_t(x)} \] 
we say $\lambda$ is a {\it Lyapunov exponent} for $x$.
By the Oseledets multiplicative ergodic theorem, 
 the above limit exists for almost all $x$ with respect to any invariant probability measure.
Compare Wilkinson~\cite{Wilkinson} and Filip~\cite{FilipMET}  for expositions of this theory.

The Teichm\"uller flow lifts to the flat vector bundle over $\Mg$,  
whose fiber over a point in $\Mg$ represented by a Riemann
surface $M$ identifies with  the de Rham cohomology group $\H^1(M,\C)$.
This lift is called the {\it Kontsevich-Zorich cocycle}.

Since the Kontsevich-Zorich cocycle is defined on a symplectic vector bundle 
(of complex rank $2g$), 
its Lyapunov spectrum is symmetric:
\[ 
\lambda_1 > \lambda_2 \geq \cdots\geq \lambda_g \geq 0 \geq -\lambda_g \geq 
\cdots \geq -\lambda_2 > -\lambda_1. \]
The first result (the ``spectral gap'' $\lambda_1 > \lambda_2$) is due to Veech~\cite{Veech1986},
for a class of measures which includes the Masur-Veech measure.
According to Forni~ \cite{Forni02}  
$\lambda_1, \dots, \lambda_g$ 
are positive for the Masur-Veech measures on the strata of $\mathcal{H}_g$.
Avila and Viana~\cite{AvilaViana07} showed they are all distinct.
This corroborates the observations that Teichm\"uller flow is non-uniformly hyperbolic 
(almost random) 
and cocycles over random systems tend to have distinct, nonzero Lyapunov exponents.

In contrast, 
$\SLtwoR$-invariant subbundles of the Hodge bundle exist, 
for certain surfaces, 
where the Lyapunov spectrum over the geodesic flow on the Teichm\"uller curve are maximally degenerate: 
\[ 
\lambda_2 = \lambda_3 = \dots = \lambda_g = 0.\]  
Compare Forni-Matheus-Zorich~\cite{FMZ14} and Forni-Matheus~\cite{Intro}.
In particular, this happens for the genus 3 curve $\EW$ and its sister genus 4 curve $\Plat$.  See Herrlich-Schmith\"usen~\cite{HS} and Forni-Matheus~\cite{Intro} for a good treatments of $\EW$.
       
These curves are the only curves which have this property by 
M\"oller~\cite{MR2787595} and Aulicino-Norton~\cite{AulicinoNorton}.
The only $\SLtwoR$-invariant probability measures with maximally degenerate
Lyapunov spectrum are the measures supported on the $\SLtwoR$-orbit
of the $\EW$ and $\Plat$ induced by Haar measure on $\SLtwoR$.

\subsection{Character varieties}\label{sec:CharacterVarieties}

Given a finitely generated group $\Gamma$ and a compact connected Lie group $G$, 
the collection 
$\Hom(\Gamma, G)$ of group homomorphisms $\Gamma\to G$ admits the natural structure
of a real algebraic set as follows.  
First note that $G$ is a real algebraic group as are all compact Lie groups 
(a key step in proving this is the Peter-Weyl theorem which allows one to show $G$ has a faithful linear representation).  
Since $\Gamma$ is finitely generated, 
there exists a set of generators $\{\gamma_1,...,\gamma_r\}$, 
so that the natural mapping
\begin{align*}
\Hom(\Gamma, G)&\longrightarrow G^r \\
\rho&\longmapsto \big(\rho(\gamma_1),...,\rho(\gamma_r)\big) \end{align*}
is injective.
The image is cut out from the real variety $G^r$ by the relations defining $\Gamma$. 
By Hilbert's basis theorem this image is a real algebraic set.  
Furthermore the structure as an $\R$-algebraic set is independent of the generating
set $\{\gamma_1,...,\gamma_r\}$.

The group $G$ acts on $\Hom(\Gamma, G)$ 
by inner automorphisms and the resulting quotient 
\[
\XX(\Gamma, G):=\Hom(\Gamma, G)/\Inn(G) \] 
is  called the {\em $G$-character variety\/} of $\Gamma$.  
Since the quotient space by a compact Lie group on a real algebraic set is always a semialgebraic set 
we know that $\XX(\Gamma, G)$ is semialgebraic (see \cite{PS85, Sch1}).

In the case when $\Gamma=\pi_1(\surf)$ where $\surf$ is a genus $g\geq 2$ connected orientable closed surface, 
$\XX(\pi_1(\surf), G)$ has the structure of a projective variety via the correspondence to the moduli space of holomorphic principal $G_\C$-bundles over $\surf$ where $G_\C$ is the complexification of $G$. 
In this case, the number of components of both $\Hom(\pi_1(\surf), G)$ and 
$\XX\big(\pi_1(\surf),G\big)$ is in bijection with $\pi_1([G,G])$ where $[G,G]$ is the derived subgroup of $G$. See Ho-Liu~\cite{HoLiu2005} and \cite{HoLiu2005a,HoLiu2004} for a description
of components of character varieties and the earlier paper Li~\cite{JunLi} for the description in the complex semisimple case.

Denote by $[\rho]\in \mathfrak{X}(\Gamma,G)$ the $G$-conjugacy class of $\rho\in \Hom(\Gamma, G)$.
If $\rho$ is smooth in $\Hom(\Gamma, G)$, then the Zariski tangent space $\o{T}_{[\rho]}\mathfrak{X}(\Gamma, G)$ is isomorphic to $$\o{T}_0\left(\H^1(\Gamma,\mathfrak{g}_{\mathsf{Ad}_\rho})/\o{Stab}_G(\rho(\Gamma))\right).$$  This follows from Mostow's slice theorem \cite{Mostow57} and Weil's theorem \cite{Weil64} that the tangent space to $\rho$ in $\Hom(\Gamma, G)$ is $\o{Z}^1(\Gamma, \mathfrak{g}_{\mathsf{Ad}_\rho})$ while the tangent space to $\rho$ in the subspace $[\rho]\subset\Hom(\Gamma, G)$ is $\o{B}^1(\Gamma, \mathfrak{g}_{\mathsf{Ad}_\rho})$. 

We say a homomorphism $\Gamma\xrightarrow{~\rho~} G$ is {\em irreducible} if the centralizer of the image $\rho(\Gamma)$ is a finite extension of the center $Z(G)$ of $G$.  A homomorphism is {\it good} if $\o{Stab}_G\big(\rho(\Gamma)\big)=Z(G)$.  So at good representations $\o{T}_{[\rho]}\mathfrak{X}(\Gamma, G)\cong \H^1(\Gamma,\mathfrak{g}_{\mathsf{Ad}_\rho})$.

Irreducible homomorphisms are smooth points in $\Hom\big(\pi_1(\surf), G\big)$. 
Consequently, the subspace of irreducible homomorphisms $\XX^i(\pi_1(\surf), G)$ is an orbifold, 
and the subspace of good homomorphisms $\XX^g(\pi_1(\surf), G)$ is a smooth manifold.  
By definition all good homomorphisms are irreducible, but they are not generally equivalent.  
However, when $G=\SU(n)$ a homomorphism is good if and only if it is irreducible.  
Thus, $\XX^i\big(\pi_1(\surf), \SU(n))$ is a smooth manifold 
and corresponds to the smooth locus of a projective variety; 
namely, the moduli space of unimodular holomorphic vector bundles of rank $n$ over an algebraic curve homeomorphic to $\surf$.

From \cite{Gold2}, we conclude that $\XX^g(\pi_1(\surf), G)$ is a symplectic manifold when $g\geq 2$.

Here is how one defines the symplectic form $\omega$ at a good representation $\rho$:
$$
\xymatrix{
\H^1(\Sigma,\,\g_\Ad) \times \H^1(\Sigma,\,\g_\Ad) \ar[r]^-\cup &
\H^2(\Sigma,\,\g_\Ad\otimes \g_\Ad) \ar[d]^{\mathfrak{B}_*}\\
& \H^2(\Sigma,\,\C) \ar[d]^{\cap [Z]}\\
\o{T}_{[\rho]}\mathfrak{X}(\Gamma, G)\times \o{T}_{[\rho]}\mathfrak{X}(\Gamma, G)\ar[uu]_{\cong} \ar[r]^-\omega &
\H_0(\Sigma,\,\C)\,\cong\,\C,}
$$
where $\mathfrak{B}$ is a symmetric, non-degenerate bilinear form on $\mathfrak g$ that is invariant under the adjoint representation (if $G$ is semisimple $\mathfrak{B}$ is a multiple of the Killing form), and $[Z]$ is the fundamental class of $\Sigma$. 

\subsubsection{Finite Measure}\label{sec:FiniteMeasure}
By \cite{Karshon}, $\omega$ is algebraic and so continuously extends from the Zariski open set $\XX^g(\pi_1(\surf), G)$ to $\XX(\pi_1(\surf), G)$.  
(See Goldman~\cite{Gold2} for an explicit expression for it using the Fox free differential calculus.)
Letting $d=\dim \XX(\pi_1(\surf), G)$, we have $\omega^{d/2}$ is a {\it singular} volume form on a compact space and hence finite.  The singular locus of $\XX(\pi_1(\surf), G)$ has measure $0$ with respect to this form since the singular locus has positive codimension, and therefore $\omega^{d/2}$ is a finite measure on the open submanifold $\XX^g(\pi_1(\surf), G)$ in $\XX(\pi_1(\surf), G)$.

\subsection{The \Teich~flow} 

 In this section we describe the general background for suspending mapping class group dynamics
on character varieties to a cocycle over the \Teich~geodesic flow,
which is a non-Abelian version of the Kontsevich-Zorich-cocycle on a Veech surface.
This program was initiated in \cite{FG17}.

When $G$ is a connected compact semisimple Lie group,
the connected components $\RepGt$ of 
$\RepG$,
are indexed by $\tau\in\pi_1(G)$. See \S\ref{sec:CharacterVarieties} for details.

The \Teich~ sphere bundle $\U \Mg$ 
consists of pairs $(M,[\mu])$ where $M$ is a Riemann surface of genus $g$ and $[\mu]$ is the equivalence class
of a Beltrami differential $\mu$ with $\Vert\mu\Vert_{\infty} = 1$. 
Using \Teich's theorem, 
$\U \Mg$ identifies with the bundle of unit-area holomorphic quadratic differentials $Q$;
this identification is via $\mu = \overline{Q}/\vert Q\vert$.
The divisor of $Q$ is positive with degree $4g-4$, 
that is,
\[
\divisor[Q]  = \sum_{i=1}^l k_i \pointdivisor[p_i]  \]
where $p_1,\dots,p_l\in \surf$ are distinct, 
$k_i > 0$,
and 
\[
\sum_{i=1}^l k_i = 4g-4. \]
For each $k = (k_1,\dots,k_l)$ as above,
the pairs $(M,Q)$ where the divisor $\divisor[Q]$ has singularity type $k$,
define a subset $\U^k \Mg$ of $\U \Mg$. 
This {\em stratum\/} is an orbifold over $\Mg$,
and $\U \Mg$ is stratified by the $\U^k \Mg$.
Furthermore each stratum is $\MCG$-invariant,
and $\MCG$ permutes the components of each stratum.

The {\em top stratum\/}  $\U^\top\Mg$ of $\U\Mg$ 
corresponds to holomorphic quadratic differentials with the maximal number (necessarily $4g-4$) 
of simple zeroes.
Kontsevich-Zorich~\cite{KontsevichZorich} proved this stratum is connected.

\subsection{The Forni-Goldman cocycle}

The mapping class group $\Mod(\Sigma_g)$ acts on $\TTg$ with quotient $\RMS$.
Furthermore $\MCG$ acts on $\RepGt$. In the second paper in this series ~\cite{MR4799915} (compare also Brown~\cite{Brown98}), we find individual mapping classes which fail to act ergodically on the character variety although the full mapping class group acts ergodically~\cite{Gold6,GoldmanXia11}.

The quotient of 
\[
\U^\top\TTg \times \RepGt \]
 by the diagonal action of $\MCG$ is the total space $\EGS$ of a $\RepGt$-fibration
\[
\EGS \longrightarrow \TTg/\MCG = \RMS. \]
Let $\UE$ denote the pullback of this bundle over $\Mg$ under the fibration 
\[
\URk \longrightarrow \Mg. \]
Let $\Phi$ be the horizontal lift of the $\SLtwoR$-action on $\URk$.

V.\  Gadre and C.\  Matheus have recently observed that the main theorem of the first paper in this series \cite{FG17} is misstated (see the Introduction of this paper for the corrected version).

The version stated in \cite{FG17} involves 
an arbitrary connected component of an {\em arbitrary\/} stratum 
$\U^k\RMS$ of $\U\RMS$.
However the proof in \cite{FG17} relies on ergodicity of the 
{\em full\/} mapping class group which only applies if  the stratum is connected. 

Kontsevich-Zorich~\cite{KontsevichZorich} proved the 
{\em top stratum\/}  $\U^\top\Mg$ of $\U\Mg$ 
is connected. Furthermore it is $\MCG$-invariant. An interesting open question is to repair the proof in \cite{FG17} to cover the cases of strata which are disconnected.

\subsection{Hodge theory on a Riemann surface}\label{sec:HodgeTheoryRS}
When $M$ is a Riemann surface (or, more generally, a K\"ahler manifold),
then the Hodge decomposition on exterior forms \eqref{eq:HodgeDecomposition1}
passes to a decomposition
\begin{equation}\label{eq:HodgeDecomposition2}
\H^1(M;\C)  = \H^{(1,0)}(M)  \oplus \H^{(0,1)}(M)  \end{equation}
(as complex vector spaces).  This decomposition arises by representing de Rham cohomology
classes in  $\H^1(M;\C)$ as harmonic forms, 
and every harmonic form decomposes uniquely as a sum of a
holomorphic $1$-form and an anti-holomorphic $1$-form. 
Furthermore, the composition (denoted $h$, following \cite{FMZ14})
\begin{equation}\label{eq:RealMapping}
\H^1(M,\R) \hookrightarrow 
\H^1(M,\C) \twoheadrightarrow   \H^{(1,0)}(M)   \end{equation}
is an isomorphism of real vector spaces; 
compare \S\ref{sec:RealStructures}.
Specifically,
suppose that $c\in\H^1(M,\R)$ is a cohomology class.
Then the $(1,0)$-part of the unique harmonic representative of $c$
is a holomorphic $1$-form which is $h(c)$.
Furthermore 
\[
c = \o{Re}\big( h(c)\big). \] 

The image of the integral lattice 
$\H^1(M,\Z) < \H^1(M,\R)$ 
under \eqref{eq:RealMapping} is a lattice $\Lambda <  \H^0(M,\kappa)$ 
whose quotient is a complex $g$-torus which identifies with the {\em Jacobi variety\/}
of $M$ (consisting of isomorphism classes of topologically trivial holomorphic line bundles).
Furthermore the quotient 
\[
\H^1(M,\R)/\H^1(M,\Z) =   \H^1(M,\R/\Z) \] 
identifies with the character variety $\Hom(\pi,G)/G$ where $G = \Uone\cong\R/\Z$.
(Compare Goldman-Xia~\cite{GoldmanXiaHiggs}.)

\subsection{Variation of Hodge structures}\label{sec:VHS}
Let $\mathscr{B}$ be a quasiprojective variety over $\C$.
Suppose that $\big\{M_s\big\}_{s\in \mathscr{B}}$ is a family of compact Riemann surfaces over 
$\mathscr{B}$.

Consider a basepoint $o\in\mathscr{B}$ and the Riemann surface $M_o$ over $o$.
Let  $G$ is a compact Lie group and 
$\mathscr{B}$ with a base point $o$ and 
$\pi_1(M_o)\xrightarrow{~\rho~} G$ a representation defining
a smooth point of the character variety $\XX(\pi_1(M_o),G)$.
Suppose that $\rho$ is fixed by the natural action of 
\[ 
\pio(\mathscr{B},o)\to \Out\big(\pi_1(M_o)\big) \]
on $\XX(\pi_1(M_o), G)$. 

The first cohomology groups 
\[
\H^1_{\o{dR}}(M_s,\g_{\Ad_{\rho}}) \] 
carry real Hodge structures polarized by 
$i_{\nabla_{\rho}}$. 
These Hodge structures fit into a polarized real variation of Hodge structure 
of weight one. 
Deligne's semisimplicity theorem~\cite{MR900821,MR498551} implies 
the natural monodromy representation 
\[
\pio(\mathscr{B},o)\to \Aut\big(
\H^1_{\o{dR}}(M_o,\g_{\Ad \rho})
\big) \] 
is semisimple. 
(Compare also M\"oller \cite{Mo06} and Filip \cite{Fil16}).
The book~\cite{Carlson} gives a good introductory treatment of Hodge theory.
Hence this representation decomposes into strongly irreducible pieces (up to taking an adequate finite cover of $\mB$).

\section{Complex vector bundles}\label{sec:VectorBundles}
To make this paper accessible to several audiences,
and to establish consistent notation and terminology,
we provide some background in fiber bundles, Riemann surfaces and holomorphic vector bundles.
We begin by reviewing the theory of complex vector bundles over smooth manifolds.
In the next section, 
we describe holomorphic vector bundles over complex manifolds,
where we can speak of local sections being {\em holomorphic.\/}

\subsection{Vector bundles defined by \texorpdfstring{\Cech~}{}  cocycles}\label{sec:Cech}
A complex vector space can be understood as a real vector space with a 
a {\em complex structure,\/} that is, an endomorphism $\Jj$ with $\Jj^2 = -\bOne$. 
This extends scalar multiplication from $\R$ to $\C = \R[\sqrt{-1}]$.
A {\em complex vector bundle of rank $r$\/} over a smooth manifold $M$ is real vector bundle $\Vv$ with
a smoothly varying family of complex structures on the fibers, 
so each fiber has the structure of a complex vector space.
Such a complex vector bundle can be defined as follows.
Consider an  open covering $\UUU = \{ U_\alpha\}_{\alpha \in\IndexSet}$,
where $\IndexSet$ is a fixed index set.

Over each coordinate patch $U_\alpha$ the bundle $\Vv_\alpha\to U_\alpha$ is the
trivial vector bundle 
\begin{equation}\label{eq:LocalTrivial}
U_\alpha \times \C^r \longrightarrow U_\alpha\end{equation}
defined by projection on the first factor. 
These bundles are identified by a \Cech 1-cochain on $\UUU$,
that is, 
a system of maps
\begin{equation}\label{eq:CechCocycle}
U_\alpha \cap U_\beta \xrightarrow{~g_{\alpha\beta}~} \GL(r,\C),\end{equation}
where $\alpha,\beta\in \IndexSet$.
The identification of the trivial vector bundles  \eqref{eq:LocalTrivial} covers the identification of the open covering $\UUU$ to obtain $M$:
Identify $\UUU$ with the disjoint union
$\coprod_{\alpha\in\IndexSet} U_\alpha$ and consider the quotient map
$\UUU\longrightarrow M$ by the following equivalence relation.
For each $u\in U_\alpha \cap U_\beta$,
denote by $u_\alpha$ (respectively $u_\beta$) the elements
of $U_\alpha \subset \UUU$ 
(respectively of $U_\beta \subset \UUU$).
Then the equivalence relation on $\UUU$ is generated by $u_\alpha \sim u_\beta$ 
for each $u\in U_\alpha\cap U_\beta$.

This equivalence relation on $\UUU$ extends to an equivalence relation
on the trivial $\C^r$-bundle
\[ 
\coprod_{\alpha\in \mathfrak{A}} U_\alpha \times \C^r \]
over $\UUU$ as follows.
Given a system of maps $g_{\alpha\beta}$ (a $1$-cochain) as above,
extend this equivalence relation on $\UUU$ to $\UUU\times \C^r$
by
\[
(u_\alpha, v) \sim (u_\beta, g_{\beta\alpha}(u_\alpha)  v) \]
where $v\in\C^r$. 
This is an equivalence relation if the $1$-cochain $g = \{g_{\alpha\beta}\}_{\alpha,\beta\in\IndexSet}$ 
is a {\em cocycle,\/} 
that is, 
if it satisfies cocycle identities  
\begin{itemize}
\item $g_{\alpha\alpha}(u)  = 1$ for $u\in U_\alpha$;
\item $g_{\alpha\beta}(u) g_{\beta\gamma}(u) g_{\gamma\alpha}(u)= 1$ 
for $u\in U_\alpha\cap U_\beta\cap U_\gamma$.
\end{itemize}
A section of $\Vv$ is a \Cech 0-cochain  $s = \{s_\alpha \}_{\alpha\in\IndexSet}$ 
consisting of maps $U_\alpha \xrightarrow{~s_\alpha~} \C^r$ such that
\[
s_\alpha(u) = g_{\alpha\beta}(u) s_\beta(u) \]
for $u\in U_\alpha \cap U_\beta$,
that is, 
the \Cech coboundary of  $s = \{s_\alpha \}_{\alpha\in\IndexSet}$
equals the cocycle  $g = \{g_{\alpha\beta}\}_{\alpha,\beta\in\IndexSet}$.

The case when the coordinate changes 
$g_{\alpha\beta}$ are locally constant maps on $U_\alpha\cap U_\beta$
is very important:
this corresponds to the bundle being {\em flat,\/} and arises from
a flat connection.
In the terminology of Steenrod~\cite{Steenrod}, 
this is a {\em reduction of the structure group\/} of the bundle from $\GLrC$
to $\GLrC$ {\em given the discrete topology.\/}
It corresponds to a representation of the fundamental group into $\GLrC$.

If $\Vv$ is such a complex vector bundle, 
its {\em conjugate\/} vector bundle $\overline{\Vv}$ has the same total space but with complex structure $-\Jj$.

\subsection{Representations of the fundamental group}\label{sec:RepsFundGp}
Suppose that $\pi = \pi_1(\surf)$ is the fundamental group of $\surf$.
Choose a universal covering $\tS \to \surf$ with group of deck transformations $\pi$
acting on the right.

Let $\rho\in \Hom(\pi, G)$ where $G$ is a closed subgroup of $\GLrC$.
We assume that  $\rho(\pi)$ lies in the maximal
compact subgroup $\Ur< \GLrC$;
equivalently $\rho$ preserves the standard positive definite Hermitian metric on $\C^r$.

Let $\Vv_\rho\to\surf$ denote the corresponding {\em flat Hermitian rank $r$ complex vector bundle over $\surf$,\/}
defined as follows.
Then the representation $\rho$ defines a (left) action of $\pi$ on the trivial
rank $r$ bundle $\tS \times \C^r\to \tS$ over $\tS$:
\begin{align*} 
\tS \times \C^r &\xrightarrow{~\gamma~} \tS \times \C^r \\
(\ts, v) &\longmapsto \big(\ts\gamma\inv, \rho(\gamma)v\big)
\end{align*}
where $\gamma\in\pi$. 
This action covers the action of $\pi$ on $\tS$,
whose quotient identifies with $\surf$.
The quotient 
of the action on $\tS\times\C^r$ by $\pi$
is the total space of the vector bundle $\Vv_\rho\to\surf$
with the base $\surf = \tS/\pi$.

\subsection{Flat connections}\label{sec:FlatConnections}
The trivial connection on $\tS\times \C^r$
(where the submanifolds $\tS\times \{v\}$, 
for $v\in\C^r$, are horizontal) passes down to 
a flat connection $\nabla_\rho$ on $\Vv_\rho$.
The $\rho$-invariant Hermitian structure on $\C^r$ defines
a Hermitian structure on $\Vv_\rho$,
which is parallel with respect to $\nabla_\rho$.
For details, 
compare 
Proposition 4.21 of Kobayashi~\cite{Kobayashi},  p.14.

\subsection{Complex and real structures}\label{sec:RealStructures}

In many cases $\Vv$ arises as the complexification of a real vector bundle $\Uu$.
For a vector space $\V$ over $\C$, denote its underlying real vector space by $\V_\R$,
and its complex structure by $\Jj$. Then a {\em real structure\/} on $\V$ consists of a 
$\C$-antilinear isomorphism 
\[ \V_\R\xrightarrow{~\rr~}\V_\R \] 
(that is, $ \rr\circ \Jj = -\Jj\circ\rr$) such that $\rr\circ\rr = \bOne$.
Then the fixed set ($+1$-eigenspace) of $\rr$ is a
subspace $\UU$ of $\V_\R$ such that $\Jj(\UU)$ is the $-1$-eigenspace and 
\[
\V_\R = \UU \oplus \Jj(\UU), \]
that is,
\[ 
\V = \UU \otimes_{\R} \C \]
is the complexification.

In this case, $\V_\R$ inherits a complex structure from $\V_\R\otimes_\R\C$ as follows.
The linear automorphism $\Jj$ on $\V_\R$ complexifies to a $\C$-linear
automorphism $\Jj\otimes_\R\C$ of $\V_\R\otimes_\R\C$ whose square is $-\bOne$.
Therefore it has eigenvalues $\pm\sqrt{-1}$ and $\V_\R\otimes_\R\C$ decomposes
as a direct sum of complex subspaces
\[
\V_\R\otimes_\R\C = \V' \oplus \V''\]
where $\V'$ is the $+\sqrt{-1}$-eigenspace and
$\V''$ is the $-\sqrt{-1}$-eigenspace.
Furthermore the real structure $\rr$ interchanges the two summands.
Projections $\V_\R\otimes\C \twoheadrightarrow \V'$ 
and $\V_\R\otimes\C \twoheadrightarrow \V''$  are  given by:
\[
v \xmapsto{~\Pi'~}  
\frac{v + \sqrt{-1}\  \Jj(v)}2 \]
and
\[
v \xmapsto{~\Pi''~}  
\frac{v - \sqrt{-1}\  \Jj(v)}2, \]
respectively.
The composition of the inclusion $\V\hookrightarrow \V_\R \otimes_\R\C$
and the projections $\Pi'$ and $\Pi''$ define isomorphisms of real vector spaces
$\UU \to (\V')_\R$ and $\UU\to(\V'')_\R$ respectively. 

\subsubsection{The isomorphism from a real form}\label{sec:IsoFromRealS}
All of this linear algebra extends seamlessly to complex vector bundles $\Vv$.
The {\em underlying real vector bundle\/} $\Vv_\R$ admits a {\em complex structure\/}
$\Jj$ whose square is $-\mathbb{I}$.
The original complex vector bundle admits a {\em real structure\/} whose fixed subspace
is a subbundle $\Uu$ of $\Vv_\R$ such that $\Vv \cong \Uu\otimes_\R \C$ is its complexification, and there is a decomposition of bundles $
\Vv_\R\otimes_\R\C \cong \Vv' \oplus \Vv''$.

\subsubsection{The real structure on differential forms}\label{sec:RealForms}
If $\Vv$ is a complex vector bundle,
then a {\em real structure\/} on $\Vv$ is a family of real structures on the fibers
of $\Vv$ and there is a real subbundle $\Uu < \Vv_\R$ such that $\Vv$ is the 
complexification of $\Uu$.
The composition of the inclusion $\Vv \hookrightarrow \Vv_\R\otimes_\R\C $ with the projection 
$\Vv_\R\otimes_\R\C \twoheadrightarrow \Vv'$
defines an isomorphism of real vector bundles 
\begin{equation}\label{eq:IsoFromRealS}
\Uu \xrightarrow{} (\Vv')_\R. 
\end{equation}
We can extend the real structure to differential forms valued in $\Vv$ as follows.
Suppose the base is a complex manifold so that the de Rham algebra
$\Aaa^*(\surf)$ admits a real structure denoted by complex-conjugation $\omega\mapsto\bo$,
then $\Aaa^k(\surf,\Vv)$ admits a real structure such that:
\begin{equation}\label{eq:ExtendingRrToForms}
a \otimes \omega \xmapsto{~\Rr~}  \Rr a \otimes \bo \end{equation}
whenever $a\in\Aaa^0(\surf,\Vv)$ is a section and $\omega\in\Aaa^{k}(\surf)$ is a scalar-valued
$k$-form. 

\subsection{Hermitian structures}\label{sec:HermitianStructures}

Let $\V$ be a complex vector space,
with underlying real vector space $\V_\R$. 
A {\em Hermitian structure\/} on $\V$ is given by a $\C$-bilinear pairing
\begin{align*}
\V \times \overline{\V} &\longrightarrow\ \ \C \\
(a,b) &\longmapsto \langle a,b\rangle 
\end{align*}
such that 
$\langle b,a\rangle = \overline{\langle a,b\rangle}$ where $a,b \in \V_\R$.
It is {\em positive definite} if $\langle a,a\rangle > 0$ if $a\neq 0$. 

Let $\V\xrightarrow{~\Rr~}\V$ be a real structure,
that is, 
a $\C$-antilinear isomorphism of $\V_\R$ with $\Rr\circ\Rr = \bOne$. 
Equivalently, $\Rr$ defines an isomorphism $\V \to \overline{\V}$ of
complex vector spaces of ``order two."
For example if $\g$ is a Lie algebra over $\C$ then a real form $\mathfrak{h}$ of $\g$ determines
a real structure on $\g$ whose fixed set equals $\mathfrak{h}$.

Suppose that $\hh$ is a Hermitian form on $\V$ and $\Rr$ is a real structure which preserves $\hh$.
Then there exists a nondegenerate $\C$-bilinear form 
\[ 
\V \times \V \xrightarrow{~\bb~} \C \] such that 
\[
\hh(\upsilon,\nu)= \bb\big(\upsilon,\Rr(\nu)\big). \]

A smoothly varying family of Hermitian structures on the fibers of a complex
vector bundle $\Vv$ defines a Hermitian pairing of complex vector bundles
(also denoted $\hh$)
\[
\Vv \times \overline{\Vv} \xrightarrow{~\hh~} C^\infty (M, \C).\]
which takes values in smooth complex-valued functions on $\surf$.

\subsection{The adjoint bundle}
Let $G$ be a compact semisimple Lie group and $\g$ its Lie algebra.
Suppose that 
\[
\g \times \g \xrightarrow{~\bb~} \R \]
is a positive definite symmetric bilinear form invariant under the adjoint representation
\[
G \xrightarrow{~\Ad~} \Aut(\g). \]
Let $\gC$ denote the complexification of $\g$ and $\Rr$ the real structure
on $\gC$ given by complex conjugation $\C\to\C$ tensored with $\g$.
Its fixed subalgebra equals $\g <  \gC$. 

Let $\bb$ also denote the corresponding 
nondegenerate symmetric $\C$-bilinear form 
\[
\gC \times \gC \xrightarrow{~\bb~} \C. \]

Suppose $\pi \xrightarrow{~\rho~} G$ is a representation of $\pi = \pi_1(\surf)$. 
Let  $\g_{\Ad\rho}$ denote the flat vector bundle over $\surf$ with holonomy
\[
\pi \xrightarrow{~\Ad\,\circ\,\rho~} \Aut(\g) \]
and $\gC_{\Ad\rho}$ its complexification.
The bilinear form $\bb$  determines a bilinear pairing of vector bundles 
\[
\gC_{\Ad\rho} \times \gC_{\Ad\rho} \xrightarrow{~\bb_*~} \C \]
where $\C$ denotes the trivial complex line bundle over $\surf$.

For example, when $G=\SUtwo$, 
which comprises the unit quaternions,
its Lie algebra $\sutwo$ identifies with the pure quaternions.
We shall use this model in the sequel.

\subsection{Flat connections and de Rham theory}\label{sec:FlatConnectionsDeRham}

Just as a holomorphic structure on a complex vector bundle
extends the local notion of holomorphicity by the vanishing of the 
extension $D''$ of the $\dbar$ operator on functions,
a {\em connection\/} extends the local constant functions by
vanishing of an extension of operator $d$ on functions.
Specifically,
a connection is a differential operator $D$ on smooth sections which extends
the exterior operator $d$ on functions in the sense that
\begin{equation}\label{eq:connection}
D(f \sigma) = f D(\sigma) + \big(df \big) \sigma,
\end{equation}
where $\sigma$ is a local section and $f$ is a smooth function.
Extending to a mapping on $\Vv$-valued $k$-forms by enforcing
\eqref{eq:connection} to the case that $f$ is a (scalar-valued) exterior $k$-form, 
one  obtains a sequence of maps
\[
... \xrightarrow{~D~} \Aaa^k(\surf,\Vv) \xrightarrow{~D~} 
\Aaa^{k+1}(\surf,\Vv) \xrightarrow{~D~} \dots \]
which is a {\em cochain complex:\/}
\[
D\circ D = 0, \]
whenever $D$ is 
{\em flat,\/} 
that is, 
arises from a representation
$\pi_1(\surf) \longrightarrow \GLrC$.	
(Compare Kobayashi~\cite{Kobayashi}, Proposition 4.21.)
In that case the cohomology of this complex is the {\em de Rham cohomology with 
values in $\Vv$,\/} denoted $\H^k(\surf;\Vv)$.
Sometimes the flat vector bundle $\Vv$ is called a {\em local system.\/}
This cohomology identifies with the sheaf cohomology of $\surf$ with
values in the sheaf of germs of  {\em locally constant sections\/} of $\Vv$.

\subsection{Flat unitary line bundles}
When $G$ is Abelian,
then the set $\Hom(\pi,G)$ identifies with the cohomology $\H^1(\surf,G)$ ---
furthermore the natural action of $\Inn(G)$ on $\Hom(\pi,G)$ is trivial
so the quotient  $\Hom(\pi,G)/G$ also identifies with $\H^1(\surf,G)$.
In particular \big(using the isomorphism $\R/\Z\to\Uone$\big),
\begin{align}\label{eq:UoneCharVar}
\H^1(\surf,\R)/\H^1(\surf,\Z) &\cong  \H^1(\surf,\R/\Z)  \notag  \\
\cong \H^1(\surf,&\Uone)  \cong
\Hom(\pi,\Uone) \notag \\ 
&\quad  = \Hom(\pi,\Uone)/\Uone
\end{align}
(since $\Uone$ is Abelian). 
When $\surf$ is given a complex structure,
that is, 
it is {\em marked\/} by a homeomorphism $\surf\approx M$
to a Riemann surface $M$,
then \eqref{eq:UoneCharVar} arises as the {\em Jacobian\/} of $M$). 
A description of this in terms of connections on complex line bundles is given 
in Goldman-Xia~\cite{GoldmanXiaHiggs}.

\section{Holomorphic vector bundles}\label{sec:holvecbun}
This section specializes the preceding discussion to the case that the base is a Riemann surface $M$
(a complex manifold of dimension $1$) and the vector bundle $\Vv\to M$ is a 
{\em holomorphic vector bundle\/} over $M$. 
Just as a complex vector bundle over a smooth manifold
has a preferred class of {\em smooth\/} (local) sections,
a holomorphic bundle will have a preferred class of 
{\em holomorphic\/} local sections. 

In the case of a flat bundle,
holomorphicity plays an intermediate role,
between smooth local sections and {\em parallel\/} local sections.
Just as a connection generalizes the ordinary derivative of functions
whose vanishing defines {\em locally constant\/} functions,
a {\em holomorphic structure\/} on a smooth complex vector bundle $\Vv$
over a complex manifold is an extension of the Cauchy-Riemann operator
$\dbar$ to local sections of $\Vv$.

 We refer to Kobayashi~\cite{Kobayashi} as a basic reference.

\subsection{Holomorphicity on a complex manifold}
We expound some basic complex differential geometry to introduce some of the
ideas we will need in the special case of {\em Riemann surfaces,\/} 
that is, 
complex manifolds of dimension $1$.
In order to motivate \Teich~ theory, 
we briefly describe infinitesimal deformations in terms of the theory
developed by Kodaira and Spencer~\cite{KodairaSpencer} 
before specializing to the case of Riemann surfaces.
\subsubsection{Classical deformation theory}\label{sec:ClassicalDeformationTheory}
If $M$ is a complex manifold,
its {\em holomorphic tangent bundle\/} $\Theta$ 
is a holomorphic vector bundle of rank equal to the dimension of $M$.
Its holomorphic sections \big( elements of $\H^0(M,\Theta)$ \big)
are {\em holomorphic vector fields\/} on $M$.
Such holomorphic vector fields generate local holomorphic $\C$-actions.

Kodaira and Spencer~\cite{KodairaSpencer} 
proved that the {\em infinitesimal deformations\/}
of the complex structure on $M$ correspond to elements
in the cohomology group $\H^1(M;\Theta)$.
The rough idea is as follows.
Consider an open covering of $M$ with coordinate patches $U_\alpha$
and coordinate charts 
\[
U_\alpha\xrightarrow{~\psi_\alpha~} \C^n \]
which are holomorphic embeddings onto open subsets
$\psi_\alpha(U_\alpha)$ of $\C^n$.
The complex manifold $M$ is obtained as an identification space
of the disjoint union 
\[
\Uu_\IndexSet := \coprod_{\alpha\in\IndexSet} U_\alpha,\]
similar to the description of a holomorphic vector bundle in
\S\ref{sec:VectorBundles}. 
The equivalence relation on $\Uu_\IndexSet$ is generated by identifications
of points $u_\alpha \in U_\alpha \subset \Uu_\IndexSet$ with the corresponding
$u_\beta \in U_\beta \in \Uu_\IndexSet$ defined by $u\in U_\alpha \cap U_\beta$. 

Suppose $M_t$ is a differentiable family of complex manifolds,
with fixed topological type $\surf$.
Assume the holomorphic structure on each coordinate patch $U_\alpha$ is
{\em not\/} varying. 
Then the identifications $u_\alpha \sim u_\beta$ corresponding to a
point $u\in U_\alpha \cap U_\beta$ are given by biholomorphisms
of $U_\alpha\cap U_\beta$ varying with $t$.
Differentiating with respect to $t$ one obtains a holomorphic vector
field $\theta_{\alpha\beta}$ on $U_\alpha\cap U_\beta$.
This family ---
taken over $(\alpha,\beta)\in \IndexSet\times\IndexSet$ ---
defines a \Cech $1$-cocycle with values in the sheaf $\Theta$,
obtaining a cohomology class in $\H^1(M,\Theta)$.

\subsubsection{de Rham theory on a complex manifold}

According to the complexified action of the almost complex structure $\Jj$ on the real tangent space $\T(M_\R)$,  
the de Rham algebra $\Aaa^*(M;\C)$ of $\C$-valued exterior differential forms on $M$
enjoys a {\em Hodge decomposition\/}
\begin{equation}\label{eq:HodgeDecomposition1}
\Aaa^k(M,\C) = \bigoplus_{p+q = k} \Aaa^{(p,q)}(M). \end{equation}
Denote the $(p,q)$-Hodge projection by
\[
\Aaa^k(M,\C)  \xrightarrow{~\Pi^{(p,q)}~}  \Aaa^{(p,q)}(M) \]
where $p+q = k$. 
In particular the $(0,1)$-Hodge projection equals:
\begin{align*}
\Aaa^1(M,\C)  &\xrightarrow{~\Pi^{(0,1)}~}  \Aaa^{(0,1)}(M) \\
\omega &\longmapsto \frac{\omega + i(\omega\circ \Jj)}2 
\end{align*}

The Cauchy-Riemann equations characterizing holomorphicity can 
be written as
\[
\dbar f = 0 \]
where $\dbar$ equals the composition
\[
C^\infty(M,\C) = \Aaa^0(M,\C)  \xrightarrow{~d~}  \Aaa^1(M,\C) 
\xrightarrow{~\Pi^{(0,1)}~}  \Aaa^{(0,1)}(M). \]

\subsection{Holomorphic structures}\label{sec:HolomorphicStructures}
Now suppose that $\Vv$ is a complex vector bundle as above,
and $M$ is a complex manifold.
Then a {\em holomorphic structure\/} on $\Vv$ is a refinement of the definition
given in \S\ref{sec:Cech},  
where the \Cech cocycle defining the  identifications are given by {\em holomorphic maps\/} $U_\alpha \cap U_\beta \longrightarrow \GL(r,\C)$.
From this definition one can define a (local) section to be holomorphic if it 
arises as the graph of a local holomorphic map $U_\alpha \longrightarrow \C^n$. 

Equivalently, a holomorphic structure on $\Vv$ is given by a differential operator $D''$ on smooth sections which extends
the Cauchy-Riemann operator $\overline{\partial}$ on functions in the sense that
\begin{equation}\label{eq:HolomorphicStructure}
D''(f \sigma) = f D''(\sigma) + \big(\overline{\partial} f \big) \sigma,\end{equation}
where $\sigma$ is a local section and $f$ is a smooth function.
One naturally extends $D''$ to a mapping (also denoted $D''$)
\[
\Aaa^k(M;\Vv) \longrightarrow \Aaa^{k+1}(M;\Vv) \]
to enforce \eqref{eq:HolomorphicStructure} when $f\in\Aaa^k(M)$ is a scalar-valued
exterior $k$-form. 
We require {\em integrability\/} in the sense that $D''\circ D''=0$,
a condition holding automatically when  $M$ has complex dimension $1$.

Then a smooth local section $\sigma$ is holomorphic if and only if $D''(\sigma) = 0$. 
A {\em holomorphic vector bundle\/} over a complex manifold $M$ is a smooth complex vector bundle with a holomorphic structure as above.
Every holomorphic structure on $\Vv$ induces a holomorphic structure on $\overline{\Vv}$. 

In particular if $D$  is a {\em flat\/} connection, 
then $D'' := \Pi''\circ D$
is a holomorphic structure.

Locally constant mappings on a complex manifold $M$ are holomorphic with respect to 
{\em any\/} choice of complex structure on $M$. 
This important observation implies that if $\Vv$ is a flat bundle over $\surf$ and
$\surf \approx M$ when $M$ is a complex manifold,
then $\Vv$ has a natural holomorphic structure.
In the sequel we shall consider a fixed flat bundle over $\Sigma$ and
investigate how the resulting holomorphic bundles vary as the complex manifolds
$M$ homeomorphic to $\surf$ vary.

\section{Riemann surfaces and \Teich~theory}
This section contains preliminary background material on the deformation theory of
Riemann surfaces.
In particular we develop the theory of \Teich~ space,
which for us will be a complex manifold with a Finsler metric,
the {\em \Teich~ metric.\/}
This metric is the {\em Kobayashi metric,\/} 
the intrinsic metric determined by the complex structure.

\subsection{Marked Riemann surfaces}
We denote $\surf$ a closed oriented connected topological surface of genus $g$.
A {\em marking\/} of a Riemann surface $M$ is a  homeomorphism 
\[
\surf\xrightarrow{~\Upsilon~} M.\]
We call the pair $(M,\Upsilon)$  a {\em marked Riemann surface of genus $g$.\/} 
Two marked Riemann surfaces $(M,\Upsilon)$ and $(M',\Upsilon')$ are {\em equivalent\/} if and only if 
there exists a biholomorphism $M \xrightarrow{~\phi~} M'$ such that $\Upsilon' \simeq \phi\circ\Upsilon$. 
The {\em \Teich~   space\/} $\TT(\surf)$ 
consists of all equivalence classes of marked Riemann surfaces
homeomorphic to $\surf$.

To emphasize the genus, we may denote $\surf$ by $\surf_g$ and $\TT_g$ for $\TT(\surf)$.

$\TT(\surf)$ enjoys a natural complex structure,
due to Ahlfors and Weil,
see, for example  \cite{Hubbard,IT,Nag}.

For our purposes,
working in the category of smooth manifolds is convenient.
For this reason,
we suppose that $\surf$ is a {\em smooth\/} surface and
the markings $\Upsilon_t$ are {\em quasi-diffeomorphisms\/}
in the sense of 
Arbarello-Cornalba-Griffiths~\cite{GeoAlgCurves}),
that is,
homeomorphisms which are diffeomorphisms outside a finite set of points.
The extremal quasiconformal homeomorphisms 
guaranteed by \Teich's theorem all have this property,
thus simplifying some of the analytic technicalities.

\subsection{The canonical line bundle and its square}\label{sec:HodgeIntersection}
If $M$ is a Riemann surface, 
then its cotangent bundle is a holomorphic line bundle (its rank is one)
denoted $\kappa = \kappa_M$ and called its {\em canonical bundle.\/}
Its holomorphic sections are {\em holomorphic $1$-forms,\/} 
or {\em Abelian differentials,\/}
comprising the space $\H^0(M;\kappa)$.
With respect to a local holomorphic coordinate $z$ an Abelian differential
may be written
\[
\alpha = a(z) dz, \]
where $a(z)$ is holomorphic.

Elements of  $\H^0(M,\kappa)$ are holomorphic $1$-forms (Abelian differentials).
These are harmonic $1$-forms of Hodge type $(1,0)$,
and thus there is an isomorphism $\H^0(M,\kappa)\cong \Hodge$. 

Holomorphic sections of the dual bundle $\kappa\inv$ are holomorphic vector fields, 
locally of the form
\[
\upsilon = u(z) \ddz \]
where $u(z)$ is holomorphic.

Holomorphic sections of the {\em square\/} 
\[
\kappa^2  = \kappa \otimes \kappa \]
of $\kappa$ are {\em holomorphic quadratic differentials,\/}
and correspond to {\em tangent covectors\/} to \Teich~space.
Thus the space $\H^0(M;\kappa^2)$ of holomorphic quadratic
differentials on $M$ identifies with the {\em holomorphic cotangent
space\/} to $\TT$ at the point corresponding to $M$ (with respect to some
marking).

\subsubsection{Holomorphic \texorpdfstring{$1$}{}-forms and cohomology}
The {\em intersection form\/}  on the de Rham cohomology group 
$\H^1(M) := \H^1(M,\C)$ is given by:
\begin{equation}\label{eq:IntersectionProduct}
(\alpha,\beta) \longmapsto \frac{i}2 \int_M   \alpha \wedge \beta, \end{equation}
pairing closed $1$-forms $\alpha,\beta$ representing cohomology classes in $\H^1(M;\C)$.
(The imaginary coefficient is there so that the image of $\H^1(M,\R)\times\H^1(M,\R)$ equals $\R$.)
This agrees with the composition pairing 
\[
\H^1(M) \times \H^1(M) \longrightarrow 
\H^2(M) \xrightarrow{\approx} \C \]
where the first pairing is the cup product and the second map is evaluation on the
fundamental homology class.
Poincar\'e duality implies that this skew-symmetric bilinear pairing is nondegenerate.
The holomorphic $1$-forms comprise the complex vector space $\H^0(M;\kappa)$.
Holomorphic $1$-forms are closed,  
two holomorphic $1$-forms are cohomologous only if they are equal,
so $\H^0(M;\kappa)$ embeds in $\H^1(M)$ as a complex linear subspace. 

The {\em Hodge inner product\/} is a {\em positive definite\/} pairing on the Hodge space
$\H^0(M,\kappa)$, which we identify with the Hodge space $\H^{(1,0)}(M)$. 
It arises from a sesquilinear modification of  \eqref{eq:IntersectionProduct} as follows.
Simply apply complex conjugation (the {\em real structure\/}) $\C\to\C$ to the second
argument:
\begin{equation}\label{eq:HodgeProduct}
\langle \alpha,\beta\rangle := 
\frac{i}2 \int_M   \alpha \wedge \overline\beta, \end{equation}
The sesquilinear form \eqref{eq:HodgeProduct} is a nondegenerate pseudo-Hermitian form
on $\H^1(M)$, 
whose restriction to $\Hodge$ is positive definite,
as we shall presently describe.

Suppose $\alpha,\beta\in\Hodge$ are Abelian differentials. 
We write the Hodge product in terms of a local holomorphic coordinate $z = x + i y$ where $x,y\in\R$. 
Then 
\[
\alpha = a(z) dz, \qquad  \beta = b(z) dz,\]
where $a(z), b(z)$ are holomorphic functions.
Then 
\[
\betab  = \overline{b(z)}\, d\zbar \]
and we decompose the integrand in  \eqref{eq:HodgeProduct} as follows.
Write 
\begin{equation}\label{eq:Alternation}
\Alt ( u \otimes v) := u \otimes v - v \otimes u \end{equation}
for the usual alternation, 
so that  the exterior product $u\wedge v := \Alt ( u \otimes v)$.
The Euclidean area form in local coordinates is:
\begin{equation}\label{eq:EuclideanAreaForm}
dx \wedge dy = \frac{i}2 \Alt (dz \otimes d\zbar) \end{equation}
and 
\begin{align*}
\frac{i}2 \alpha \wedge \betab &\  = \ \frac{i}2 \Alt \big( a(z) \overline{b(z)}\ 
dz \otimes d\zbar\big) \\
&\quad  = \quad  a(z) \overline{b(z)} \, dx \wedge dy.  \end{align*}
It follows that natural pseudo-Hermitian structure on $\Hodge$  given by \eqref{eq:HodgeProduct} is positive definite (that is, {\em Hermitian\/}) when restricted to $\Hodge$.

By applying complex conjugation,
the restriction of the Hodge intersection form to $\H^{(0,1)}(M)$ is negative definite.
Furthermore if $\alpha\in\Hodge$ and $\beta\in\H^{(0,1)}(M)$,
then $\alpha\wedge\betab$ is a $(2,0)$-form and therefore vanishes.
Thus the splitting \eqref{eq:HodgeDecomposition1} is an orthogonal direct sum
into a positive definite summand and a negative definite summand.

We systematically identify the {\em Hodge space\/} $\Hodge$ with the space 
$\H^0(M;\kappa)$ 
of Abelian differentials (global holomorphic sections of the canonical line bundle 
$\kappa$ of $M$).
If $M$ is a complex manifold and $\Vv\to M$ is a holomorphic vector bundle over $M$,
then we identify $\Vv$ with the sheaf of holomorphic sections of $\Vv$.
In particular $\H^k(M,\Vv)$ denotes the \Cech cohomology of $M$
with coefficients in the sheaf of holomorphic sections of $\Vv$.

\subsubsection{Bilinear pairing of Abelian differentials}
In a different direction the (symmetric tensor) product (denoted $\alpha \beta$) of Abelian differentials $\alpha,\beta$ is a holomorphic quadratic differential, 
that is, a holomorphic section of the holomorphic line bundle $\kappa^2 = \kappa \otimes \kappa$.
A holomorphic quadratic differential $Q \in \H^0(M,\kappa^2)$ is given in local coordinates by:
\[ 
Q = q(z) dz^2 \]
where $q(z)$ is holomorphic. 
If $\alpha,\beta \in \Hodge$ as above, 
the coefficient $q(z)$ in their product $Q = \alpha\otimes \beta$ equals
\[
q(z) = a(z) b(z)\]
in local coordinates.
Holomorphic quadratic differentials on a Riemann surface $M$ form the (holomorphic) {\em cotangent space\/}
$\T^*\TS$ to \Teich~space at the point of $\TS$ corresponding to the $M$, 
once a marking $\surf\to M$ is chosen.

When $M$ is a Riemann surface 
that is, $\dim(M) =1$,
then $\Theta = \kappa\inv$
when $\kappa$ is the canonical line bundle of $M$.
Thus the Kodaira-Spencer space of infinitesimal deformations
equals $\H^1(M;\kappa\inv)$.
Serre duality (see, for example, Hubbard~\cite{Hubbard}, \S A9, pp.407--412)
implies this vector space is dual to the space $\H^0(M,\kappa^2)$ 
of holomorphic quadratic differentials by the pairing
\[
\H^0(M;\kappa^2) \times \H^1(M;\kappainv) \longrightarrow 
\H^1(M;\kappa) \cong \C. \]

\subsubsection{Beltrami differentials}

Dolbeault's theorem implies that
cohomology classes in  $\H^1(M,\kappainv)$ can be
represented by $(0,1)$-forms on $M$ taking values in the
holomorphic tangent bundle $\kappainv$ of $M$.
Such an essentially bounded section of the line bundle 
$\kappainv\otimes\kappab$
is by definition a a {\em Beltrami differentials\/} on $M$,
and correspond to {\em infinitesimal quasiconformal mappings.\/}
(Compare Hubbard~\cite{Hubbard}, \S 6.6, Nag~\cite{Nag} or Imayoshi-Taniguchi~\cite{IT}.) 

Two Beltrami differentials are equivalent if they pair identically with holomorphic quadratic differentials
in $\H^0(M,\kappa^2)$ under the natural pairing \eqref{eq:pairing}, 
or equivalently, they induce isotopic quasiconformal homeomorphisms of $M$.

Beltrami differentials and quadratic differentials pair by:
\begin{equation*}
(Q,\mu) \longmapsto \frac{i}2  \int_M \Alt(Q\otimes \mu)  \end{equation*}
where we interpret the integrand as follows.
The tensor field $Q\otimes \mu$ is an $L^\infty$ section of the line bundle 
\[
\kappa^2 \otimes \big(
\kappainv\otimes\kappab
\big) \cong
\kappa\otimes\kappab. \]
In local coordinates,
\[ 
\mu = m(z)\,    
\ddz\otimes d\zbar 
\]
Using \eqref{eq:EuclideanAreaForm},
\begin{align*}
Q\otimes \mu & =  q(z) m(z) \quad
dz^2 \otimes\ddz\otimes d\zbar \\
& =  q(z) m(z)\quad  dz\ \otimes \ d\zbar  \end{align*}
and 
\[
\frac{i}2 \Alt(Q\otimes \mu) =   q(z) m(z)\  dx \wedge  dy\]
where the usual {\em alternation\/} $\Alt$ is defined in \eqref{eq:Alternation}.
The pairing 
\begin{align}\label{eq:pairing} 
\H^0(M,\kappa^2) 
\times 
L^\infty\big(
\kappab \otimes \kappainv\big) 
& \longrightarrow \C  \\
(Q,\mu) \qquad  &\longmapsto \frac{i}2 \int_M \Alt(Q  \otimes \mu) 
\notag
\end{align}
 establishes a duality between the complex vector space $\H^0(M,\kappa^2)$ 
(which has dimension $3g-3$ by the Riemann-Roch theorem) and
the quotient of $L^\infty\big(\kappa\inv \otimes\overline{\kappa}\big)$ by
the annihilator of $\H^0(M,\kappa^2)$ under the pairing \eqref{eq:pairing}. 
This annihilator identifies with the subspace of {\em trivial Beltrami differentials\/}
\[
\mu_f = \bar{\partial} f \otimes (\partial f)\inv = 
f_{\zbar}/f_z\  \ddz\otimes d\zbar ,\]
taken in the distributional sense,
where $f$ is a quasiconformal homeomorphism.

Thus the tangent covectors to $\TS$ are holomorphic quadratic differentials,
and the tangent vectors to $\TS$ are equivalence classes of essentially
bounded Beltrami differentials, dually paired by \eqref{eq:pairing}.

Two Beltrami differentials are equivalent if they pair identically with holomorphic quadratic differentials
in $\H^0(M,\kappa^2)$ under the natural pairing \eqref{eq:pairing}, 
or equivalently, they induce isotopic quasiconformal homeomorphisms of $M$.
Nag~\cite{Nag}, \S 3.1,  
refers to this basic fact as {\em \Teich's Lemma.\/} 
(See also \S 5.2-5.3 of Imayoshi-Taniguchi~\cite{IT} or 
\S 6.6 of Hubbard~\cite{Hubbard} for details.)

\subsection{Extremal quasiconformal mappings and \Teich~ discs}\label{sec:TeichDiscs}
We recall the \Teich~ deformation.
Suppose that 
\[ \Sigma \xrightarrow{~\Upsilon_t~} M_t. \]
is a path of marked Riemann surfaces $M_t$.
By Teichm\"uller's theorem, 
we may assume that $\Upsilon_t$ is a {\em quasi-diffeomorphism,\/}
that is a homeomorphism which is a diffeomorphism except at a finite set of points.

Furthermore, we may assume that these homeomorphisms define quasiconformal
maps as follows.
The composition
\[
g_t := \Upsilon_t \circ (\Upsilon_0)\inv \]
is a path of quasi-diffeomorphisms
\[
M_0 \xrightarrow{~g_t~} M_t, \]
which are  {\em admissible\/} in the sense of \cite{GeoAlgCurves}.
This just means that the Beltrami derivatives 
\[
\mu_t := \dbar g_t \otimes \left(\partial g_t\right)\inv \]
are essentially bounded,
that is,
$\mu_t \in \Beltrami.$

Although not technically necessary,
this assumption simplifies the discussion
since we may assume the Beltrami differentials are smooth away from a finite subset
(``admissible'' in the terminology of  \cite{GeoAlgCurves}, p.468). 
\Teich's theorem guarantees a one-parameter group of extremal quasiconformal map 
(necessarily a homeomorphism) $M_0\to M_t$.
This {\em \Teich~ mapping\/} is affine with respect to the singular Euclidean structures
defined by holomorphic quadratic differentials on $M_0$ and $M_t$ respectively. 
The zeroes of the initial quadratic differential on $M_0$ form a finite set $S$,
and $g_t$ is an affine transformation with respect to the Euclidean structure on the complement
$M_0\setminus S$ defined by $Q_0$.
In particular the homeomorphisms $g_t$ restrict to diffeomorphisms on $M_0 \setminus S$.

In local coordinates $z$ on $M_0$ and $w$ on $M_t$
such that $w = g_t(z)$,
 the differential $\Dsf g_t$ pulls back the holomorphic $1$-form $dw$ on $M_t$ to
 a multiple of $ dz + m(z) d\zbar$,
 where $m(z) \ddz \otimes d\zbar$ is the Beltrami differential $\mu(z)$. 

If $\mu$ is a Beltrami differential with $\Vert\mu\Vert_\infty \le 1$,
the we shall consider a {\em Teichm\"uller disc,\/} 
defined by quasiconformal homeomorphisms $M_0\xrightarrow{~g_\zeta~} M_t$
where 
\[
\zeta
\in \DD := \{\zeta\in\C\mid\vert\zeta\vert<1\}\]
and $g_\zeta$ has Beltrami derivative $\mu_{g_\zeta} = \zeta \mu$.

Such \Teich~ discs are holomorphic curves in $\TS$.
Give $\DD$ the Poincar\'e metric 
$\o{g} := \vert d\zeta\vert^2/\left(1 - \vert \zeta\vert^2\right)^2$
having curvature $-4$. 
Then for each direction parametrized by $\theta\in\R$,
the path
\begin{align*}
\R &\longrightarrow \quad\TS \\
t  &\longmapsto  \left[\left(M_{r e^{i\theta}},\Upsilon_{te^{i\theta}}\right)\right] \end{align*}
is a \Teich~ geodesic with velocity vector $e^{i\theta}[\mu]$ at $t=0$,
where $r = \tanh(t)$ and $\zeta = re^{i\theta}$.

When $Q$ is a holomorphic quadratic differential,
then \Teich's theorem (compare \S 8.4 of Bers~\cite{Bers}, p.104) 
implies that, 
for $\zeta\in\DD$,
the  Beltrami differentials $\zeta\mu$, 
where
\[
\mu = \frac{\overline{Q}}{\vert Q\vert}, \]
achieves the infinitesimal deformation corresponding to the 
\Teich~curve corresponding to the infinitesimal complex line $\C\mu$
in the tangent space to $\TS$. 
Here, 
if in local coordinates $Q(z) = q(z) \,dz^2$,
then $\vert Q\vert$ denotes the $(1,1)$-form given in local coordinates by:
\[
\vert Q(z) \vert  := \vert q(z) \vert \, \vert dz\vert^2   = 
\vert q(z) \vert \, dz\wedge d\zbar \]
and the Beltrami differential is:
\[
\mu (z) = \frac{\overline{Q(z)}}{\vert Q(sz) \vert}  = \quad \frac{\overline{q(z)}}
{\vert q(z)\vert} \ \ 
 \ddz\otimes d\zbar
\]
Compare also \cite{IT,Nag,Hubbard} and
\S 15 of Arbarello-Cornalba-Griffiths~\cite{GeoAlgCurves},
where these Beltrami differentials are called {\em admissible.}

The admissible Beltrami differentials arising from \Teich's theorem satisfy 
\[
\Vert\mu\Vert_\infty = \left\Vert\frac{\overline{Q}}{\Vert Q\Vert} \right\Vert_\infty  = 1, \]
so multiplying $\mu$ by $\zeta\in\DD$
gives  \Teich~ mappings. 
The corresponding {\em \Teich~ disc\/}
\begin{align}\label{eq:TeichDisc}
\DD &\xrightarrow{~\Upsilon^\mu~} \quad\TS \notag \\
\zeta  &\longmapsto \left[\left(M_\zeta,\Upsilon_\zeta\right)\right] \end{align}
is  a {\em complex geodesic,\/} 
that is,
a totally geodesic isometric embedding  of $(\DD,\o{g})$
to a holomorphic curve in $\TS$.

\subsection{Orientable quadratic differentials}\label{sec:Orientable}
In the case that $Q$ is orientable, 
it is the {\em square\/} $\omega^2$ of an Abelian differential $\omega$,
then 
\begin{equation} \label{eq:OrientableQuadDiff}
\mu  = \frac{\overline{Q}}{\vert Q\vert}  
 =\frac{\bo^2}{\omega\, \bo} = \bo/\omega \in L^\infty(\kappab\otimes \kappainv)
\end{equation}
defines an admissible Beltrami differential  whose equivalence class generates a \Teich~ disc.

In this case, 
the \Teich~ flow has a simple description.
Writing $\omega_t$ for the Abelian differential $(g_t)_*\omega_0$ on $M_t$,
it relates back to $\surf$ using the marking $\Upsilon_t$.
\begin{equation}\label{eq:TeichFlowAbelianDiff}
\Upsilon_t^* \omega_t = 
\cosh(t)\, \Upsilon_0^* (\omega_0) + 
\sinh(t)\, \Upsilon_0^* (\bo_0) 
\end{equation}

\section{Rauch's formula}
In this section we extend the variational formulas used by Forni~\cite{Forni02} to flat bundles.
The main result is Theorem~\ref{thm:TwistedRauch} and its proof follows the standard proofs. 
This result may be regarded as a nonabelian extension of the classic work of Rauch~\cite{Rauch1960}
(based on \cite{Rauch1955a,Rauch1955b};
compare also Earle~\cite{Earle})
on the differential of the period map from \Teich~ space to the Siegel moduli space of principally polarized abelian varieties.

For simplicity we first reduce to the case that $\mu$ corresponds to an orientable quadratic
differential.
We first describe how to extend the machinery to  coefficients in a flat unitary vector bundle $\Vv$. 
This involves a parallel Hermitian structure $\langle,\rangle$, 
determined by a 
parallel bilinear form $\bb$ and a parallel real structure $\Rr$
as discussed in \S\ref{sec:RealStructures}.
The triple 
$\big(\langle,\rangle,\bb,\Rr\big)$ 
is compatible in the sense of \S\ref{sec:HermitianStructures}.

\subsection{The first variation formula with twisted coefficients}

We consider a \Teich\\ disc $\DD\to\TS$ determined by 
an equivalence class $[\mu]$ of Beltrami differentials
$\mu\in\Beltrami$, as in \S\ref{sec:TeichDiscs}.

Suppose that $\Uu$ is a flat unitary vector bundle over $\surf$ and $\Vv$ its complexification.
Consider cohomology classes in $\H^1(\surf,\Uu)$ represented by
$D$-closed $\Uu$-valued $1$-forms $\alpha,\beta$.
For each $t\in\DD$, there exists unique
$\Vv$-valued holomorphic $(1,0)$-forms 
$\alpha_t,\beta_t$ on $M_t$ such that
the corresponding  cohomology classes 
\[
[\Up(\alpha_t+\Rr\alpha_t)] = [\alpha],\qquad
[\Up(\beta_t+\Rr\beta_t)] = [\beta] \]
in $\H^1(\surf,\Uu)$ 
remain constant in $t$
(compare  \S\ref{sec:IsoFromRealS} and \S\ref{sec:RealForms}).
\begin{theorem}\label{thm:TwistedRauch}
Then for all $t\in\DD$,
\begin{equation}\label{eq:TwistedRauch}
\frac{d}{dt} \langle \alpha_t,\beta_t\rangle =
2 \Re~ \B^\mu (\alpha_t,\beta_t)  \end{equation}
where $\langle,\rangle$ is the Hodge inner product on $\H^{1,0}(M_t,\Vv)$
defined by the parallel Hermitian product on $\Vv$  and
$\B^\mu$ denotes the second fundamental form with respect to the infinitesimal
deformation $[\mu]$ and the bilinear form $\bb$ on $\Vv$.
\end{theorem}
\noindent
The Hodge inner product is defined in \S\ref{sec:ExtensionToFlat}
and the second fundamental form is defined in \S\ref{sec:SecondFF}.
The proof is given in \S\ref{sec:FirstVariation}.

\subsubsection{Extension to flat vector bundles}\label{sec:ExtensionToFlat}
Suppose $\Vv\to\surf$ is a flat vector bundle with Hermitian structure $\langle,\rangle$
defined by a nondegenerate symmetric bilinear form $\bb$ and a real structure $\Rr$ on 
$\Vv$ as in \eqref{sec:RepsFundGp}.
Then  the pointwise Hermitian pairing on local sections $a,b$ of $\Vv$ is the function (scalar field) given by
\[
(a,b) \longmapsto   \bb_*(a\otimes \Rr(b)) \]
Generalizing the Hodge product \eqref{eq:HodgeProduct} for scalar-valued holomorphic $1$-forms,
the scalar-valued pseudo-Hermitian form on $\Vv$-valued $1$-forms is given by:
\begin{equation}\label{eq:VvaluedHodgeProduct}
\langle \alpha,\beta \rangle :=
\frac{i}2 \int_M 
\Alt \Big(\bb_*\big(\alpha\otimes\Rr\beta\big)\Big)
\end{equation}
where $\Rr\beta$ is described in \eqref{eq:ExtendingRrToForms}.
By the same computations as in \S\ref{sec:HodgeIntersection},
the restriction of this pseudo-Hermitian form to $\H^1(M_t,\Vv_t)$
is a (positive definite) Hermitian form on $\H^{(1,0)}(M_t,\Vv_t)$.

Now let $\Upsilon^\mu$ be a \Teich~ disc as in \eqref{eq:TeichDisc} above.
For each $t\in\DD$,
there is a flat holomorphic Hermitian vector bundle $\Vv_t$ over $M_t$ such
that pullback $\Upsilon_t^*\Vv_t = \Vv$.
Each of these  bundles carries a parallel Hermitian structure/real structure/bilinear form,
which are mutually compatible as in \S\ref{sec:RealStructures}.
For brevity we denote the flat connection by $\nabla$ and
the  triple of pseudo-Hermitian structure/real structure/bilinear form by
$\big(\langle,\rangle, \Rr,\bb\big)$.

\subsubsection{The second fundamental form with respect to a Beltrami differential}
\label{sec:SecondFF}
Now we define the symmetric bilinear form  $\B^\mu$ on the Hodge space
$\H^{(1,0)}(M;\Vv)$ associated to an infinitesimal deformation $[\mu]$ of the 
complex structure on $M$.
This is a {\em second fundamental form\/} obtained from the difference
of two natural connections on the Hodge bundle.
Just as for a general submanifold of a Riemannian manifold,
the classical second fundamental form depends on a normal vector,
the Beltrami differential $\mu$ roughly plays the role of the normal vector to the submanifold.
For simplicity  we may assume that $\mu$ is an admissible Beltrami differential
corresponding to a \Teich~ disc $\Upsilon^\mu$,
as in \S\ref{sec:TeichDiscs}.

Suppose that $\alpha,\beta\in \H^{(1,0)}(M,\Vv)$ 
are  $\Vv$-valued holomorphic $1$-forms on $M$.
Using $\bb$ as a coefficient pairing we obtain a holomorphic quadratic differential
which we denote 
\[
\bb_* (\alpha\otimes\beta)  \in \H^0(M,\kappa^2). \]
It pairs with a Beltrami differential 
\[
\mu\in \Beltrami \]
to obtain an essentially bounded $(1,1)$-form
\[
\Alt\big( \bb_* (\alpha\otimes\beta) \otimes \mu \big) \]
which can be integrated over $M$ to obtain a scalar.
In summary,
given 
an infinitesimal deformation $[\mu]\in\H^1(M,\kappainv)$
(where $\mu\in\Beltrami$ as in \S\ref{sec:ClassicalDeformationTheory}),
the {\em second fundamental form with respect to $[\mu]$\/}
is the bilinear form $\B^\mu$:
\begin{align}\label{eq:SFFmu}
\H^{(1,0)} (M,\Vv) 
\times
\H^{(1,0)} (M,\Vv) 
&\xrightarrow{~\B^\mu~} \qquad \C \\
(\alpha,\beta)\qquad\qquad &\longmapsto 
\frac{i}2 \int_M \Alt \big(\bb_*(\alpha\otimes\beta)  \otimes \mu\big) \notag
\end{align}
In the special case that $\mu=\bo/\omega$ for an Abelian differential $\omega$,
this formula simplifies because $\omega^2 \otimes \mu = \omega\otimes\bo$.
Namely,
if $\alpha = a \omega$ and $\beta = b \omega$
for sections $a,b\in \Aaa^0(M;\Vv)$,
then  
\begin{equation}\label{eq:SecondFFIntegrandOrientable}
\bb_*(\alpha\otimes\beta)  \otimes \mu = 
\bb_*(a\otimes b)\,  \omega^2 \otimes \mu = 
\bb_*(a\otimes b)\,  \omega\otimes\bo \end{equation}

When $\Vv = \C$, 
the trivial coefficient system,
this leads to the {\em first variation theorem for the Hodge bundle,\/}
originally proved by Forni~\cite{ Forni02, ForniHandbook,Intro,FMZ14}.  
It corresponds to the first derivative of the Hodge intersection form on $\H^{(1,0)}(M_t)$
as the Riemann surface $M_t$ varies with respect to the infinitesimal deformation $[\mu]$.
This may be considered an analog of the formula for the first variation of the Riemann period matrix, 
going back to H.\ Rauch~\cite{Rauch1955a} and L.\ Ahlfors~\cite{Ahlfors}.

\subsection{Proof of the first variation formula}\label{sec:FirstVariation}

We reduce the proof of Theorem~\ref{thm:TwistedRauch} 
to a calculation in the vector spaces
$\Aaa^1(\surf)$ and  $\Aaa^1(\surf,\Vv)$ using the markings $\Upsilon_t$.

We first deal with the case that $\mu$ is an admissible Beltrami differential
whose corresponding holomorphic quadratic differential $Q$ 
is {\em orientable\/} as in \S\ref{sec:Orientable}.
Then,
as in \eqref{eq:OrientableQuadDiff},  
an Abelian differential $\omega$ exists such that $\mu=\bo/\omega$
and $Q = \omega^2$. 
Furthermore \eqref{eq:TeichFlowAbelianDiff} implies that:
\begin{equation}\label{eq:TeichFlowAbelianDiff2}
\ddt
\ \Upsilon_t^*\omega_t
= 
\Upsilon_t^*\ \bo_t  \end{equation}
Since $\omega_t\in\H^{1,0}(M_t)$ is closed,
the corresponding $1$-form $\Up\omega_t$ on $\surf$ is closed.

Furthermore,
when $\mu=\bo/\omega$, 
by \eqref{eq:SecondFFIntegrandOrientable},
the second fundamental form
$\B^\mu$ admits the following expression,
when $\alpha = a \omega$ and $\beta = b\omega$:
\begin{equation}\label{eq:SecondFFOrientable}
\B^{\bo/\omega} (\alpha,\beta) = 
\frac{i}2 \int_M 
\bb_*(a\otimes b)\,  \omega\wedge\bo  \end{equation}

\subsubsection{Transporting twisted differentials over \Teich~ discs}

Let $\Upsilon^\mu$ be a \Teich~ disc as above.
Let 
\[
\alpha_0,\beta_0\in \H^{(1,0)}(M_0,\Vv). \]
Then $\alpha_t := (g_t^*)\inv\alpha_0$ and 
$\beta_t := (g_t ^*)\inv\beta_0$ have the property that the respective cohomology classes
\[
\Up[\alpha_t+ \Rr\alpha_t],\ 
\Up [\beta_t+ \Rr\beta_t]  \ 
\in \H^1(\surf,\Uu) \]
are constant. 
In that case the derivative
\[
 \ddt 
\left[\Up\left(\alpha_t + \Rr\alpha_t\right)\right]  
=
\left[
\frac{d\Up\big(\alpha_t + \Rr\alpha_t\big)}{dt}
\right] = 0  \]
vanishes in $\H^1(\surf,\Vv)$. 
Thus the $\Vv$-valued $1$-form
\[
\frac{d\Up\big(\alpha_t + \Rr\alpha_t\big)}{dt} \]
is $D$-exact,
that is,
sections $s_t$ of $\Vv\to\surf$ exist such that
\begin{equation}\label{eq:DtExact}
D s_t = \frac{d\Up\big(\alpha_t + \Rr\alpha_t\big)}{dt} =
\frac{d\Up\big(\alpha_t\big)}{dt} +
\frac{d\Up\big(\Rr\alpha_t\big)}{dt}
. \end{equation}

The pointwise Hermitian product is the $(1,1)$-form
\[
\Alt\big(\bb_*(\alpha_t\otimes\Rr\beta_t)\big)  = 
\bb_*(\alpha_t\wedge\Rr\beta_t)
\]
whose integral 
\begin{equation}\label{eq:HermitianProdVvaluedAbDiffs}
\frac{i}2 \int_{M_t} \bb_*(\alpha_t\wedge\Rr\beta_t
) =
\frac{i}2 \int_{\surf} \bb_*\big(\Up(\alpha_t)\wedge\Up(\Rr\beta_t)\big)   
 \end{equation}
equals the Hermitian product $\langle\alpha_t,\beta_t\rangle$,
the Hodge inner product on $\H^{(1,0)}(M_t,\Vv)$.

Differentiate the integrand in \eqref{eq:HermitianProdVvaluedAbDiffs}:
\begin{align}\label{eq:ExpandProduct}
\ddt
\bb_*\big(\Up(\alpha_t)\wedge\Up(\Rr\beta_t)\big)
&  =
\bb_*\left(\ddt\Up(\alpha_t)\wedge\Up(\Rr\beta_t)\right)\notag
\\ & \qquad +
\bb_*\left(\Up(\alpha_t)\wedge\ddt\Up( \Rr\beta_t)\right).
\end{align}
Evaluate and integrate each of the two summands in \eqref{eq:ExpandProduct}:
\begin{proposition}
\begin{equation}\label{eq:DerivFirstPart}
\frac{i}2\int_\surf 
\bb_*\left(\ddt\Up(\alpha_t)\wedge\Up(\Rr\beta_t)\right) = 
\overline{\B^\mu} (\alpha_t,\beta_t) \end{equation}
and 
\begin{equation}\label{eq:DerivSecondPart}
\frac{i}2\int_\surf 
\bb_*\left(\Up(\alpha_t)\wedge\ddt\Up( \Rr\beta_t)\right) = 
\B^\mu (\alpha_t,\beta_t) \end{equation}
\end{proposition}

\noindent
We prove \eqref{eq:DerivFirstPart} and observe that the proof of 
\eqref{eq:DerivSecondPart} is analogous.

\noindent To this end, 
break up the integrand in \eqref{eq:DerivFirstPart} using \eqref{eq:DtExact}:
\begin{align*}
\bb_*\left(\ddt\Up(\alpha_t)\wedge\Up(\Rr\beta_t)\right) & = 
\bb_*\Big(D s_t \wedge\Up(\Rr\beta_t)\Big)  \\ &\qquad  -
\bb_*\left(\ddt\Up(\Rr\alpha_t)\wedge\Up(\Rr\beta_t)\right)  
\end{align*}
and observe that the first summand is $d$-exact:
\[
\bb_*\big(D s_t \wedge\Up(\Rr\beta_t)\big)   = 
d \bb_*\big(s_t \wedge\Up(\Rr\beta_t)\big)  
\]
since $\Up(\Rr\beta_t)$ is $D_t$-closed.
Thus upon integration it contributes nothing and:
\[
\frac{i}2\int_\surf 
\bb_*\left(\ddt\Up(\alpha_t)\wedge\Up(\Rr\beta_t)\right) = 
- \frac{i}2\int_\surf 
\bb_*\left(\ddt\Up(\Rr\alpha_t)\wedge\Up(\Rr\beta_t)\right) \] 
Since $\Rr$ preserves $\bb$ 
\big(in the sense that $\bb(\Rr a, \Rr b) = \overline{\bb}(a,b)$\big),
rewrite the integrand of the right-hand side:
\begin{align}\label{eq:Integrand}
\bb_*\left(\ddt\Up(\Rr\alpha_t)\wedge\Up(\Rr\beta_t)\right) 
& = 
\bb_*\left(\Rr\ddt\Up(\alpha_t)\wedge\Rr\Up(\beta_t)\right) \notag \\
& \qquad\quad = 
\overline{
\bb_*}\left(\ddt\Up(\alpha_t)\wedge\Up(\beta_t)\right) 
\end{align}
Now rewrite $\alpha_t = a_t \omega_t$ for a section $a_t\in\Aaa^0(M_t,\Vv)$ and 
differentiate:
\begin{align*}
\ddt\Up(\alpha_t) & = \ddt\Up(a_t) \Up\omega_t +  \Up(a_t) \ddt\Up\omega_t  \\
& = \ddt\Up(a_t) \Up\omega_t +  \Up(a_t) \Up\bo_t  
\end{align*}
by \eqref{eq:TeichFlowAbelianDiff2}.
Thus the integrand \eqref{eq:Integrand} decomposes into two summands,
the first vanishing,
since it has Hodge type $(2,0)$ and $M$ is a Riemann surface:
\[
\bb_* 
\left( \ddt\Up(a_t) \Up\omega_t \wedge \Up \beta_t\right) = 0 \]
The second summand can be evaluated by writing $\beta_t = b_t \omega_t$
where $b_t$ is a section:
\[
\bb_*
\left( a_t \otimes  b_t\right)\, 
\bo_t \wedge \omega_t  = 
- \bb_*(\alpha_t\otimes\beta_t)\otimes\mu \]
by
\eqref{eq:SecondFFIntegrandOrientable} and
by the definition \eqref{eq:SFFmu}, 
of the second fundamental form,
\eqref{eq:DerivFirstPart} follows.
The proof of  \eqref{eq:DerivSecondPart} is analogous.

\subsubsection{The nonorientable case}
When the quadratic differential $Q$ associated to $\mu$ is nonorientable,
we find a branched double covering space 
\[
\hat M\xrightarrow{~\xi~} M \] 
over which $Q$ lifts to a quadratic differential $\hat Q := \xi^*Q$ which is orientable.
(Compare the beginning of Douady-Hubbard~\cite{DouadyHubbard},
where this is called the {\em Riemann surface of the quadratic differential $Q$.}
They construct $\xi$ as the submanifold of the cotangent bundle $\T^*M$ 
consisting of tangent covectors $\omega_p\in\T^*_p(M)$ such that
$(\omega_p)^2 = Q_p$. )
Let $\hat\omega$ be the Abelian differential on $\hat M$ so that $\hat Q = \hat\omega^2$.

The corresponding admissible Beltrami differential $\hat\mu$ pulls back by $\xi$
and defines a \Teich~ disc in $\widehat{M}$.

The flat unitary vector bundle $\Vv$ pulls back to $\hat M$,
along with its parallel Hermitian package. 
So does the family $\alpha_t,\beta_t$,
and Theorem~\ref{thm:TwistedRauch} for $M$ follows immediately from
Theorem~\ref{thm:TwistedRauch} for $\hat M$.

\section{Special \texorpdfstring{$\SUtwo$}{}-characters}
\label{sec-2}

The representations of $\pi_1(\surf)$ into $\SUtwo$ of interest in this paper have very
special properties enabling us to compute explicitly the action of the Veech group on
the corresponding points of the $\SUtwo$-character varieties. 
This preliminary section describes  these properties in terms of the algebra of Hamilton quaternions.

\subsection{Quaternionic representation}

The 	group $\Qu$ of unit quaternions acts on the quaternion algebra $\mathbb{H}$ by
the left-action by right-multiplication:
\begin{align*}
\mathbb{H} &\longrightarrow \mathbb{H}  \\
g &\longmapsto ( h \mapsto h g\inv ) \end{align*}
This action commutes with the action of $\bj$ by left-multiplication, 
preserving the $\C$-linear isomorphism $\H \xrightarrow{~\cong~} \C + \bj \C$,
and defines an embedding $\Qu \hookrightarrow \SUtwo$.

The quaternion group $\Qu$ is generated by the  {\em Pauli matrices:\/}
\begin{align*}
\Qu &\hookrightarrow \SUtwo \\
\bi &\longmapsto \bmatrix \sqrt{-1} & 0 \\ 0 & -\sqrt{-1} \endbmatrix \\
\bj &\longmapsto \bmatrix 0  &\sqrt{-1}  \\  \sqrt{-1} & 0 \endbmatrix \\
\bk &\longmapsto \bmatrix 0  &-1   \\  1 & 0 \endbmatrix. \end{align*}
which extends to an orthonormal basis for the inner product on the Lie algebra $\sutwo$:
\begin{align*}
\sutwo \times \sutwo &\xrightarrow{~\mathbb{B}~} \R \\
(x,y) &\longmapsto -\frac12 \tr(xy) \end{align*}(which is $-1/8$ of the Killing form).
For example, if $x = \left[ \begin{smallmatrix} i r & -z \\ \overline{z} & -ir \end{smallmatrix} \right], $
then 
\[
\mathbb{B}(x,x) = r^2 + \vert z\vert^2. \]

\begin{proposition}\label{prop:DiagonalHolonomy}
Let $\surf$ be a closed oriented surface of genus $g>1$ and $\pi=\pi_1(\surf)$ its
fundamental group, 
contained in a group $\widehat\pi$.
Suppose $\widehat\pi \xrightarrow{~\hrho~} \SUtwo$ is a homomorphism which factors through an
epimorphism $\widehat\pi \twoheadrightarrow \Qu$
and let $\rho$ be the restriction of $\hrho$ to $\pi$.   
Then:
\begin{enumerate}
\item $\rho$ is a smooth point in $\Hom(\pi,\SUtwo)$.
\item $\Inn\left(\SUtwo\right)$ acts freely on $\rho$.
\item 
Let $\R\bi, \R\bj, \R\bk < \sutwo$ be the lines spanned by the basic quaternions
$\bi,\bj,\bk$ respectively. 
Then $\rho$ is conjugate to a representation such that the composition $\Ad\,\rho:= \Ad\circ\rho$ 
\[
\pi \xrightarrow{~\rho~} \SUtwo \xrightarrow{~\Ad~} \Aut\left(\sutwo\right) \]
preserves the orthogonal decomposition
\[
\sutwo = \R\bi \oplus \R\bj \oplus \R\bk.\]
That is, 
the $\pi$-module $\sutwo_{\Adrho}$ 
decomposes as a direct sum of rank one $\pi$-modules
\begin{equation}\label{eq:DecomposePiModule}
\sutwo_{\Adrho} = (\R\bi)_{\Adrho} \oplus (\R\bj)_{\Adrho} 
\oplus (\R\bk)_{\Adrho}.    \end{equation}
\item 
The Zariski tangent space to the $\SUtwo$-character variety at $\rho$ decomposes as an orthogonal direct sum
\begin{align}\label{eq:splitEMcohomology}
\H^1(\pi,\suAr) & = \H^1\big(\pi,(\R\bi)_{\Adrho}\big) \\  
& \qquad\quad \oplus \H^1\big(\pi,(\R\bj)_{\Adrho}\big) 
\oplus \H^1\big(\pi,(\R\bk)_{\Adrho}\big). \notag \end{align}
\end{enumerate}
\end{proposition}
\begin{proof}
Since $\Qu < \SUtwo$ acts irreducibly on $\C^2$,
the representation $\rho$ is a smooth point of $\Hom(\pi,\SUtwo)$.
Furthermore the group $\Inn\left(\SUtwo\right)$ acts freely on $\rho$; 
compare Goldman~\cite{Gold2} and Sikora~\cite{Sikora}.
The center of $\Qu$ equals $\{\pm 1\}$ and the quotient is the Klein four-group isomorphic
to $\Z/2 \oplus \Z/2$, 
and restriction of $\Ad$ to $\Qu$ equals:
\begin{align*} 
\Qu &\xrightarrow{~\Adrho~} \Aut\left(\sutwo\right) \\
\bi &\longmapsto \bmatrix 1 &0 & 0 \\ 0 & -1 & 0 \\0 & 0 & -1 \endbmatrix \\
\bj &\longmapsto \bmatrix -1 & 0& 0 \\ 0 & 1 & 0 \\0 & 0 & -1 \endbmatrix \\
\bk &\longmapsto \bmatrix -1 & 0 & 0 \\ 0 & -1 & 0 \\0 & 0 & 1 \endbmatrix \end{align*}
with respect to the basis $\{\bi,\bj,\bk\}$ of $\sutwo$.
Thus the basic quaternions form a basis of eigenvectors,
establishing \eqref{eq:DecomposePiModule}.

Applying cohomology to \eqref{eq:DecomposePiModule} yields \eqref{eq:splitEMcohomology}.
\end{proof}

The $\pi$-module $\sutwo_{\rho}$ determines a rank three flat $\R$-vector bundle (local system) over $\surf$.
We denote the total space by $\mathfrak{E}_\rho$ and the flat connection by $\nabla$.
The decomposition \eqref{eq:DecomposePiModule}
defines a decomposition of the flat rank $3$ vector bundle 
\[ 
\mathfrak{E}_\rho = \Ll_\bi \oplus \Ll_\bj \oplus \Ll_\bk,   \]
where,
for each $\bu = \bi,\bj,\bk$,
the $\R$-bundle $\Ll_\bu$ is a parallel flat subbundle bundle of $\mathfrak{E}_\rho$ of rank $1$.
Since $\Ll_\bu$ is parallel,
it is necessarily flat with respect to the induced connection.

The Eilenberg-MacLane cohomology $\H^1(\pi,\suAr)$ is isomorphic to the de Rham cohomology with coefficients
in the flat vector bundle $\mathfrak{E}_\rho$ by the {\em de Rham theorem with local
coefficients;\/} 
compare Raghunathan~\cite[\S VII.1--2]{Raghunathan}. 
Similarly each summand 
$\H^1\big(\pi,(\R\bu)_{\Adrho}\big)$ is isomorphic to the de Rham cohomology 
$\H^1(\surf,\Ll_\bu)$,
and \eqref{eq:splitEMcohomology} leads to an orthogonal direct sum decomposition

\begin{equation}\label{eq:splitdRcohomology}
\H^1(\surf,\mathfrak{E}_\rho) = \H^1(\surf,\Ll_\bi) \oplus \H^1(\surf,\Ll_\bj) \oplus \H^1(\surf,\Ll_\bk). \end{equation}
Since each summand $(\R\bu)_{\Adrho} < \sutwo_{\rho}$ is an irreducible $\pi$-module
(where $\bu\in \{\bi,\bj,\bk\}$ is a basic quaternion), 
\[
\dim  \H^1\big(\pi, (\R\bu)_{\Adrho} \big)= 2g-2. \]
(Compare~\cite{Gold2}.)
By the isomorphism of de Rham cohomology of $\surf$ with coefficients in a local system with Eilenberg-MacLane cohomology of $\pi$,
each summand $(\R\bu)_{\Adrho}$ of $\sutwo_{\rho}$ 
determines a summand $\H^1(\surf,\Ll_\bu)$. 
The dimension of each of these three summands equals $2g-2$.

The corresponding holomorphic line bundles correspond to elements of order two in the Jacobian $\o{Jac}(X)$. 
Given a holomorphic line bundle $\Ll$ such that $\Ll^2$ is the trivial bundle,
the corresponding flat unitary bundle has real monodromy,
and corresponds to a homomorphism $\pi\to \{\pm 1\}$.
Compare Goldman-Xia~\cite{GoldmanXiaHiggs}.

Suppose $\Ll$ is a $2$-torsion line bundle as above.
Choose a meromorphic section $\sigma$ of $\Ll$.
Then $\sigma^2$ is a meromorphic section of the trivial bundle,
that is,
the graph of a meromorphic function $\widehat{\sigma}$ on $X$. 
At a zero or pole $z_0$  of $\widehat{\sigma}$, 
choose a local holomorphic coordinate $z$ so that $\widehat{\sigma} = f(z)$, 
and $f(z) =z^n$,
where $n=\o{ord}_{z_0}$.
Then there exist (exactly two) sections of $\Ll$ whose square equals $\sigma$,
related by multiplication by $\pm 1$.
At $z_0$,
\[
\sqrt{\sigma} = \pm  z^{n/2}. \]

The monodromy $\pi \xrightarrow{~\rho_\Ll~} \{\pm 1\}$
determines a double covering space
$\hat{\surf}\to \surf$ as follows. 
The monodromy of $\Ll$ is a homomorphism 
\[
\pi_1(\surf) \xrightarrow{~h_\Ll~} \{\pm 1\}, \]
whose kernel is an index two group.
Take $\hat{\surf}\to \surf$ to be the corresponding covering space,
that is, 
the total space $\hat{\surf}$ is the quotient of a universal covering space $\widetilde{\surf}$
of $\surf$ by $\o{Ker}(h_\Ll)$. 
Then the total space of $\Ll$ identifies with the fiber product $\widetilde{\surf}\times_{h_\Ll} \R$.

\subsubsection{Real line bundles and double coverings}
Suppose that $(X,\Upsilon)$ is a marked Riemann surface representing
a point in the Teichm\"uller space $\TS$. 
That is, 
$X$ is a Riemann surface and the {\em marking\/} 
\[
\surf \xrightarrow{~\Upsilon~} X \]
is a quasi-diffeomorphism    
defined up to isotopy. 
In that case every {\em flat\/}  complex vector bundle over $X$ has the underlying
structure of a {\em holomorphic vector bundle,\/} since locally constant mappings
are always holomorphic. 
Therefore we work over $X$ and complexify the bundles $\mathfrak{E}_\rho, \Ll_\bu$ obtaining
flat holomorphic vector bundles over $X$,
which we denote by $\mathfrak{E}^\C_\rho, \Ll^\C_\bu$ over $X$. The connection $\nabla$ defining the flat structure 
is sometimes called the {\em Gauss-Manin connection.\/}
The inner product complexifies to a Hermitian structure on these vector bundles.
As above there is an orthogonal direct sum decomposition of flat holomorphic vector bundles over $X$:
\begin{equation}\label{eq:IntoLineBundles}
\mathfrak{E}^\C_\rho = \Ll^\C_\bi \oplus \Ll^\C_\bj \oplus \Ll^\C_\bk   \end{equation}
and a corresponding orthogonal decomposition of (complex) de Rham cohomology
complexifying \eqref{eq:splitdRcohomology}.

The first de Rham cohomology of $X$ with values in any of these flat holomorphic vector bundles
$\Vv$ enjoys a {\em Hodge structure of weight one,\/} 
that is a splitting of complex vector spaces 
\[
\H^1(X,\Vv) = 
\H^{1,0}(X,\Vv) \oplus 
\H^{0,1}(X,\Vv)  \]
where
\[
\H^{0,1}(X,\Vv) = \overline{\H^{1,0}(X,\Vv)}. \]
This follows from our discussion in Section \ref{sec:holvecbun}.

\subsection{Flat \texorpdfstring{$\R$}{}-line bundles and double covering spaces}
Suppose that $\Ll\to\surf$ is a flat $\R$-line bundle over $\surf$,
with  structure group $\o{O}(1) = \{\pm 1\}$.
Then $\Ll$ is induced by a double covering space $\hat{\surf}\to \surf$ as follows. 
The monodromy of $\Ll$ is a homomorphism 
\[
\pi_1(\surf) \xrightarrow{~h_\Ll~} \{\pm 1\}, \]
whose kernel is an index two group.
Take $\hat{\surf}\to \surf$ to be the corresponding covering space,
that is, 
the total space $\hat{\surf}$ is the quotient of a universal covering space $\widetilde{\surf}$
of $\surf$ by $\o{Ker}(h_\Ll)$. 
Then the total space of $\Ll$ identifies with the fiber product $\widetilde{\surf}\times_{h_\Ll} \R$.
The cohomology $\H^*(\surf,\Ll)$ identifies with the cohomology of the complex 
of $\R$-valued cochains on $\widehat{\surf}$ which are equivariant respecting the action of the
deck involution given by $h_\Ll$. 
In de Rham cohomology, 
one can represent cohomology classes by closed equivariant differential forms on $\widehat{\surf}$. 

Suppose that $(X,\Upsilon)$ is a marked Riemann surface as above.
Then the marking $\Upsilon$ pushes forward the double covering $\widehat{\surf}\to \surf$
to a holomorphic double covering $\widetilde{X}\to X$.
Furthermore the Hodge structure on $\H^1(X,\Ll^\C)$ induces one on $\H^1(\widehat{X},\C)$, 
and the original space $\H^{1,0}(X,\Ll^\C)$ appears as the set of $\H^{1,0}(\widehat{X}),\C)$
which is equivariant with respect to the deck involution.

\section{The Eierlegende Wollmichsau}\label{sec:EW}

In this section we prove the main theorem 
(Theorem~\ref{main_thm:EW}) 
for the genus three surface $\EW$.
In Appendix~\ref{a.Fermat} this surface is shown to be the Fermat quartic curve.
For background on this translation surface,
see Forni \cite{ForniHandbook}, Herrlich-Schmith\"usen~\cite{HS}, Matheus-Yoccoz \cite{MY}, Matheus-Yoccoz-Zmiaikou \cite{MYZ}, Forni-Matheus~\cite{Intro}, and Forni-Matheus-Zorich \cite{FMZ14}.

This surface lies in a family,
based on the {\em Legendre family\/} of elliptic curves $$y^2 = x(x-1)(x-\lambda),$$
parametrized by $\lambda \in \Pdd :=  \C \setminus \{0,1\}.$
The translation surface $\EW(\lambda)$ is a double covering the elliptic curve $\Cp_2(\lambda)$
parametrized by $\lambda$. 
The original symmetric square-tiled surface corresponds to the parameter value $\lambda=-1$.

\subsection{Projective geometry of the translation surface \texorpdfstring{$\EW$}{}}

The affine plane curve 
$\Ca_4(\lambda)$ defined by
\begin{equation}\label{eq:AffineEW}
y^4 = x(x-1)(x-\lambda) \end{equation}
has a single ideal point $\qI$
and the corresponding projective curve $\Cp_4$ has genus $3$.
The Abelian differential $y^{-2}dx$ extends to an Abelian differential $\Phi_4(\lambda)$ on 
$\Cp_4(\lambda)$. 
The {\em Eierlegende Wollmichsau\/}
is the corresponding translation surface 
\[
\EW(\lambda) := \big(\Cp_4(\lambda),\Phi_4(\lambda)\big). \] 
The complex projective plane $\P := \P^2(\C)$ 
comprises projective equivalence classes 
\[ 
p := [\vv] = \begin{bbmatrix} X \\ Y \\ Z \end{bbmatrix}. \] 
of nonzero vectors
\begin{equation} 
\vv := \bmatrix X \\ Y \\ Z \endbmatrix \in \C^3,
\label{eq:HomogCoordinates2} \end{equation}
and $X,Y,Z$ are the {\em homogeneous coordinates\/} of the point $p$.

The affine patch 
$\APatch[3]$ defined by $Z\neq 0$ has an affine chart
\begin{align*}\C^2 &\xrightarrow{~\AChart[3]~} \APatch[3] \\
(x,y) & \xmapsto{\phantom{~\AChart[3]~}}  
\begin{bbmatrix}  x \\ y  \\ 1  \end{bbmatrix} \end{align*}
in which the curve 
$\Ca_4(\lambda)$ together with the ideal point 
\[
\qI := \begin{bbmatrix}  1 \\ 0  \\ 0  \end{bbmatrix} \]
defines the projective curve
\[
\Cp_4(\lambda) := 
\left\{ \begin{bbmatrix} X \\ Y \\ Z \end{bbmatrix} \ \Bigg|\ 
Y^4 - XZ (X-Z)(X-\lambda Z) = 0 \right\}.\]
This curve has an Abelian differential
\[
 \Phi_4(\lambda) := (Y/Z)^{-2} d\, X/Z  = Y^{-2} \ZdX \]
with $(\AChart[3])^*\Phi_4(\lambda) = y^{-2}dx$ and
divisor 
\[ 
\ldiv \Phi_4(\lambda) \rdiv = \pointdivisor[p_0]+ 
\pointdivisor[p_1]+ 
\pointdivisor[p_\lambda]+ \pointdivisor[\qI] \]
where 
\[
p_0 :=  \AChart[3](0) = \bbmatrix 0 \\ 0 \\ 1\endbbmatrix, \quad 
p_1 := \AChart[3](1) = \bbmatrix 0 \\ 1 \\ 1\endbbmatrix, \quad 
p_\lambda := \AChart[3](\lambda) =  \bbmatrix 0 \\ \lambda \\ 1\endbbmatrix \]
are the points of $\Ca_4$ corresponding to 
$0,1,\lambda $ respectively.
We normalize the  Abelian differential so that its area is $1$:
\[
\omega(\lambda) := c(\lambda)\,  \Phi_4(\lambda)  \]
where $c(\lambda)>0$ is defined so that 
\begin{equation}\label{eq:areaOne}
\frac{i}2 \int_{\Ca_4(\lambda)} \omega(\lambda)\wedge \overline{\omega(\lambda)} = 1. \end{equation}

The affine coordinate $x = X/Z$ is a local holomorphic coordinate on $\Ca_4(\lambda)$ away from
$p_0,p_1,p_\lambda$.
The affine coordinate $y = X/Z$ is a local holomorphic coordinate at $p_0,p_1,p_\lambda\in\Ca_4$.
The affine coordinate $\ty = Y/X$ in the the affine chart 
\begin{align*}
\C^2 &\xrightarrow{~\AChart[1]~} \APatch[1] \\
(\ty ,z) & \xmapsto{\phantom{~\AChart[3]~}} 
\begin{bbmatrix}  1 \\ \ty   \\z \end{bbmatrix}
\end{align*}
on the patch $\APatch[1]$ defined by $X\neq 0$ is a local holomorphic coordinate at $\qI$.

\subsection{The fundamental group and orbifold coverings}
We analyze the fundamental group in terms of the quadruple covering
\begin{align*}
\Cp_4(\lambda) &\longrightarrow \P^1 \\
\begin{bbmatrix} X \\ Y \\ Z \end{bbmatrix}&\longmapsto
\begin{bbmatrix} X \\  Z \end{bbmatrix}. \end{align*}

The above quadruple covering extends the map
\begin{align*}
\Ca_4(\lambda) &\longrightarrow \C \\
(x,y) &\longmapsto x \end{align*}
and maps the ideal point  $\qI\in\Cp_4(\lambda)$ to the ideal point
\[
\begin{bbmatrix} 1 \\  0 \end{bbmatrix} \longrightarrow 1/0  = \infty \in \P \]
with respect to the affine chart $\AChart[3]$. 
It is branched over $p_0,p_1,p_\lambda,\qI \in \Cp_4(\lambda)$ with
group of covering transformations  cyclic of order four.
Topologically this is an orbifold covering space
\[
\Z/4 \to \surf_3 \twoheadrightarrow \Sphere_{(4,4,4,4)}.\]

\subsubsection{Factorization into double coverings}

This quadruple covering factors as the composition of two double coverings
\[
\Cp_4(\lambda) \to \Cp_2(\lambda)\to \P^1 \] 
where the intermediate surface $\Cp_2(\lambda)$ is the elliptic curve in the Legendre family,
which $\Cp_4(\lambda)$ doubly covers.
The Abelian differential $\Phi_4(\lambda)$ is the pullback of the everywhere nonzero Abelian
differential $\Phi_2(\lambda)$ on $\Cp_2(\lambda)$.

We shall introduce new coordinates $u,U$,
where $u$ is an affine coordinate corresponding to $y^2$ and 
$U$ is corresponding homogeneous coordinate.
Specifically, 
the affine plane curve $\Ca_2(\lambda)$ defined by:
\begin{equation}\label{eq:AffineLegendre}
u^2 = x(x-1)(x-\lambda)\end{equation}
has
corresponding genus 1 projective curve $\Cp_2(\lambda)$ 
defined by
\[
\Cp_2(\lambda) := 
\left\{ \begin{bbmatrix} X \\ U \\ Z \end{bbmatrix} \ \Bigg|\ 
U^2Z - X (X-Z)(X-\lambda Z) = 0 \right\}.\]
Away from its single ideal point 
\begin{equation}\label{eq:DefpI}
\pI := 
\begin{bbmatrix} 0 \\ 1 \\ 0 \end{bbmatrix}\end{equation}
 the affine coordinates $u := U/Z$ and $x:= X/Z$ satisfy \eqref{eq:AffineLegendre}.
 The Abelian differential on $\Ca_2(\lambda)$ equals $u\inv dx$ which in homogeneous coordinates
 on  $\Cp_2(\lambda)$ equals
  \[
\Phi_2 := Y\inv Z\inv \ZdX .\]

\subsubsection{The elliptic involution and the Euclidean pillowcase}
The involution 
\[
\bbmatrix X \\ Y \\ Z \endbbmatrix 
\xmapsto{~\iota~} 
\bbmatrix X \\ -Y \\ Z \endbbmatrix 
\]
is the {\em elliptic involution\/} $\iota$ on $\Cp_2$.
The quotient map 
$\Cp_2 \longrightarrow \Cp_2/\langle\iota\rangle \cong \P^1$ 
identifies with the restriction to $\Cp_2$ of the projection
\begin{align}\label{eq:QuotientProjection}
\P^2 \setminus \{\pI\}&\longrightarrow \P^1 \notag \\
\bbmatrix X \\ Y \\ Z \endbbmatrix 
&\longmapsto 
\bbmatrix X  \\ Z \endbbmatrix 
\end{align}
which extends the coordinate projection $\Ca_2 \xrightarrow{~x~} \C$.
It is a branched covering space,
which we may view as a  double covering space of orbifolds:
\begin{equation}\label{eq:2coveringPillowcase}
\Torus \longrightarrow \Sphere_{(2,2,2,2)}. \end{equation}
The Riemann-Hurwitz theorem is just the multiplicativity of orbifold Euler characteristic under coverings:
\[
\chi(\Torus) = 2 \cho\big(\Sphere_{(2,2,2,2)}\big) \]
since 
\[
\chi(\Torus) = 0 = \cho\big(\Sphere_{(2,2,2,2)}\big) = 2 - 4 (1-1/2) .\]
We call this orbifold $\Sphere_{(2,2,2,2)}$ the 
{\em (Euclidean) pillowcase.\/}
Figure~\ref{fig:Pillowcase} depicts this covering space.

The (orbifold) fundamental group of the pillowcase is:
\begin{align*}
\pio\big(\Sphere_{(2,2,2,2)}\big)  & = 
\langle \eta_1,\eta_2,\eta_3,\eta_4  \mid  
\ \eta_1\,\eta_2\,\eta_3\ \eta_4    =  \\ & 
\qquad\qquad  (\eta_1)^2 = (\eta_2)^2 = (\eta_3)^2 = (\eta_4)^2  = e\rangle \end{align*}
and the monomorphism
of orbifold fundamental groups induced by the 
covering space \eqref{eq:2coveringPillowcase} is:
\begin{align*} 
\pi_1(\Torus) &\hookrightarrow \pio\big(\Sphere_{(2,2,2,2)}\big) \\
\mu_1 &\longmapsto  \eta_1\,\eta_2 \\
\mu_2 &\longmapsto   \eta_3\,\eta_2 
\end{align*}
where $\{\mu_1, \mu_2\}$ bases the free Abelian group 
$\pi_1(\Torus) \cong \Z \oplus \Z$.
Note that under the above homomorphism,
\[
[\mu_1,\mu_2] \longmapsto   (\eta_1\eta_2\eta_3) \big(\eta_1\inv\eta_2\inv\eta_3\inv\big). \]

We describe why this  defines a homomorphism.
To clarify the calculation, 
remove neighborhoods of the branch points on $\Torus$,
replacing $\Torus$ by the $4$-holed torus which we denote
$\Torus_{(\infty,\infty,\infty,\infty)}$.
Its image under the covering space is the $4$-holed sphere
which by analogy we denote $\Sphere_{(\infty,\infty,\infty,\infty)}$.
In other words replace the orbifold covering space
\eqref{eq:2coveringPillowcase}  by the covering space 
\begin{equation}\label{eq:TorusToPillowcaseWithHoles}
\Torus_{(\infty,\infty,\infty,\infty)} \longrightarrow \Sphere_{(\infty,\infty,\infty,\infty)} 
\end{equation}
of degree $2$.
The fundamental groups of 
$\Torus_{(\infty,\infty,\infty,\infty)}$ and $\Sphere_{(\infty,\infty,\infty,\infty)}$ are free,
of respective ranks $5$ and $3$ respectively.
They admit respective presentations:
\begin{align*}
\pi_1\Torus_{(\infty,\infty,\infty,\infty)} 
&= \langle \mu_1,\mu_2, \alpha_1,
\alpha_2,\alpha_3,\alpha_4  \mid  
[ \mu_1,\mu_2 ]\,  \alpha_1
\alpha_2\alpha_3\alpha_4 = e \rangle \\
\pi_1\Sphere_{(\infty,\infty,\infty,\infty)} &= 
 \langle \eta_1, \eta_2, \eta_3, \eta_4 \mid
\eta_1 \eta_2 \eta_3 \eta_4 = e \rangle  \end{align*}
where $[a,b]$ denotes $a b a\inv b\inv$.
We also write $a^b$ for $bab\inv$.
Then the monomorphism induced by \eqref{eq:TorusToPillowcaseWithHoles} is:
\begin{align*}
\pi_1\Torus_{(\infty,\infty,\infty,\infty)} &\longrightarrow \pi_1\Sphere_{(\infty,\infty,\infty,\infty)}  \\
\mu_1 &\longmapsto \eta_1\eta_2 \\
\mu_2 &\longmapsto \eta_3\eta_2 \\
\alpha_1 &\longmapsto (\eta_3)^2 \\
\alpha_2 &\longmapsto \big(\eta_2^{\eta_3\inv}\big)^2 \\
\alpha_3 &\longmapsto \big(\eta_2^{\eta_3\inv\eta_2\inv}\big)^2 \\
\alpha_4 &\longmapsto (\eta_4)^2 \\
\end{align*}

\begin{figure}[ht]
\centering
\def\svgwidth{16cm}
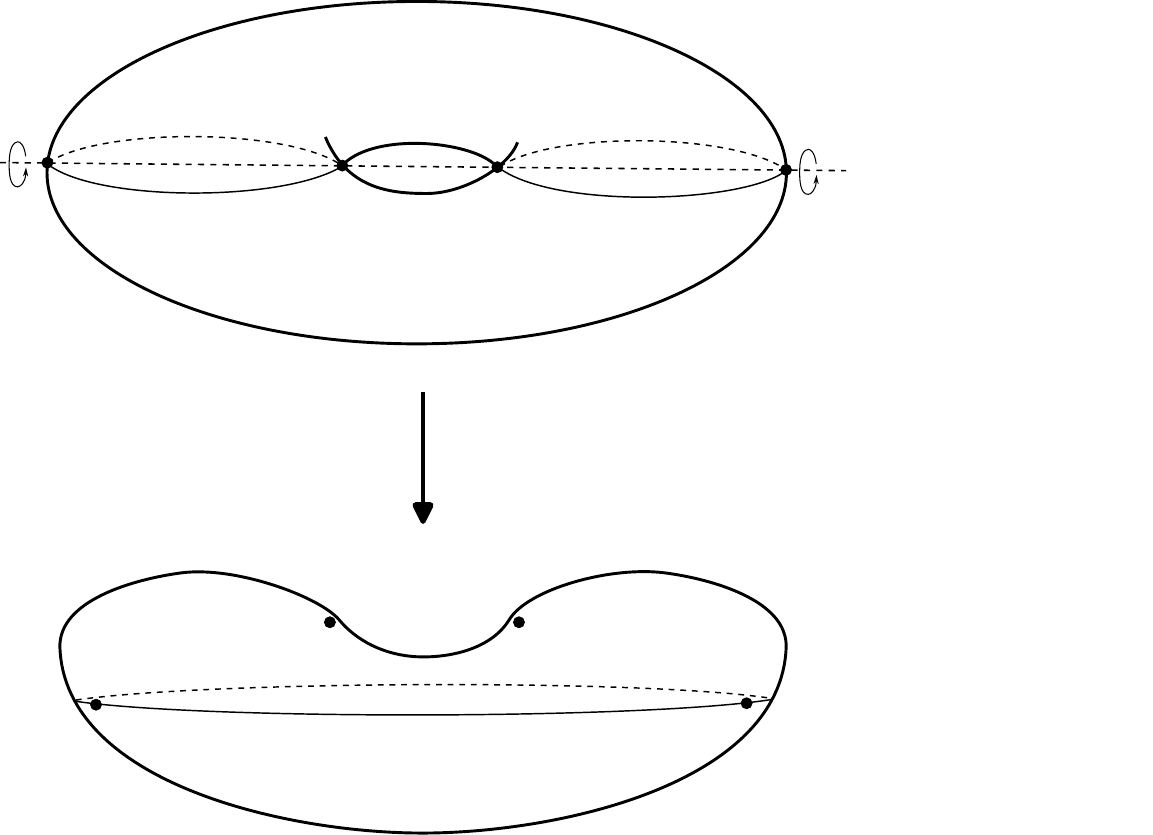
\caption{Mapping the torus to the pillowcase}
\label{fig:Pillowcase}
\end{figure}

\newpage

\subsubsection{The fundamental group of \texorpdfstring{$\pi_1(\surf_3)$}{}}
\ 
Using these ideas, we explicitly determine the  homomorphism
\[
\pi_1(\surf_3) \to \pio(\Torus_{(2,2,2,2)})\] induced by the mapping
$\Cp_4(\lambda)\to \Cp_2(\lambda)$.
Let
\[
\pi_1(\surf_3) = \langle
\xi_1,\xi_2,\xi_3,\xi_4,\xi_5,\xi_6 \mid\ 
[\xi_1,\xi_2]\,  [\xi_3,\xi_4]\,  [\xi_5,\xi_6] \,  = e 
\rangle \]
be the standard presentation of the genus $3$ surface group and
\begin{align*}
\pio(\Torus_{(2,2,2,2)}) 
& = \langle \mu_1,\mu_2, \alpha_1,
\alpha_2,\alpha_3,\alpha_4 \mid  \\ 
&  \qquad\quad 
(\alpha_1)^2 = 
(\alpha_2)^2 = 
(\alpha_3)^2 = 
(\alpha_4)^2 = \\ 
& \qquad \qquad\qquad \qquad [ \mu_1,\mu_2 ]\,  \alpha_1\,\alpha_2\,\alpha_3\, \alpha_4 \ 
= \  e \rangle. \end{align*}
and 
\begin{align*}
\pi_1(\surf_3) & \longrightarrow \pio\big(\Torus_{(2,2,2,2)}\big)  \\
\xi_1 &\longmapsto     \alpha_1\alpha_2\\
\xi_2 &\longmapsto     \alpha_3\alpha_2\\
\xi_3 &\longmapsto     (\mu_1)^{\alpha_4}\\
\xi_4 &\longmapsto     (\mu_2)^{\alpha_4}\\
\xi_5 &\longmapsto     \mu_1\\
\xi_6 &\longmapsto     \mu_2
\end{align*}
represents the monomorphism of fundamental groups.

Analogous to \eqref{eq:TorusToPillowcaseWithHoles},
presenting $\pio\big(\Sphere_{(4,4,4,4)}\big)$ as:
\[
\pio\big(\Sphere_{(4,4,4,4)}\big) = 
\langle \eta_1,\eta_2,\eta_3,\eta_4 \mid
\eta_1 \eta_2 \eta_3 \eta_4 =
(\eta_1)^4 = (\eta_2)^4 = (\eta_3)^4 = (\eta_4)^4  = e\rangle, \]
the induced monomorphism on fundamental groups is:
\begin{align*}
\pio\big(\Torus_{(2,2,2,2)}\big) &\longrightarrow \pio\big(\Sphere_{(4,4,4,4)}\big)  \\
\mu_1 &\longmapsto \eta_1\eta_2 \\
\mu_2 &\longmapsto \eta_3\eta_2 \\
\alpha_1 &\longmapsto (\eta_3)^2 \\
\alpha_2 &\longmapsto \big(\eta_2^{\eta_3\inv}\big)^2 \\
\alpha_3 &\longmapsto \big(\eta_2^{\eta_3\inv\eta_2}\big)^2 \\
\alpha_4 &\longmapsto (\eta_4)^2 \\
\end{align*}

\subsubsection{The deck transformation}
The mapping 
\[ 
\begin{bbmatrix} X \\ Y \\ Z \end{bbmatrix} \xrightarrow{~\TTT~} 
\begin{bbmatrix} X \\ \sqrt{-1}\, Y \\ Z \end{bbmatrix} 
\]
preserves $\Cp_4(\lambda)$ and has order $4$.
Furthermore it fixes the hyperplane $Y = 0$ which
meets $\Cp_4(\lambda)$ in the four points $p_0,p_1,p_\lambda,\qI$.
Since $y$ is a local coordinate at each of 
$p_0,p_1,p_\lambda$ and 
\[ 
\TTT^*y = \sqrt{-1}\, y, \]
the branching order at these points equals $4$.
A similar computation applies to $\qI$. 
Thus $\Phi$ is a deck transformation for the holomorphic quadruple branched covering map
$\Cp_4(\lambda)\longrightarrow \P^1$ defined by $x = X/Z$ on $\Ca_4(\lambda)$.

\subsubsection{The quaternionic representation}\label{sec:QuatRep}
We now describe the representation 
\[ 
\pi_1(\surf_3)\xrightarrow{~\rho~} \SUtwo \]
in terms of the orbifold covering space $\surf_3\to\Sphere_{(4,4,4,4)}$
and a representation 
\[ 
\widehat\pi := \pio\big(\Sphere_{(4,4,4,4)}\big)\xrightarrow{~\hrho~} \Qu \group \SUtwo. \]
The representation $\rho$ of $\pi_1(\surf_3)$  is the restriction of $\hrho$ to $\pi_1(\surf_3)$,
that is,
the composition of $\hrho$ with the monomorphism $\pi_1(\surf_3)\hookrightarrow  \pio\big(\Sphere_{(4,4,4,4)}\big)$.
The fundamental group 
\[
\widehat\pi = \pio\big(\Sphere_{(4,4,4,4)}\big) \] 
is generated by four based loops $\gamma_i$ ($i=1,2,3,4$) which encircle the points $0,1,\lambda,\infty \in \P^1$. 
The fundamental group of the complement $\Pddd $ is a $3$-generator free group
with redundant geometric presentation
\[
\pi_1(\Pddd) = 
 \langle \gamma_1,\gamma_2,\gamma_3,\gamma_4 \mid
 \gamma_1\gamma_2\gamma_3\gamma_4 =  e\rangle. \]
The orbifold fundamental group is obtained by adding relations making each of
$0,1,\lambda,\infty$ an order $4$ branch point:
\begin{align*}
 \pio\big( \Sphere_{(4,4,4,4)}\big)
&  =  
 \langle \gamma_1,\gamma_2,\gamma_3,\gamma_4  \mid
\gamma_1\gamma_2\gamma_3\gamma_4  = \\  
&\qquad  (\gamma_1)^4 = (\gamma_2)^4= (\gamma_3)^4= (\gamma_4)^4 = e \rangle. \end{align*}

We apply Proposition~\ref{prop:DiagonalHolonomy} to this representation, 
using the inclusion $\pi \hookrightarrow\widehat\pi$,
obtaining a decomposition 
\eqref{eq:IntoLineBundles} 
of the local system as a direct sum of three flat $\R$-line bundles $\Ll^\C_\bu$
where $\bu = \bi,\bj,\bk$. 

Consider the homomorphism $\hrho$ given by:
\begin{align*}
\pio\big( \Sphere_{(4,4,4,4)}\big) &\xrightarrow{~\hrho~} \Qu \group \SUtwo \\
\gamma_1 &\longmapsto \bi \\
\gamma_2 &\longmapsto \bj \\
\gamma_3&\longmapsto \bk \\
\gamma_4 &\longmapsto -\bOne 
\end{align*}
and its restriction $\rho$ to

\[ 
\pi_1(\Sigma_3)\cong \pi_1(\Cp_4) \group \pio\big( \Sphere_{(4,4,4,4)}\big). \]
Proposition~\ref{prop:DiagonalHolonomy} implies 
the representation of $\pi_1(\surf_3)$ decomposes as a direct sum of three $1$-dimensional real representations,
depending on the three basic unit quaternions $\bu = \bi, \bj, \bk$, 
respectively.

We find holomorphic sections of $\V_\rho^\bu := \Lu $ of the  form 
$ (x-x_0)^{1/2} \xi$ where $\xi$ is a parallel section 
defined near the branch point $x_0 = 0, 1,\lambda$.
Then $\V_\rho^\bu$-valued holomorphic $1$-forms will be obtained by tensoring 
holomorphic sections of $\V_\rho^\bu$ with 
the extension $\Psi^{(3)}$ to $\Cp_4$,
of the everywhere nonzero $y^{-3} dx$ on $\Ca_4$, 
defined in \eqref{eq:Psi3} below.

The representation $\hrho$ is surjective,
so its kernel $\Ker(\hrho)$ is a normal subgroup of
$\pio\big( \Sphere_{(4,4,4,4)}\big)$ whose quotient is isomorphic to $\Qu\group\SUtwo$.
Since the adjoint representation of $\SUtwo$ contains $-\bOne$ in its kernel, 
the image of the composition
\[
\pio\big( \Sphere_{(4,4,4,4)}\big) \xrightarrow{~\Ad\,\hrho~} \Ad\ \SUtwo
\]
equals the quotient $\Qu/\{ \pm\bOne \} \cong \Z/2 \oplus\Z/2$.
In particular we will define a 
$\Z/2 \oplus\Z/2$-covering space  of $\Sphere_{(4,4,4,4)}$ upon which meromorphic differentials are defined
to identify twisted holomorphic differentials.

\subsubsection{Construction of  twisted meromorphic differentials}
The representation
$\rho$ is the restriction of a representation $\trho$ of 
$\pi_1(\Pddd)$ to $\Qu \group \SUtwo$,
which we presently describe.
On an intermediate covering space,
which we shall describe,
lives {\em untwisted\/} holomorphic differentials which are equivariant 
respecting the deck transformations,
and pass down to the {\em twisted\/} (that is, $\V_\rho$-valued)
holomorphic differentials on $\Cp_4$.
We obtain this intermediate covering space  as the quotient of $\tPddd$ by 
an extension of $\pi_1(\surf_3)$ by $\Z/2\oplus\Z/2$
inside $\pio\big( \Sphere_{(4,4,4,4)}\big)$.

Here is the construction. 
Note that $\pi_1(\Pddd)$ is a free group of rank $3$,
redundantly presented by generators $\gamma_1,
\gamma_2,\gamma_3,\gamma_4$ subject to the relation
$\gamma_1\gamma_2\gamma_3\gamma_4=e$. 
Its quotient group $\pio\big( \Sphere_{(4,4,4,4)}\big)$ is obtained
by adding the relations $(\gamma_i)^4 = e$ for $i=1,2,3,4$.

Choose an arbitrary basepoint in $\Pddd$.
Take a based loop $\gamma_1$ which encircles $0$ once in the positive orientation,
a based loop $\gamma_2$ encircling $1$ once in the positive orientation,
a based loop $\gamma_3$ the based loop encircling $\lambda$ once in the positive orientation
and a based loop $\gamma_4$ encircling all three points once in the negative orientation.
These based loops  satisfy 
\[
\gamma_1\gamma_2\gamma_3\gamma_4 \simeq 1 \]
(where $1$ is the constant based loop)
and  define elements of  $\pi_1(\Pddd)$ which satisfy
\[
[\gamma_1][\gamma_2][\gamma_3][\gamma_4] = e \] 
in $\pi_1(\Pddd)$. 

Let $\V_\rho$ denote the complex local system $\sltwoC_\Adrho = \sutwo_\Adrho\otimes\C$ over $\Cp_4$.
We will find a basis for the Hodge space $\H^{(1,0)}(\Cp_4,\V_\rho)$ using the decomposition of 
$\V_\rho$ into flat (and hence holomorphic) line bundles $\Ll_\bi\oplus\Ll_\bj\oplus\Ll_\bk$,
following Proposition~\ref{prop:DiagonalHolonomy}.

Pass to the universal covering $\tPddd$ and consider the trivial $\V$-bundle 
\[ 
\widetilde{\Vv} := \tPddd \times \V \to \tPddd. \]
Elements of $\V$ define parallel (constant) sections of this bundle.
For $x_0 = 0, 1, \lambda$,
choose a holomorphic function,
which we denote  $\sqrt{x - x_0}$ on $\tPddd$,
whose square is $x - x_0$. 
Since, for example, $\rho(\gamma_1) = \bi$ 
the deck transformation corresponding to $\gamma_1$  acts on $\V$ by  conjugation by $\bi$. 
Similarly, 
the deck transformation corresponding to $\gamma_2$  acts on $\V$ by  conjugation by $\bj$,
and the deck transformation corresponding to $\gamma_3$  acts on $\V$ by  conjugation by $\bk$.
Since the monodromy of $\sqrt{x-x_0}$ is $-1$ on the generator encircling $x_0$,
the following $\V$-valued holomorphic $1$-forms on $\tPddd$ are anti-invariant under
$\widetilde\rho\big(\pi_1(\Pddd)\big)$
and invariant under the index-two subgroup of 
$\widetilde\rho\big(\pi_1(\Pddd)\big)$ corresponding to the double covering $\Cp_2(\lambda)\to \P^1$:

\begin{align}\label{eq:MeroSections}
\sigma_{0} & := \sqrt{x} \otimes \bi \notag\\
\sigma_{1} & := \sqrt{x-1} \otimes \bj \notag\\
\sigma_{\lambda} & := \sqrt{x-\lambda} \otimes \bk \notag\\
\sigma_{1,\lambda} & := \sqrt{(x-1)(x-\lambda)} \otimes \bi \notag \\
\sigma_{0,\lambda} & := \sqrt{x(x-\lambda)} \otimes \bj \notag\\
\sigma_{0,1} & := \sqrt{x(x-1)} \otimes \bk 
\end{align}

Therefore they induce holomorphic $\V_\rho$-valued sections on $\Pddd$,
which extend meromorphically to  sections of $\V_\rho$ over the orbifold on $\P^1$,
having branch points of order $4$ at $0,1,\lambda,\infty$. 

As $y$ is a local coordinate at $p_0$ and $y^4 = x(x-1)(x-\lambda) \approx x$ near $p_0$,
the function $x$ on $\Ca_4(\lambda)$ has quadruple zero at $p_0$.
On $\Cp_4(\lambda)$,
\[
\ldiv X/Z \rdiv = 4 [p_0] - 4 [\qI] \]
so on the double covering of $\Cp_4$ the function $\sqrt{X/Z}$ has divisor
\[
\ldiv \sqrt{X/Z} \rdiv = 2 [p_0] - 2 [\qI] \]
(where $p_0$ and $\qI$ denote the corresponding points on the double cover).
In particular $\sigma_{0,\bj}$ has a double zero at the point $p_0$ on the double cover.
Similarly each of the six meromorphic sections in \eqref{eq:MeroSections}
has a double zero at one of $p_0,p_1,p_\lambda$ and a double pole at $\qI$.

The meromorphic sections in \eqref{eq:MeroSections} each have double zeroes at exactly two of the points
$p_0,p_1,p_\lambda$ of the double cover, and a quadruple pole at $\qI$.

\subsubsection{The holonomy covering}
To write down single-valued holomorphic $1$-forms representing the cohomology classes
we find a covering space of  $\Cp_4$ upon  which the meromorphic functions
$\sqrt{x-x_j}$ are defined,
for $j=1,2,3,4$. 
This covering space will itself cover each of the three double coverings corresponding to
the flat line bundles $\V_\rho^\bu$ for $\bu = \bi, \bj, \bk$ defined in  \S\ref{sec:QuatRep}.
The smallest such covering space is the $\Z/2\oplus\Z/2$-covering space $X$ of $\Cp_4(\lambda)$
corresponding to $\Ker(\Ad \rho)$ upon which $\sqrt{x-x_1},$ $\sqrt{x-x_2},$ and $\sqrt{x-x_3}$ are defined.
Each of these meromorphic functions on $X$ is only defined up to $\pm 1$,
and we choose representatives so that 
\begin{equation}\label{eq:ProductOfSquareRoots}
\sqrt{x-x_1} \sqrt{x-x_2} \sqrt{x-x_3} = y^2 \end{equation}
on $X$.

The Abelian differential on $\Cp_4$ extending $y^{-3}dx$ on $\Ca_4$
\begin{equation}\label{eq:Psi3}
\Psi^{(3)}  =   ZY^{-3} \ZdX 
\end{equation}
has divisor 
\[
\ldiv \Psi^{(3)} \rdiv  =  4 \pointdivisor[\qI],\]
that is,
it's holomorphic on $\Cp_4$ and has only a quadruple zero at  $\qI$.
Multiply the sections in \eqref{eq:MeroSections} by $\Psi^{(3)}$ 
to obtain $\V_\rho$-valued holomorphic $1$-forms on $\Cp_4(\lambda)$ with double zeroes at exactly two of 
$p_0,p_1,p_\lambda,\qI$. 

\begin{align}\label{eq:HoloSections}
\sigma_{0}  \Psi^{(3)}  & := \sqrt{x}  \,  y^{-3}dx \otimes \bi  \notag\\
\sigma_{1,\lambda}\Psi^{(3)}  & := \sqrt{(x-1)(x-\lambda)} \, y^{-3}dx\otimes \bi \notag \\
\sigma_{1} \Psi^{(3)} & := \sqrt{x-1} \, y^{-3}dx \otimes \bj \notag\\
\sigma_{0,\lambda}\Psi^{(3)}  & := \sqrt{x(x-\lambda)} \, y^{-3}dx\otimes \bj \notag\\
\sigma_{\lambda}\Psi^{(3)} & := \sqrt{x-\lambda} \, y^{-3}dx\otimes \bk \notag\\
\sigma_{0,1}\Psi^{(3)}  & := \sqrt{x(x-1)} \, y^{-3}dx\otimes \bk 
\end{align}
having divisors
\begin{align}
\ldiv \sigma_{0}  \Psi^{(3)} \rdiv & = 2 \pointdivisor[p_0] + 2\pointdivisor[\qI]   \notag\\
\ldiv \sigma_{1,\lambda}  \Psi^{(3)} \rdiv & = 2 \pointdivisor[p_1] + 2\pointdivisor[p_\lambda]  \notag \\
\ldiv \sigma_{1}  \Psi^{(3)} \rdiv & = 2 \pointdivisor[p_1] + 2\pointdivisor[\qI]   \notag\\
\ldiv \sigma_{0,\lambda}  \Psi^{(3)} \rdiv & = 2 \pointdivisor[p_0] + 2\pointdivisor[p_\lambda]  \notag \\
\ldiv \sigma_{\lambda}  \Psi^{(3)} \rdiv & = 2 \pointdivisor[p_\lambda] + 2\pointdivisor[\qI]  \notag \\
\ldiv \sigma_{0,1}  \Psi^{(3)} \rdiv & = 2 \pointdivisor[p_0] + 2\pointdivisor[p_1]   \end{align}
The six $\V_\rho$-valued holomorphic $1$-forms of \eqref{eq:MeroSections} base
the complex vector space $\H^{(1,0)}(\Cp_4,\V_\rho)$ (the Hodge space).
With respect to this basis,
the matrix
\begin{equation}\label{eq:RankSixMatrix}
{\bigoplus}^3 \bmatrix 0 & 1 \\ 1 & 0 \endbmatrix =
\bmatrix
0 & 1 & 0 & 0 & 0 & 0 \\
1& 0  & 0 & 0 & 0 & 0 \\
0 & 0 & 0 & 1& 0 & 0  \\
0 & 0 & 1&  0& 0 & 0   \\
0 & 0 & 0 & 0 & 0 &1  \\
0 & 0 & 0 & 0 &1 & 0  
\endbmatrix
\end{equation}
represents the second fundamental form with respect to the Beltrami differential 
\[
\mu = \bo/\omega = (y/\overline{y})^2 d\overline{x} \otimes \partial_x.
\]
To this end, note that two basic elements whose coefficients $\bi,\bj,\bk$ are different are orthogonal, 
since different unit basic quaternions are orthogonal.
Then the $\V_\rho$-valued holomorphic differentials in \eqref{eq:HoloSections} fall into three disjoint pairs
$\sigma_{x_1}\Psi^{(3)}$ and $\sigma_{x_2,x_3} \otimes\bu$ 
where $\{x_1,x_2,x_3\} = \{0,1,\lambda\}$.
Then the value of the second fundamental form 
\begin{equation}\label{eq:ComputationSFFEW}
\B^\mu\Big(
\sigma_{x_1}  \Psi^{(3)}\otimes\bu,\,
\sigma_{x_2,x_3}  \Psi^{(3)}\otimes\bu \Big) =  1.\end{equation}
By \eqref{eq:SFFmu},
this value equals the integral
\[
\frac{i}2 \int_M  
\Alt
\big( 
\bb_*(\sigma_{x_1}\Psi^{(3}\otimes \sigma_{x_2,x_3}\Psi^{(3}
) 
\otimes\mu\big).
\]
The holomorphic quadratic differential
$\bb_*(\sigma_{x_1}\Psi^{(3)}\otimes \sigma_{x_2,x_3}\Psi^{(3)})$
equals
\[
\sqrt{(x-x_1)(x-x_2)(x-x_3)} y^{-6} dx^2 
= y^{-4} dx^2  \]
since  $\sqrt{(x-x_1)(x-x_2)(x-x_3)} =  y^2$
 by \eqref{eq:ProductOfSquareRoots}.

Then the pairing 
\[
\Alt
\big(\bb_*(\sigma_{x_1}\Psi^{(3)}\otimes \sigma_{x_2,x_3}\Psi^{(3)})
\otimes \mu \big) =
\vert y\vert^{-4} dx\wedge\overline{dx} = \omega\wedge\bo.\]
Then
\eqref{eq:ComputationSFFEW} follows since
\[
i/2 \int_{\Cp_4(\lambda)} \omega\wedge\bo = 1,\]
that is,
the translation surface $\EW= (\Cp_4(\lambda),\omega)$ has unit area \big(see \eqref{eq:areaOne}\big).

\begin{figure}[hbt!]
\centering
\def\svgwidth{15cm}
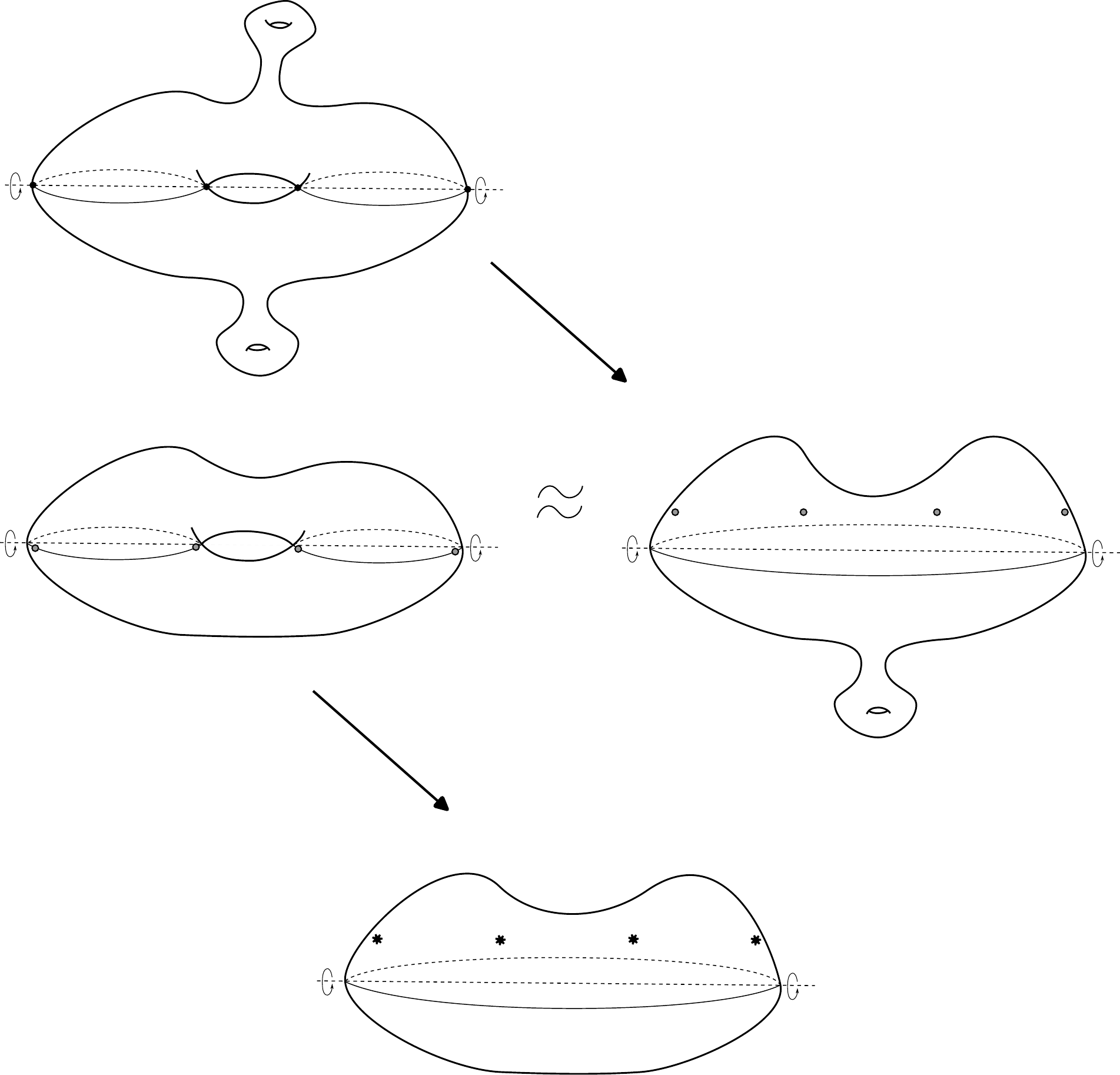
\caption{The Eierlegende Wollmilchsau $\Cp_4\approx \Sigma_3$ doubly covering $\Torus$ 
which in turn doubly covers  a right-angled pillowcase}
\label{fig:QuadrupleCovering}
\end{figure}

\subsection{The flat bundle over \texorpdfstring{$\surf_3$}{}}

For the special point $[\rho]$ on the $\SUtwo$-character variety of the surface $\surf_3$ underlying $\EW$,
the monodromy of this $\mathbb{R}$-variation of Hodge structure is exactly the action of a subgroup of the group $\Aff(\EW)$ of affine homeomorphisms on  
$
\H^1\big(\pi, \sutwo_{\Adrho}\big).
$

Applying $\SLtwoR$ to the square-tile picture of $\EW$,  
gives a family $\widetilde{\mathcal{C}}$ of translation surfaces of genus three with marked homotopies, 
so that we can construct the trivial bundle 
\[
\widetilde{\mathcal{C}} \times \XX\big(\pi,\SUtwo\big). \]
Let $\Gamma_\rho$ be the subgroup of $\Aff(\EW)$ generated by the affine homeomorphisms $T^2$ and $S^2$ pointwise fixing the conical singularities of $\EW$ with linear parts $\left(\begin{array}{cc}1&2\\0&1\end{array}\right)$ and $\left(\begin{array}{cc}1&0\\ 2&1\end{array}\right)$. 
By a direct calculation, one sees that it fixes the point $[\rho]\in\XX\big(\pi, \SUtwo\big)$.
Moreover,
$\Gamma_\rho$ possesses an index $8$ subgroup 
consisting of $g$ such that
$\rho\circ g = \rho$. 

The quotient 
\[
\Big(\widetilde{\mathcal{C}}\times\XX\big(\pi,\SUtwo\big)\Big)/\Gamma_\rho\] 
of this trivial bundle by the natural diagonal action of $\Gamma_\rho$ yields the  
Forni--Goldman bundle (see \cite{FG17}) over the quasi-projective base 
\[
\mB :=\widetilde{\mathcal{C}}/\Gamma_\rho\simeq \SOtwo\backslash\SLtwoR/\Gamma_\rho\]
corresponding to a finite cover of the Teichm\"uller curve of $\EW$. 
The special (fixed) point $[\rho]$ gives rise to a global section $\mu_0$ of the Forni--Goldman bundle. 

The tangent spaces to the fibers of the Forni--Goldman bundle at the points of $\mu_0(\mB)$ 
provide a vector bundle $V_{\mathbb{R}}$ over $\mB$. 
Alternatively, this real vector bundle can be seen as the quotient of the trivial vector bundle 
\begin{align*}
\widetilde{\mathcal{C}} &\times \T_{[\rho]}\XX\big(\pi, \SUtwo\big) \\
    &\downarrow   \\
&\widetilde{\mathcal{C}}  \end{align*}
by the natural action of $\Gamma_\rho$.
Precisely, 
$f\in \Gamma_\rho$ acts in the usual manner in the first factor and by the derivative $df([\rho])$ 
at the fixed point $[\rho]$ in the second factor. 
In terms of the interpretation of 
$\T_{[\rho]}\XX\big(\pi, \SUtwo\big)$
in Eilenberg-MacLane cohomology 
$\H^1\big(\pi,\sutwo_{\Ad \rho}\big)$, 
an element $f\in \Gamma$ acts on 
\[
(X,u)\in \widetilde{\mathcal{C}}\times \H^1\big(\pi,\sutwo_{\Ad \rho}\big) \]
by sending it to 
\[
(f^*(X), \Ad_h(u\circ f^{-1})),\]
where 
\[
\rho\circ f^{-1} = \Ad_{h^{-1}}(\rho) = h^{-1}\cdot\rho\cdot h \] 
with $h\in\SUtwo$. 

In this setting, 
we see that the action of $\Gamma_{\rho}$ on $\H^1(\pi, \mathfrak{su}(2)_{\Ad_{\rho}})$ 
is the monodromy of the vector bundle $V_{\mathbb{R}}$ over $\mB$ carrying a polarized $\R$-variation of Hodge structure 
after passing from the Betti to the de Rham point of view.

 We recall that the Kontsevich--Zorich cocycle has a single strictly positive and a single strictly negative Lyapunov 
 exponent, as well as $4$ zero exponents, on the Hodge bundle.  
 In contrast, for the twisted cohomology cocycle, that is, 
 for the fiber component of the tangent cocycle to the  moduli space of flat bundles, we have the result below (a reformulation of Theorem \ref{main_thm:EW}). 

 \begin{theorem} 
 \label{thm:exp_EW}
 The Lyapunov spectrum of the lift of the Teichm\"uller flow with respect to $\mu_0$ has $6$ strictly positive
 $($and $6$ strictly negative$)$  fiber Lyapunov exponents. 
 It follows that $\mu_0$ is non-uniformly hyperbolic for the lift of the Teichm\"uller flow to the moduli space of flat bundles.
 \end{theorem} 
 
 \begin{proof}
This result  follows from invoking a theorem of S. Filip~\cite{Fil17} 
on a suitable finite cover which trivializes the monodromy representation.
Indeed, 
the pull-back of the twisted cohomology bundle on this cover is an $\SL(2, \R)$-invariant subbundle of the Hodge bundle which 
has a second fundamental form of rank $6$ by \eqref{eq:RankSixMatrix}.
By a theorem of Filip's (see \cite{Fil17}, Theorem 1.1), 
the restriction of the Kontsevich--Zorich cocycle to this subbundle has exactly $6$ strictly positive exponents.  
By construction the fiber tangent cocycle of the lifted Teichm\"uller flow is isomorphic to the Kontsevich--Zorich cocycle on the pull-back of the twisted cohomology bundle, hence the statement follows. 
  \end{proof}

\subsection{Uniform expansion at \texorpdfstring{$\rho$}{} of the \texorpdfstring{$\Gamma_{\rho}$}{}-action}\label{ss.expansion-EW} 
Let $\nu$ be a probability measure on $\Gamma_\rho$. 
If $\nu$ satisfies the moment condition 
\[
\int_{\Gamma_{\rho}}
\big(\log\|g_*\|+\log\|g_*^{-1}\|\big)\, d\nu(g)<\infty,\]
then the random Lyapunov exponent of the $\Gamma_{\rho}$-action on $\H^1(\pi,\mathfrak{su}(2)_{\Ad_{\rho}})$ is well-defined. 

Recall that the $\Gamma_\rho$-action is called \emph{uniformly expanding} at $[\rho]$ if  
there exists $c>0$ and $n_0\in\mathbb{N}$ such that  
\[
\int\log\|g_*(v)\| \, d\nu^{(n_0)}(g)\geq c \]
for all $v\in \H^1(\pi,\mathfrak{su}(2)_{\Ad_{\rho}})$ where $\nu^{(k)}$ denotes the $k$-fold convolution of $\nu$.

Since the action of $\Gamma_{\rho}$ can be decomposed into strongly irreducible blocks (thanks to the semisimplicity) and it is unbounded on each block, 
a result of Furstenberg (see Benoist-Quint~\cite{BenoistQuint}) ensures that, 
for each block $V$, 
there exists $\lambda_V>0$ such that   
\[
\frac{1}{n}\log\|(g_n\dots g_1)_*v\|\to\lambda_V \]
for all $v\in V\setminus\{0\}$ and $\nu^{\mathbb{N}}$-almost every $(g_1,\dots, g_k,\dots)\in \Gamma_{\rho}^{\mathbb{N}}$. Therefore, we have that $\Gamma_\rho$ is uniformly expanding at $[\rho]$ in view of Lemma 2.3 of Cantat--Dujardin paper \cite{CD23} (see also Section 8.2 of this paper).

\begin{remark} Somewhat in the spirit of Section 9 of Cantat--Dujardin paper \cite{CD23}, it would be nice to know if the uniform expansion of $\Gamma_\rho$ at a finite number of periodic points is sufficient to get the uniform expansion of $\Gamma_\rho$ on the character variety with respect to the symplectic measure.
On the nonsingular components, 
the symplectic measure is a smooth measure and is finite.
For the singular components,
\eqref{sec:FiniteMeasure} discusses the finiteness of the measure
(compare also Huebschmann~\cite{Hueb} and Sjamaar-Lerman~\cite{SjamaarLerman}).
\end{remark} 

\begin{remark}[Integrality]
In the specific case of $\rho$, 
the variation of Hodge structure determined by 
the vector bundle $\V_\R$ over $\mB$ is \emph{integral}:
\begin{align*}
\rho\circ T^2 & =\bj^{-1}\cdot\rho\cdot \bj\\ 
\rho\circ S^2 & =\bi^{-1}\cdot\rho\cdot \bi,
\end{align*}
implies that $\Gamma_\rho$ is the stabilizer of $[\rho]$ in  $\Aff_{(1)}(\EW)$.
Furthermore  $\Ad_{\bu}$, $\bu\in \Qu$ permutes the set
$\Qu\setminus\{\pm \bOne\}$
and the action of $\Gamma_{\rho}$ preserves  the integral lattice 
$\Z[\bi,\bj,\bk]\group \Ham_0$ consisting of {\em Lipschitz integers\/} 
(integral linear combinations of $\bOne, \bi,\bj,\bk \in \Ham$).
The image of the Lipschitz-integral Eilenberg-MacLane cocycles
\[
\big\{u\in \o{Z}^1(\pi,\sutwo_{\Ad\rho} )\mid
u(\gamma)\in \Z[\bi,\bj,\bk] \big\}\]
in 
$\H^1(\pi,\sutwo_{\Ad\rho})$
defines a
$\Gamma_{\rho}$- invariant lattice in $\H^1(\pi,\sutwo_{\Ad\rho})$.
\end{remark}

\section{The Platypus}\label{sec-3}
In this section, we first establish a \emph{partial} version of the main theorem (Theorem \ref{main_thm:P}) via a Hodge-theoretical discussion which is similar to our previous arguments for the Eierlegende Wollmilchsau $\EW$, and, subsequently, we complete the proof of the main theorem with a concrete computation.

The affine plane curve $\Ca_6(\lambda)$ defined by 
$$y^6=x(x-1)(x-\lambda)$$ 
corresponds to a projective curve $\Cp_6$ of genus four. The Abelian differential $dx/y^3$ extends to an Abelian differential $\Phi_6(\lambda)$ on $\Cp_6(\lambda)$ and the \emph{Platypus} -- or \emph{Ornithorynque} -- is the translation surface 
$$\Plat(\lambda):=(\Cp_6(\lambda),\omega(\lambda)),$$ 
where $\omega(\lambda)=c(\lambda)\Phi_6(\lambda)$ and $c(\lambda)>0$ is chosen so that 
$$\frac{i}{2}\int_{\Ca_6(\lambda)}\omega(\lambda)\wedge\overline{\omega(\lambda)}=1.$$

See Forni-Matheus \cite{FM08}, Matheus-Yoccoz \cite{MY}, Matheus-Yoccoz-Zmiaikou \cite{MYZ}, and Forni-Matheus-Zorich \cite{FMZ14} for background on this translation surface.

For later use, we recall the following description of Platypus as a translation surface of genus four with three conical singularities $A_{1,0}$, $A_{0,1}$, $A_{1,1}:=\circ$.

Combinatorially, $\Plat$ is obtained from twelve unit squares $Sq(i,\mu,\nu)$, $i\in\mathbb{Z}/3\mathbb{Z}$, $\mu,\nu\in\mathbb{Z}/2\mathbb{Z}$ such that the neighbor to the right of $Sq(i,\mu,\nu)$ is 
\begin{itemize}
\item $Sq(i,\mu+1,\nu)$ if $\mu=1$;
\item $Sq(i+1,\mu+1,\nu)$ if $\mu=0$ and $\nu=0$; 
\item $Sq(i-1,\mu+1,\nu)$ if $\mu=0$ and $\nu=1$,
\end{itemize} 
and the neighbor on the top of $Sq(i,\mu,\nu)$ is 
\begin{itemize}
\item $Sq(i,\mu,\nu+1)$ if $\nu=1$; 
\item $Sq(i+1,\mu,\nu+1)$ if $\nu=0$ and $\mu=1$; 
\item $Sq(i-1,\mu,\nu+1)$ if $\nu=0$ and $\mu=0$. 
\end{itemize} 
As it turns out, for each $i\in\mathbb{Z}/3\mathbb{Z}$, the four squares $Sq(i,\mu,\nu)$, $\mu,\nu\in\mathbb{Z}/2\mathbb{Z}$ can be put together into a $2\times2$ square whose center $A_i$ is a regular point of $\Plat$ while its vertices correspond to a conical singularity $A_{1,1}$, the midpoints of its horizontal sides correspond to a conical singularity $A_{0,1}$, and the midpoints of its vertical sides correspond to $A_{1,0}$. The horizontal segment at the middle of this $2\times 2$ square passing through $A_i$ and joining $A_{1,0}$ to itself is denoted $\vec{A_i}$. 

\begin{figure}[hbt!]
\centering
\def\svgwidth{11cm}
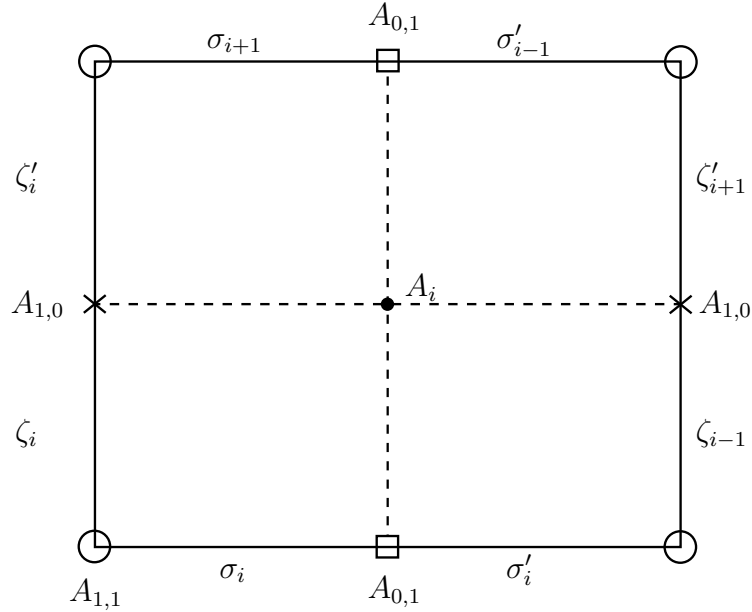
\caption{Four squares in the Platypus}
\label{f.hor-cyl-EWpicture3}
\end{figure} 

For our purposes, it is convenient to label $\sigma_i$ the horizontal bottom side of $Sq(i,1,1)$, $\zeta_i$ the vertical left side of $Sq(i,1,1)$, $\sigma_i'$ the horizontal bottom side of $Sq(i,0,1)$, and $\zeta_i'$ the vertical left side of $Sq(i,1,0)$. 

\subsection{Choice of generators of \texorpdfstring{$\pi_1(\Plat,\circ)$}{}} A quick inspection of the $1$-skeleton of $\Plat$ associated to the decomposition of $\Plat$ into horizontal cylinders in the figure below shows that the fundamental group $\pi_1(\Plat,\circ)=\pi_1(\Plat,A_{1,1})$ is generated by 
$$\{\sigma_j'\sigma_i: i,j\in\mathbb{Z}/3\mathbb{Z}\}\cup \{\sigma_j^{-1}\sigma_i: i,j\in\mathbb{Z}/3\mathbb{Z}, i\neq j\} \cup \{\zeta_0'\zeta_2, \zeta_0'\vec{A}_1(\zeta_0')^{-1}, \zeta_0'\vec{A_2}(\zeta_0')^{-1}\}.$$ 

\begin{figure}[hbt!]
\centering
\def\svgwidth{11cm}
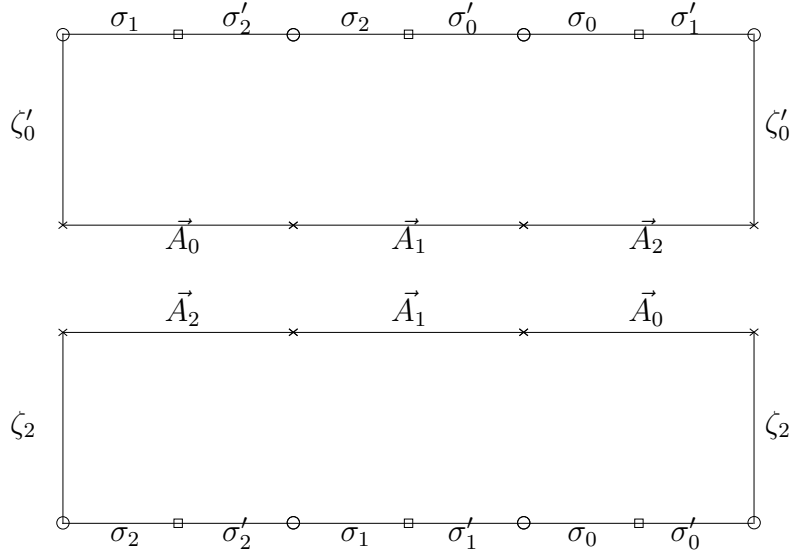
\caption{Decomposition of the Platypus into two horizontal cylinders}
\label{f.hor-cyl-EWpicture4}
\end{figure} 

This list can be simplified\footnote{Actually, we can reduce to $8$ generators thanks to the ``long'' relation $$\sigma_2'\sigma_1 = (\sigma_2'\sigma_2)(\sigma_0'\sigma_2)^{-1}(\sigma_0'\sigma_0)(\sigma_1'\sigma_0)^{-1}(\sigma_1'\sigma_1),$$ but we temporarily avoid to do so because the usage of $\sigma_2'\sigma_1$ allows us to keep ``shorter'' formulas in what follows.} to $9$ generators 
$$\{\sigma_0'\sigma_0\, \sigma_1'\sigma_0, \sigma_1'\sigma_1, \sigma_2'\sigma_1,\sigma_0'\sigma_2, \sigma_2'\sigma_2\}\cup \{\zeta_0'\zeta_2, \zeta_0'\vec{A_1}(\zeta_0')^{-1},\zeta_0'\vec{A_2}(\zeta_0')^{-1}\}$$ 
satisfying the relation 
$$\zeta_0'\vec{A_2}\vec{A_1}\zeta_2(\sigma_0'\sigma_0)(\sigma_1'\sigma_1)(\sigma_2'\sigma_2) = (\sigma_1'\sigma_0) (\sigma_0'\sigma_2)(\sigma_2'\sigma_1)\zeta_0'\vec{A_1}\vec{A_2}\zeta_2.$$ 
\begin{remark} Note that $\zeta_0'\vec{A_i}\vec{A_j}\zeta_2 = (\zeta_0'\vec{A_i}(\zeta_0')^{-1})(\zeta_0'\vec{A_j}(\zeta_0')^{-1})(\zeta_0'\zeta_2)$. 
\end{remark}

For the ease of notation, we shall denote the previous generators by 
$$a=\sigma_0'\sigma_0, \quad b=\sigma_1'\sigma_0, \quad c=\sigma_1'\sigma_1,$$ 
$$d=\sigma_2'\sigma_1, \quad e=\sigma_0'\sigma_2, \quad f=\sigma_2'\sigma_2,$$ 
$$g=\zeta_0'\zeta_2, \quad h=\zeta_0'\vec{A_1}(\zeta_0')^{-1}, \quad i=\zeta_0'\vec{A_2}(\zeta_0')^{-1}.$$ 

\subsection{A special quaternionic representation \texorpdfstring{$\rho_0$}{}} The Platypus $\Plat$ is a cyclic cover of the flat pillowcase $\overline{\mathbb{C}}$ branched over four points $0,\pm1,\infty$: more concretely, a generator $C$ (of order $6$) of the deck group of this cover corresponds to composing the translation of the $2\times 2$ square with center $A_i$ to the corresponding square with center $A_{i+1}$ with the involution of this square about its center $A_{i+1}$. We set $\ast=-1/2$ as the base point of $\overline{\mathbb{C}}\setminus\{0,\pm1,\infty\}$, and we denote by $p_1, p_2, p_3, p_4\in \pi_1(\overline{\mathbb{C}}\setminus\{0,\pm1,\infty\},\ast)$ the positively oriented loops going around $0$, $-1$, $1$ and $\infty$ (resp.), so that $p_1p_2p_4p_3=\textrm{id}$. Observe that the branched cyclic cover $\Plat\to\overline{\mathbb{C}}$ is determined by the facts $A_{1,1}$ projects to $0$, $A_{1,0}$ projects to $1$, $A_{0,1}$ projects to $-1$, each $A_{i}$ projects to $\infty$, and $p_1^6$, $p_2^6$, $p_3^6$ and $p_4^2$ lift to closed loops around $A_{1,1}$, $A_{0,1}$, $A_{1,0}$ and $A_i$'s. This setting is summarized by the following picture.  

\begin{figure}[hbt!]
\centering
\def\svgwidth{11cm}
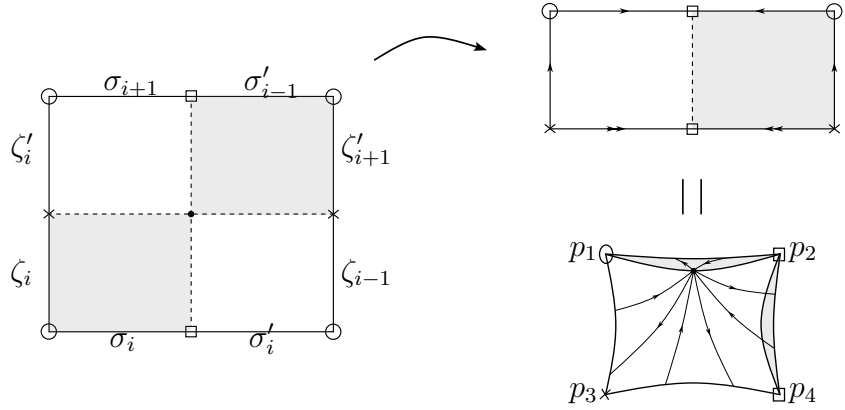 
\caption{Pillowcase as the quotient of the Platypus}
\label{f.hor-cyl-EWpicture5} 
\end{figure} 

By a slight abuse of notation, denote by $\ast$ the point of $\sigma_0$ projecting to $\ast$ in the flat pillowcase. Using the piece of $\sigma_0$ between $\circ$ and $\ast$, we can build an isomorphism $\pi_1(\Plat,\circ)\simeq \pi_1(\Plat,\ast)$. By deforming the elements of $\pi_1(\Plat,\ast)$ away from the points of the set $\Sigma=\{A_{1,1}, A_{0,1}, A_{1,0}, A_0, A_1, A_2\}$ of marked points, we obtain natural morphisms 
\begin{eqnarray*} 
\pi_1(\Plat,\circ)\simeq \pi_1(\Plat,\ast)&\to& \pi_1(\Plat\setminus\Sigma, \ast)/\langle \textrm{loops encircling points in }\Sigma\rangle \\ 
&\to& \pi_1(\overline{\mathbb{C}}\setminus\{0,\pm1,\infty\},\ast)/\langle p_1^6, p_2^6, p_3^6, p_4^2\rangle 
\end{eqnarray*} 

The images of our preferred generators of $\pi_1(\Plat,\circ)$ under the composition of these morphisms are 

$$
\begin{cases}
a  &\to p_1 p_2^{-1}    \\
b  &\to   p_1^{-1}  p_2    \\ 
c  &\to  p_1^{-1} p_2^{-1}   p_1^{2}  \\
d &\to      p_1^{-3} p_2   p_1^{2}      \\
e  &\to   p_1 p_2    p_1^{-2}          \\
f  &\to  p_1^{3}  p_2^{-1} p_1^{-2}   \\ 
g   &\to    p_1^{-2}    p_3^{-3}    p_1^{-1}      \\
h &\to   p_1^{-2}  p_3 p_4  p_3^2    p_1^2           \\
i &\to  p_1^{-2} p_3^{-1}  p_4  p_3^{-2}  p_1^2
\end{cases} 
$$

In this setting, we can define a representation $\rho_0:\pi_1(\Plat,\circ)\to \SU(2)$ with values in the binary tetrahedral group $\biTet$ along the following lines. Recall that $\biTet$ is the finite group of order $24$ of $\SU(2)$ generated by 
$$s=\frac{1}{2}\left(\begin{array}{cc}1+i & 1+i \\ -1+i & 1-i \end{array}\right) \quad \textrm{ and } \quad t=\frac{1}{2}\left(\begin{array}{cc}1+i & 1-i \\ -1-i & 1-i\end{array}\right).$$ 
One has that $\biTet=\langle s, t: s^3=t^3=(st)^2=-I_2\rangle$ (where $-I_2$ is the $2\times 2$ identity matrix) and $\biTet$ is isomorphic to $\SL(2,\mathbb{F}_3)$ via 
$$\varphi(t)=-A \quad \textrm{ and } \quad \varphi(s) = -B$$ 
where 
$$A\equiv_3\left(\begin{array}{cc} 1 & 1\\ 0 & 1\end{array}\right) \quad \textrm{and} \quad B\equiv_3\left(\begin{array}{cc} 1 & 0\\ 1 & 1\end{array}\right).$$ 
At this point, we can define $\rho_0$ by sending $p_1\to A$, $p_2\to BAB^{-1}$, $p_3\to B^{-1}$, $p_4\to -I_2$, so that our preferred generators of $\pi_1(\Plat,\circ)$ are mapped under $\rho_0$ according to the following formula: 
$$
\rho_0 : \begin{cases}
a  &\to A  B  A^{-1} B^{-1}    \\
b  &\to   A ^{-1}  B A  B^{-1} \equiv_3  -A B     \\
c  &\to  A^{-1}  B  A^{-1} B^{-1}   A^{2} \equiv_3  - A^{-1}  B^{-1}     \\
d &\to      A^{-3} B A  B^{-1}    A^{2}  \equiv_3  B A  B^{-1}    A^{-1}    \\
e  &\to   A B A  B^{-1}     A^{-2}   \equiv_3     -  B A     \\  
f  &\to  A^{3} B  A^{-1} B^{-1}  A^{-2}   \equiv_3  - B^{-1} A^{-1}   \\ 
g   &\to  A^{-2} B^3 A^{-1}            \equiv_3    I_2       \\
h &\to  -  A^{-2}  B^{-3}     A^2    \equiv_3   - I_2      \\
i &\to  - A^{-2} B^3  A^2    \equiv_3   -I_2
\end{cases} 
$$

\begin{remark} For later reference, we note that the adjoint actions of $A=\varphi(-t)$ and $B=\varphi(-s)$ on the Lie algebra $\mathfrak{su}(2)$ are given by 
$$\Ad_A(u_1) = u_2, \quad \Ad_A(u_2) = - u_3, \quad \Ad_A(u_3) = - u_1,$$ 
$$\Ad_B(u_1) = u_3, \quad \Ad_B(u_2) =  - u_1, \quad \Ad_B(u_3) =  - u_2,$$ 
where 
$$u_1=\left(\begin{array}{cc}0&i\\i&0\end{array}\right) = \bj, \quad u_2=\left(\begin{array}{cc}0&-1\\1&0\end{array}\right) = \bk, \quad u_3 = \left(\begin{array}{cc}i&0\\0&-i\end{array}\right) = \bi$$ 
are the Pauli matrices. 
\end{remark} 

\begin{remark} By definition, the representation $\rho_0$ factors through the orbifold fundamental group of the pillowcase obtained from the quotient $\Plat$ by the generator $C$ of the relevant branched cover. In particular, $\rho_0$ must be $C$-invariant up to conjugacy and, as a sanity test, this can be verified directly by computing the action of $C$ on our preferred generators of $\pi_1(\Plat,\circ)$ in order to eventually conclude that $\rho_0(C(\theta)) = A^2\rho_0(\theta) A^{-2}$ for all $\theta\in\pi_1(\Plat,\circ)$. 
\end{remark}

\begin{remark} As it is discussed in Appendix \ref{a.SU2O}, $\rho_0$ is one of the four smooth points of the $\SU(2)$ character variety of $\Plat$ which is fixed by a natural finite index subgroup of its affine homeomorphisms. 
\end{remark}

\subsection{Twisted holomorphic differentials} In the sequel, we consider an affine model 
$$w^6=(z-z_1)(z-z_2)(z-z_3)(z-z_4)^3$$ 
of the Platypus. 

\begin{lemma} \label{lem:twist_hol_diff_O}  The space $\o{H}^{1,0}_{\rho_0}  ( \pi_1 (\Plat, \circ),  \mathfrak{su}(2)^\C_{\Ad_{\rho_0}}  )$ of twisted holomorphic differentials splits as the direct sum of $3$ invariant spaces $\V_{\rho_0}^\bu$ for $\bu = \bi, \bj, \bk$ of  complex dimension $3$ spanned by sections which locally tensor products with an eigenvector $\bu$ of the adjoint action. A basis for the  space $\V_{\rho_0}^\bu$ can be written 
in the coordinates $(z,w)$ as follows: for $i , j, k \in \{1, 2,3\}$ distinct, 
$$
\begin{aligned} 
(z-z_i)^{1/2} &(z-z_4)^{3/2}   dz/w^{4}  \otimes \bu \,, \quad (z-z_i)^{1/2} (z-z_4)^{5/2}   dz/w^{5}   \otimes \bu \,,   \\ &  (z-z_j)^{1/2} (z-z_k)^{1/2}  (z-z_4)^2   dz/w^5 \otimes \bu \,.
\end{aligned}
$$

\end{lemma} 

\begin{proof}
We recall that a basis of the space  $\o{H}^{1,0}$ of holomorphic differentials of the Platypus 
of equation
$$
w^6 = (z-z_1) (z-z_2) (z-z_3) (z-z_4)^3 \,,
$$
is given by the formulas 
$$
\begin{aligned}
&(i)\,\,  (z-z_4) dz/w^3 , \quad 
(ii) \,\,  (z-z_4)^2 dz/w^4  \\ 
&(iii)\,\,  (z-z_4)^2 dz/w^5   \,, \quad 
(iv)\,\,  (z-z_4)^3 dz/w^5  \,.
\end{aligned}
$$
Recall that the translation structure of $\Plat$ is given by a holomorphic differential $\omega= c(z-z_4) dz/w^3$ which is
anti-invariant with respect to the cyclic automorphism  
\[
\Psi :   (z, w) \to   (z,  \zeta w)   
\]
with $\zeta \in \C$ a primitive $6$-th root of unity. Hence, a basis of quadratic  differentials is
$$
\begin{aligned}
&(1) \,\, (z-z_4)^2 dz^2/w^6     =  (i)^2\, \\ 
&(2)\,\, (z-z_4)^3 dz^2/w^7    =           (i)\cdot  (ii)\\ 
&(3) \,\,(z-z_4)^3 dz^2/w^8      =         (i) \cdot (iii) \\
&(4) \,\,(z-z_4)^4 dz^2/w^8    =        (i) \cdot (iv) =   (ii)^2 \\ 
& (5)\,\, (z-z_4)^4 dz^2/w^9       =        (ii) \cdot (iii)\\
&(6) (z-z_4)^5 dz^2/w^9        =       (ii) \cdot (iv) \\
&(7) \,\, (z-z_4)^4 dz^2/w^{10}     =        (iii)^2 \\
&(8) \,\,(z-z_4)^5 dz^2/w^{10}      =       (iii) \cdot (iv) \\
&(9) \,\, (z-z_4)^6 dz^2/w^{10}     =        (iv)^2\,.
\end{aligned} 
$$
In fact, since $\Plat$ has genus $g=4$,  the space of quadratic differentials has dimension $3g-3= 9$, and the above differentials are linearly independent.   All of the above quadratic differentials have zero integral, with the exception of $(1) = (i)^2$, which is invariant under the cyclic group $\langle \Psi \rangle$. 

In particular, the following holomorphic quadratic differentials vanish only at $z_4$:
$$
\begin{aligned}
& (z-z_4)^2 dz^2/w^6  \,, \quad  (z-z_4)^3 dz^2/w^7 \\
& (z-z_4)^3 dz^2/w^8 \,, \quad  (z-z_4)^4 dz^2/w^9   \,.
\end{aligned} 
$$
and the following holomorphic quadratic differentials vanish also
on a set of the set $\{z_1, z_2, z_3\}$: for any
$i,j \in \{1, 2, 3\}$, 
$$
\begin{aligned}
& (z-z_i) (z-z_4)^3 dz^2/w^8 \,, \quad  (z-z_i) (z-z_4)^4 dz^2/w^9 \\
& (z-z_i) (z-z_4)^4 dz^2/w^{10}\,, \quad  (z-z_i) (z-z_4)^5 dz^2/w^{10}    \,, \\ & (z-z_i) (z-z_j) (z-z_4)^4 dz^2/w^{10}\,.
\end{aligned} 
$$

Next we write a basis of multi-valued holomorphic differentials, which are square roots of holomorphic quadratic differentials and are not single-valued, hence are well defined on $\Plat$ up to multiplication times $\pm 1$, but are well-defined on the universal cover of $\Plat \setminus \pi^{-1}\{z_1, \dots, z_4\}$:
$$
\begin{aligned}
&(z-z_4)^{3/2} \frac{dz}{w^{7/2}} \,, \quad  (z-z_4)^{3/2} \frac{dz}{w^{4}} \,,  \quad (z-z_i)^{1/2}  (z-z_4)^{3/2} \frac{dz}{w^{4}}  \,,  \\ &(z-z_4)^2 \frac{dz}{w^{9/2}} \,, \quad (z-z_i)^{1/2}(z-z_4)^2  \frac{dz}{w^{9/2}} \,, 
\quad (z-z_i)^{1/2} (z-z_4)^2 \frac{dz}{w^{5}} \,,  \\ & (z-z_i)^{1/2} (z-z_j)^{1/2} (z-z_4)^2 \frac{dz}{w^{5}} \,.
\end{aligned}
$$

We have therefore completed the list of multi-valued
holomorphic sections whose square is a holomorphic quadratic differential.  

Next we impose the condition that 
$3$ selected multi-valued holomorphic sections have all  $3$ mutual products equal to a well-defined holomorphic
differential.  This condition immediately implies that the $3$ (up to multiplicative constants)  multi-valued differential
for the form  $f(z) dz/ w^{7/2}$ and $g(z) dz/w^{9/2}$ cannot appear, since their products  are not all (single-valued) holomorphic 
quadratic differentials  (and of course none of the products with multi-valued differentials of the form $h(z) dz/w^c$ with
$c$ an integer is holomorphic).   

\smallskip 
At this point we have argued that the list of the $3$ multi-valued differentials can only include those which are of the form
$f(z)dz/w^4$ or $g(z) dz/w^5$   However the multi-valued differential $(z-z_4)^{3/2}   dz/w^{4}$ cannot appear since its products with the other retained differentials are not single-valued holomorphic quadratic differentials.

\smallskip
We therefore conclude that the only possible triplets of multi-valued differentials (under the conditions that all products are holomorphic quadratic differentials) are of the form
$$
\begin{aligned} 
(z-z_i)^{1/2} (z-z_4)^{3/2}   dz/w^{4} \,, \quad & (z-z_i)^{1/2} (z-z_4)^{5/2}   dz/w^{5} \,,   \\ &  (z-z_j)^{1/2} (z-z_k)^{1/2}  (z-z_4)^2   dz/w^5 
\end{aligned}
$$
While the product of the first two multi-valued differentials in the above list is clearly a holomorphic quadratic differentials,
the product of the first two with the third is a holomorphic quadratic differential for $i\not=j\not =k$ since 
$$
(z-z_i)^{1/2} (z-z_j)^{1/2} (z-z_k)^{1/2}  (z-z_4)^{3/2}  =  w^3\,,
$$
thus we have 
$$
\begin{aligned}
{[}(z-z_i)^{1/2} (z-z_4)^{3/2}  \frac{dz}{w^{4}} {]} &  \cdot   [(z-z_j)^{1/2} (z-z_k)^{1/2}  (z-z_4)^2   \frac{dz}{w^{5}}]  \\ &= 
(z-z_4)^{2}  \frac{dz^2}{w^{6}}  \\
{[} (z-z_i)^{1/2} (z-z_4)^{5/2}   dz/w^{5} {]} &  \cdot   [(z-z_j)^{1/2} (z-z_k)^{1/2}  (z-z_4)^2   \frac{dz}{w^{5}}]  \\ &= 
(z-z_4)^{3}  \frac{dz^2}{w^{7}}  \,.
\end{aligned} 
$$
The argument is therefore concluded.
\end{proof} 

\begin{lemma}  \label{lem:B_O}
 The matrix $\B^{\mu}$, $\mu=\frac{\overline{\omega}}{\omega}$, of the second fundamental form has a diagonal
 block structure with $3$  blocks $3\times 3$ of the form
 $$
 \begin{pmatrix}  0  & 0 & 1 \\ 0  & 0 & 0  \\ 1  & 0 & 0   \end{pmatrix}\,.
 $$
 In particular, it has maximal rank $2\cdot 3=6$.
 \end{lemma} 
 \begin{proof}  By definition it is enough to compute
 $$
  \frac{i}{2}\int_{\Plat} \Alt \big(\bb_*(\theta_{ij}\otimes\theta_{kl})  \otimes \mu\big)\,,
 $$
 where $\theta_{ij}$ are the twisted holomorphic differentials found in Lemma \ref{lem:twist_hol_diff_O} and $\mu=\overline{\omega}/\omega$ with  
 $$
 \omega= c(z-z_4) \frac{dz}{w^3} \,.
 $$  
 
 In the computation of the second fundamental form we first observe that the block structure is given by the fact that 
 the basis $\{u_1, u_2, u_3\}$ of the Lie algebra $\mathfrak{su}(2)$ is orthogonal (and by the form of the holomorphic differentials).
 Next, we compute the blocks. For instance, the first block has entries
 $$
 \begin{aligned}
 &B_{11}=  \frac{i}{2}\int_{\Plat } \frac{(z-z_i) (z-z_4)^{2}}{w^2}  \vert z-z_4\vert^2      \frac{dz \wedge \bar{dz} }{\vert w \vert^6} \, , \\
 & B_{22} =  \frac{i}{2}\int_{\Plat } \frac{(z-z_i) (z-z_4)^{3}}{w^4 }   \vert z-z_4\vert^2     \frac{dz \wedge \bar{dz} }{\vert w \vert^6}  \,, \\
 & B_{33} =  \frac{i}{2}\int_{\Plat } \frac{(z-z_j) (z-z_k)  (z-z_4)^2}{ w^4 }     \vert z-z_4\vert^2      \frac{dz \wedge \bar{dz} }{\vert w \vert^6}  \,, \\
 &B_{12}=B_{21} =  \frac{i}{2}\int_{\Plat } \frac{(z-z_i)  (z-z_4)^2} {w^{3} }     \vert z-z_4\vert^2    \frac{dz \wedge \bar{dz} }{\vert w \vert^6}     
 \,, \\
  & B_{23}=B_{32} \\
 &\quad =  \frac{i}{2}\int_{\Plat }     \frac{ (z-z_i)^{1/2} (z-z_j)^{1/2} (z-z_k)^{1/2}   (z-z_4)^{3/2}}{w^4 }    (z-z_4)\vert z-z_4\vert^2       \frac{dz \wedge \bar{dz} }{\vert w \vert^6} \\
 & \quad   =  \frac{i}{2}\int_{\Plat }    \frac{ (z-z_4) }{w }   \vert z-z_4\vert^2      \frac{dz \wedge \bar{dz} }{\vert w \vert^6}
 \,.
\end{aligned}
 $$ 
 
 Now (since $\Psi^\ast (w) =\zeta w$) it is easy to see that by the symmetry argument  that
 $$
 B_{11} = B_{22} =B_{33}=  B_{12}= B_{23}= 0 \,.
 $$
 For the coefficients $B_{13}= B_{31}$, we observe that by the defining equation we have
 $$
 \prod_{i=1}^{3} (z-z_i)^{1/2} (z-z_4)^{3/2} = w^3\,,
 $$
 hence we have
 $$
 \begin{aligned}
 &B_{13}=B_{31} \\
 &\quad=  \frac{i}{2}\int_{\Plat }    \frac{(z-z_i)^{1/2}  (z-z_j)^{1/2} (z-z_k)^{1/2}  (z-z_4)^{3/2}}{w^3}        
         \vert z-z_4\vert^2      \frac{dz \wedge \bar{dz} }{\vert w \vert^6}   \\
         &\quad  =  \frac{i}{2}\int_{\Plat }         
         \vert z-z_4\vert^2      \frac{dz \wedge \bar{dz} }{\vert w \vert^6}  = \text{area}(\Plat) =1 \,.
 \end{aligned}
 $$
 The calculation of the other blocks is entirely analogous.
 \end{proof} 
 
 Let $\mu_0$ denote the  $\SL(2, \R)$-invariant measure obtained from the canonical measure
 on the $\SL(2, \R)$ orbit of $\Plat$ and from the $\text{Aff}(\Plat)$-invariant representation $\rho_0$ by the suspension construction. 
 We can prove the following lower bound on the number of non-zero fiber Lyapunov exponents of the Teichm\"uller flow with respect
 to the measure $\mu_0$, in analogy with Theorem~\ref{thm:exp_EW} for the case of $\EW$.
 
 \begin{proposition} The Lyapunov spectrum of the lift of the Teichm\"uller flow with respect to $\mu_0$ has at least $6$ strictly positive
 $($and $6$ strictly negative$)$  fiber Lyapunov exponents (hence at most $6$ zero fiber exponents).  \end{proposition} 
 
 \begin{proof}
 The pull-back of the twisted cohomology bundle on this cover is an $\SL(2, \R)$-invariant bundle of the Hodge bundle which by Lemma
 \ref{lem:B_O} has a second fundamental form of rank $6$. By a theorem of Filip's (see \cite{Fil17}, Theorem 1.1), the restriction of the Kontsevich--Zorich cocycle to this subbundle has at least $6$ strictly positive exponents.  
  \end{proof} 
  
  \begin{remark} In the sequel, we will employ concrete calculations to show our main result asserting that the lift of the Teichm\"uller flow is in fact non-uniformly hyperbolic with respect to the measure $\mu_0$. However, contrary to the case of ${\EW}$, in the context of $\Plat$ we are unable to give a Hodge theoretic proof since the rank of the second fundamental form $B$ is not maximal and the above mentioned theorem by Filip only gives an inequality. Consequently, the measure $\mu_0$ for the Platypus provides a variation of Hodge structure coming from the moduli space of Riemann surfaces such that the inequality in Filip's theorem is strict. 
  \end{remark} 
  
  Given that the previous proposition falls short from establishing the desired Theorem \ref{main_thm:P}, we shall now perform explicit calculations in the homotopy of $\Plat$ in order to derive our main result. 

\subsection{Affine homeomorphisms of \texorpdfstring{$\Plat$}{}} Each element $h\in\mathbb{Z}/3\mathbb{Z}$ determines an automorphism of $\Plat$ by sending $Sq(i,\mu,\nu)$ to $Sq(i+h,\mu,\nu)$. As it turns out, the group $\textrm{Aut}(\Plat)$ of automorphisms of $\Plat$ is isomorphic to $\mathbb{Z}/3\mathbb{Z}$. Note that the elements of $\textrm{Aut}(\Plat)$ fix each point of the set $\Sigma=\{A_{1,1}, A_{0,1}, A_{1,0}\}$ of conical singularities. 

The group of affine homeomorphisms of $\Plat$ fits a short exact sequence 
$$1\to \textrm{Aut}(\Plat)\to \Aff(\Plat)\to \SL(2,\mathbb{Z})\to 1$$ 
and, in fact, $\Aff(\Plat)$ is isomorphic to the direct product of $\textrm{Aut}(\Plat)$ and $\SL(2,\mathbb{Z})$. We denote by $T$ and $S$ the elements of $\Aff(\Plat)$ fixing $A_0$ with linear parts 
$$\left(\begin{array}{cc}1 & 1 \\ 0 & 1\end{array}\right) \quad \textrm{and} \quad \left(\begin{array}{cc}1 & 0 \\ 1 & 1\end{array}\right).$$ 
The group $\Aff_{(1)}(\Plat)$ is spanned by $T$ and $S$ and consists exactly of those $f\in\Aff(\Plat)$ fixing each $A_i$. Denote by $\Gamma_{\rho_0}$ the subgroup of $\Aff(\Plat)$ of index $36$ generated by $T^2$ and $S^2$. 

\subsection{Action of \texorpdfstring{$\Gamma_{\rho_0}$ on $\pi_1(\Plat,\circ)$}{}} We have that $T(A_{1,1}) = A_{0,1}$, $T(A_{0,1}) = A_{1,1}$,  $T(A_{1,0})=A_{1,0}$, and 
$$T(\sigma_i')=\sigma_i, \quad T(\sigma_{i-1})=\sigma_i',$$
$$T(\zeta_{i-1})=\zeta_{i-1}\sigma_i', \quad T(\zeta_i')=\sigma_{i+1}\zeta_i', \quad T(\vec{A_i})=\vec{A_i}.$$  

Similarly, $S(A_{1,1}) = A_{1,0}$, $S(A_{1,0})=A_{1,1}$, $S(A_{0,1})=A_{0,1}$, and 
$$S(\sigma_i')=\zeta_{i-1}\sigma_i', \quad S(\sigma_{i+1})= \sigma_{i+1}\zeta_i',$$ 
$$S(\zeta_i')=\zeta_i, \quad S(\zeta_{i+1})=\zeta_i', \quad S(\vec{A_i})=\zeta_{i+1}'\zeta_{i-1}\sigma_i'\sigma_i.$$ 

Therefore, $T^2$ fixes $A_{1,1}=\circ$ and it acts on $\pi_1(\Plat,\circ)$. In terms of our preferred generators, one has 
$$
\begin{aligned}
\sigma_0'\sigma_0 &\to \sigma_1'\sigma_1 \\
\sigma_1'\sigma_0 &\to \sigma_2'\sigma_1 \\
\sigma_1'\sigma_1 &\to  \sigma_2'\sigma_2 \\
\sigma_2'\sigma_1 &\to  \sigma_0'\sigma_2 \\
\sigma_0'\sigma_2 &\to  \sigma_1'\sigma_0\\
\sigma_2'\sigma_2 &\to  \sigma_0'\sigma_0 \\ 
\zeta_0'\zeta_2 &\to (\sigma_2'\sigma_1)(\zeta_0'\zeta_2)(\sigma_0'\sigma_0)\\ 
\zeta_0' \vec{A_1} (\zeta_0')^{-1} &\to (\sigma_2'\sigma_1)(\zeta_0'\vec{A}_1(\zeta_0')^{-1})(\sigma_2'\sigma_1)^{-1}\\ 
\zeta_0' \vec{A_2} (\zeta_0')^{-1} &\to (\sigma_2'\sigma_1)(\zeta_0'\vec{A}_2(\zeta_0')^{-1})(\sigma_2'\sigma_1)^{-1}
\end{aligned} 
$$
Similarly, $S^2$ fixes $\circ$ and its action on $\pi_1(\Plat,\circ)$ in terms of our preferred generators is 
$$
\begin{aligned}
\sigma_0'\sigma_0 &\to (\zeta_1'\zeta_2)(\sigma_0'\sigma_0)(\zeta_2'\zeta_2)   \\
\sigma_1'\sigma_0 &\to (\zeta_2'\zeta_0)(\sigma_1'\sigma_0)(\zeta_2'\zeta_2)  \\
\sigma_1'\sigma_1 &\to (\zeta_2'\zeta_0)(\sigma_1'\sigma_1)(\zeta_0'\zeta_0) \\
\sigma_2'\sigma_1 &\to (\zeta_0'\zeta_1)(\sigma_2'\sigma_1)(\zeta_0'\zeta_0) \\
\sigma_0'\sigma_2 &\to (\zeta_1'\zeta_2)(\sigma_0'\sigma_2)(\zeta_1'\zeta_1) \\
\sigma_2'\sigma_2 &\to (\zeta_0'\zeta_1)(\sigma_2'\sigma_2)(\zeta_1'\zeta_1) \\ 
\zeta_0'\zeta_2 &\to \zeta_2'\zeta_1 \\ 
\zeta_0' \vec{A_1} (\zeta_0')^{-1} &\to (\zeta_2'\zeta_2)(\zeta_2'\vec{A}_1\zeta_1)(\zeta_0'(\zeta_2')^{-1}) \\ 
\zeta_0' \vec{A_2} (\zeta_0')^{-1} &\to (\zeta_2'\zeta_0)(\zeta_0'\vec{A}_2\zeta_2)(\zeta_1'(\zeta_2')^{-1})
\end{aligned} 
$$
Hence, the action of $S^2$ on $\pi_1(P,\circ)$ can be fully described in terms of our preferred generators once we compute $\zeta_1'\zeta_2$, $\zeta_2'\zeta_2$, $\zeta_2'\zeta_0$, $\zeta_0'\zeta_0$, $\zeta_0'\zeta_1$, $\zeta_1'\zeta_1$, $\zeta_2'\zeta_1$, $\zeta_2'(\zeta_0')^{-1}$, $\zeta_1'(\zeta_2')^{-1}$. Here, it suffices to calculate the first six paths because 
$$\zeta_2'(\zeta_0')^{-1} = (\zeta_2'\zeta_0)(\zeta_0'\zeta_0)^{-1}, \quad \zeta_1'(\zeta_2')^{-1} = (\zeta_1'\zeta_2)(\zeta_2'\zeta_2)^{-1}, \quad \zeta_2'\zeta_1 = (\zeta_2'(\zeta_0')^{-1})(\zeta_0'\zeta_1)$$ 
\begin{remark} In the sequel, the following relations will be useful: 
$$(\zeta_{i+1}'\zeta_{i-1})(\sigma_i'\sigma_i) = (\sigma_{i-1}'\sigma_{i+1})(\zeta_i'\zeta_i)$$ 
and 
$$\zeta_{i-1}^{-1}\vec{A}_i\zeta_i = \sigma_i'\sigma_i, \quad \zeta_{i+1}'\vec{A}_i(\zeta_i')^{-1} = \sigma_{i-1}'\sigma_{i+1}$$ 
for all $i\in\mathbb{Z}/3\mathbb{Z}$. 
\end{remark}
Since 
$$
\begin{aligned} 
(\zeta_{1}'\zeta_{2})(\sigma_0'\sigma_0) &= (\sigma_{2}'\sigma_{1})(\zeta_0'\zeta_0), \quad (\zeta_{2}'\zeta_{0})(\sigma_1'\sigma_1) = (\sigma_{0}'\sigma_{2})(\zeta_1'\zeta_1), \\ (\zeta_{0}'\zeta_{1})(\sigma_2'\sigma_2) &= (\sigma_{1}'\sigma_{0})(\zeta_2'\zeta_2),
\end{aligned}
$$ 
our task is reduced to compute $\zeta_{1}'\zeta_{2}$, $\zeta_{1}'\zeta_{1}$ and $\zeta_{0}'\zeta_{1}$. We can determine $\zeta_{1}'\zeta_{2}$ by noticing that 
\begin{eqnarray*} 
\zeta_{1}'\zeta_{2} &=& (\zeta_1'(\zeta_0')^{-1}))(\zeta_0'\zeta_2) = (\zeta_1'\vec{A}_0(\zeta_0')^{-1}))(\zeta_0'\vec{A}_0^{-1}(\zeta_0')^{-1}))(\zeta_0'\zeta_2) \\ 
&=& (\sigma_2'\sigma_1)(\zeta_0'\vec{A}_0^{-1}(\zeta_0')^{-1}))(\zeta_0'\zeta_2)
\end{eqnarray*}
and 
$$\zeta_0'\vec{A}_2\vec{A}_1\vec{A}_0(\zeta_0')^{-1} = (\sigma_1'\sigma_0)(\sigma_0'\sigma_2)(\sigma_2'\sigma_1)\,,$$ 
so that we have
$$
\zeta_{1}'\zeta_{2}  =(\sigma_0'\sigma_2)^{-1}  (\sigma_1'\sigma_0)^{-1} (\zeta_0'\vec{A}_2(\zeta_0')^{-1}))  
(\zeta_0'\vec{A}_1(\zeta_0')^{-1}))     (\zeta_0'\zeta_2)  = e^{-1} \cdot b^{-1} \cdot   i \cdot  h  \cdot g
$$
Similarly, $\zeta_1'\zeta_1$ is given by 
$$\zeta_1'\zeta_1 = (\zeta_1'\zeta_2)(\zeta_2^{-1}\zeta_1)$$ 
where $\zeta_1'\zeta_2$ was computed above and 
$$
\begin{aligned} 
\zeta_2^{-1}\zeta_1&=(\zeta_0'\zeta_2)^{-1}(\zeta_0'\vec{A}_2(\zeta_0')^{-1})(\zeta_0'\zeta_2)(\zeta_1^{-1}\vec{A}_2\zeta_2)^{-1} \\
&= (\zeta_0'\zeta_2)^{-1}(\zeta_0'\vec{A}_2(\zeta_0')^{-1})(\zeta_0'\zeta_2)(\sigma_2'\sigma_2)^{-1} = g^{-1}\cdot  i \cdot  g \cdot f^{-1}    \,,
\end{aligned}
$$
so that
$$
\zeta_1'\zeta_1 = (\zeta_1'\zeta_2)(\zeta_2^{-1}\zeta_1)= e^{-1} \cdot b^{-1} \cdot   i \cdot  h  \cdot  i \cdot  g \cdot f^{-1} 
$$
Finally, $\zeta_{0}'\zeta_{1}$ can be calculated by noticing that 
$$
\begin{aligned}
\zeta_{0}'\zeta_{1} &= (\zeta_0'\vec{A}_2\zeta_2)(\zeta_1^{-1}\vec{A}_2\zeta_2)^{-1} \\ &= (\zeta_0'\vec{A}_2(\zeta_0')^{-1})(\zeta_0'\zeta_2)(\sigma_2'\sigma_2)^{-1}  = i \cdot g \cdot f^{-1} \,.
\end{aligned}
$$ 

\subsection{The action of \texorpdfstring{$T^2$}{} and \texorpdfstring{$S^2$}{} on \texorpdfstring{$\rho_0$}{}} 
We recall that 
$$
 T^2 : \begin{cases}
a  &\to c  \\
b  &\to d \\
c  &\to  f \\
d &\to  e \\
e  &\to  b\\
f  &\to a  \\ 
g   &\to d\cdot g \cdot  a  \\ 
h &\to d \cdot h  \cdot d^{-1}\\ 
i &\to d \cdot i  \cdot d^{-1} 
\end{cases} 
$$
Hence, 
$$
T^2( \rho_0 ) : \begin{cases}   
a \to & \rho_0(c)  \equiv_3   - A^{-1}  B^{-1}   \equiv_3 A ( A B A^{-1} B^{-1}) A^{-1}   = A \rho_0(a) A^{-1}        \\  
b \to  &\rho_0 (d) \equiv_3  B A  B^{-1}  A^{-1}  =  A (A^{-1}  B A  B^{-1})  A^{-1} \equiv_3  A \rho_0(b) A^{-1}    \\ 
c \to &\rho_0(f) \equiv_3     - B^{-1} A^{-1}    \equiv_3  -A (A^{-1} B^{-1} ) A^{-1}  =  A \rho_0(c) A^{-1}        \\ 
d \to  &\rho_0(e) \equiv_3  -  B A    = A ( B A  B^{-1}    A^{-1}  ) A^{-1} =  A \rho_0(d) A^{-1}    \\  
e \to &\rho_0(b) \equiv_3   -A B  = - A (B A )A^{-1} = A \rho_0(e) A^{-1}      \\   
f \to &\rho_0(a)=  A  B  A^{-1} B^{-1} = -A  (B^{-1} A^{-1}  ) A^{-1} =  A \rho_0(f) A^{-1}    \\  
g  \to  &\rho_0(d)\rho_0(g) \rho_0(a)  =- (B A  B^{-1}    A^{-1} ) I_2 ( A B A^{-1} B^{-1}) \equiv_3 A \rho_0(g) A^{-1}   \\ 
h   \to  &\rho_0(d) \rho(h) \rho(d)^{-1}  \equiv_3     \rho_0(d) (-I_2) \rho(d)^{-1} =  -I_2 = A \rho_0(h) A^{-1}   \\   
i \to  &\rho_0(d)  \rho_0(i)  \rho_0(d)^{-1} =  \rho_0(d) (-I_2) \rho(d)^{-1} =  -I_2=  A \rho_0(i) A^{-1}       \end {cases} 
$$

Similarly, we have
$$
 S^2 : \begin{cases}
a  &\to  e^{-1} \cdot b^{-1}\cdot i \cdot h \cdot g \cdot a \cdot  b^{-1}  \cdot   i  \cdot g    \\
b  &\to   b^{-1}  \cdot  i  \cdot h  \cdot  i  \cdot g  \cdot  f^{-1}   \cdot  c^{-1}    \cdot   i  \cdot g    \\
c  &\to   b^{-1}  \cdot  i  \cdot h  \cdot  i  \cdot g  \cdot  f^{-1}  \cdot  d^{-1} \cdot e^{-1} \cdot b^{-1}\cdot i \cdot h \cdot g \cdot a  \\
d &\to   i  \cdot g  \cdot f^{-1}  \cdot e^{-1} \cdot b^{-1}\cdot i \cdot h \cdot g \cdot a    \\
e  &\to e^{-1} \cdot b^{-1}\cdot i \cdot h \cdot g   \cdot b^{-1}\cdot i \cdot h  \cdot i   \cdot   g \cdot f^{-1} \\
f  &\to  i  \cdot g  \cdot   e^{-1}  \cdot b^{-1}\cdot i \cdot h  \cdot i   \cdot   g \cdot f^{-1}   \\ 
g   &\to  b^{-1}  \cdot  i  \cdot h  \cdot  i  \cdot g  \cdot  f^{-1}  \cdot  c^{-1}  \cdot a^{-1}  \cdot  g^{-1}  \cdot   h^{-1}  \cdot  i^{-1}  \cdot  b  \cdot  e  \cdot  d  \\ & \quad   \cdot i  \cdot g \cdot f^{-1}   \\ 
h &\to b^{-1}  \cdot   i  \cdot g  \cdot   b^{-1}  \cdot  i  \cdot h  \cdot  i  \cdot g  \cdot  f^{-1}   \cdot  c^{-1} \cdot a^{-1}  \cdot  g^{-1}  \cdot  h^{-1}  \cdot  i^{-1} 
 \\ &\quad   \cdot  b  \cdot  e  \cdot  d  \cdot h \cdot  i \cdot g \cdot f^{-1} \cdot d^{-1} \cdot e^{-1} \cdot b^{-1}\cdot i \cdot h \cdot g \cdot a \cdot c \cdot f \cdot
 g^{-1}    \\ 
 & \quad  \cdot  i^{-1} \cdot h^{-1} \cdot  i^{-1} \cdot b
 \\ 
i &\to      b^{-1}  \cdot  i  \cdot h  \cdot  i  \cdot g  \cdot  f^{-1}   \cdot  c^{-1}  \cdot i \cdot g \cdot 
e^{-1} \cdot b^{-1}\cdot i \cdot h \cdot  i^{-1} \cdot  b  
\end{cases} 
$$
so that 

$$
 S^2(\rho_0) : \begin{cases}
a  &\to   - A^2 ( A  B^{-1}  A)  A^{-2} =  A^2 ( A  BA^{-1} B^{-1})  A^{-2}  = A^2 \rho_0(a) A^{-2}   \\ 
b  &\to  - A^2 (A B )A^{-2} = A^2 \rho_0(b) A^{-2}    \\ 
c  &\to   -A^2 (A^{-1} B^{-1}) A^{-2} = A^2 \rho_0(c) A^{-2}    \\
d &\to     -A^2 (A^{-1} B A^{-1})     A^{-2}  =  A^2 (BA B^{-1} A^{-1}) A^{-2}  = A^2 \rho_0(d) A^{-2}     \\
e  &\to      - A^2 (BA)  A^{-2}  =A^2 \rho_0(e) A^{-2}   \\  
f  &\to   - A^2 ( B^{-1}A^{-1}) A^{-2} = A^2 \rho_0(f) A^{-2}    \\ 
g   &\to   I_2 =A^2 \rho_0(g) A^{-2}   \\ 
  h &\to  -I_2 = A^2 \rho_0(h) A^{-2}     \\ 
i &\to   -I_2 = A^2 \rho_0(i) A^{-2} 
\end{cases} 
$$
hence $S^2(\rho_0) = A^2 \rho_0 A^{-2}$.

\subsection{Tangent space to \texorpdfstring{$\XX(\Plat,\SU(2))$}{} at \texorpdfstring{$\rho_0$}{}} 

The tangent space to $\XX(\Plat,\SU(2))$ at $\rho_0$ is the quotient of the space of $\Ad_{\rho_0}$-twisted cocycles by the space of $\Ad_{\rho_0}$-twisted coboundaries. 

In view of the relations $d=f e^{-1} a b^{-1} c$ and $ihgacf = bedhig$, any cocycle is determined by its values on $7$ elements of the set $\{a,b,c,d,e,f,g,h,i\}$ because 
$$u(d) = u(f) + \Ad_{\rho_0(fe^{-1})}(u(a)-u(e)) + \Ad_{\rho_0(fe^{-1}ab^{-1})}(u(c)-u(b))$$ 
and 
\begin{align*} 
&u(i)+\Ad_{\rho_0(i)}(u(h))+\Ad_{\rho_0(ih)}(u(g)) + \Ad_{\rho_0(ihg)}(u(a)) + \Ad_{\rho_0(ihga)}(u(c)) +\\ &\Ad_{\rho_0(ihgac)}(u(f)) = u(b) + \Ad_{\rho_0(b)}(u(e)) + \Ad_{\rho_0(be)}(u(d)) + \Ad_{\rho_0(bed)}(u(h)) +\\ &\Ad_{\rho_0(bedh)}(u(i)) + \Ad_{\rho_0(bedhi)}(u(g)),
\end{align*} 
that is, 
$$u(d) = u(f) + \Ad_{AB^{-1}A}(u(a)-u(e)) + \Ad_{AB}(u(c)-u(b))$$
and   $\Ad_{A  B^{-1} A}(u(c)) + \Ad_{AB}(u(f)) =$
$$u(b) + \Ad_{AB}(u(e)) + \Ad_{AB^{-1}A}(u(f))-u(e)+\Ad_{BA}(u(c)-u(b)). 
$$
In particular, any twisted cocycle of $\rho_0$ is determined by its values on $a, b, c, e, g, h, i$.

It follows that, since $\mathfrak{su}(2)$ is a $3$-dimensional real vector space spanned by the Pauli matrices $u_1$, $u_2$ and $u_3$, we can write down the $18\times 18$ matrix of a linear operator of the form $u\mapsto \Ad_M(u\circ\phi)$, $\phi\in\Gamma_{\rho_0}$, in terms of a fixed  \emph{ad-hoc} basis. More concretely, the space of coboundaries is spanned by $\partial_i(\gamma) = \Ad_{\rho_0(\gamma)}u_i-u_i$, $i=1, 2, 3$, that is, 
$$\partial_1(a) = 0, \partial_1(b) = -2u_1, \partial_1(c) = -2u_1, \partial_1(d) = 0, \partial_1(e) = -2u_1, \partial_1(f) = -2u_1,$$ 
$$\partial_2(a) = -2u_2, \partial_2(b) = -2u_2, \partial_2(c) = 0, \partial_2(d) = -2u_2, \partial_2(e) = 0, \partial_2(f) = -2u_2,$$
$$\partial_3(a) = -2u_3, \partial_3(b) = 0, \partial_3(c) = -2u_3, \partial_3(d) = -2u_3, \partial_3(e) = -2u_3, \partial_3(f) = 0,$$ 
and $\partial_k(g)=\partial_k(h)=\partial_k(i)=0$ for $k=1,2,3$. 

This suggests to freely choose the values at $g,h,i$, to restrict the value at $e$ to $\mathbb{R}u_2$ (after correcting by $\partial_1$ and $\partial_3$), to restrict the value at $a$ to $\mathbb{R}u_1\oplus\mathbb{R}u_3$ (after correcting by $\partial_2$), and freely choose the values at $b$ and $c$. For this reason, we introduce the following $18$ cocycles: 
$$e_1(a) = u_1, e_1(b) = 0, e_1(c) = 0, e_1(d) = u_1, e_1(e) = 0, e_1(f) = 0$$
$$e_2(a) = u_3, e_2(b) = 0, e_2(c) = 0, e_2(d) = -u_3, e_2(e) = 0, e_2(f) = 0$$
$$e_3(a) = 0, e_3(b) = u_1, e_3(c) = 0, e_3(d) = 0, e_3(e) = 0, e_3(f) = -u_1$$
$$e_4(a) = 0, e_4(b) = u_2, e_4(c) = 0, e_4(d) = u_2, e_4(e) = 0, e_4(f) = 0$$
$$e_5(a) = 0, e_5(b) = u_3, e_5(c) = 0, e_5(d) = 0, e_5(e) = 0, e_5(f) = u_3$$
$$e_6(a) = 0, e_6(b) = u_1, e_6(c) = u_1, e_6(d) = 0, e_6(e) = 0, e_6(f) = 0$$
$$e_7(a) = 0, e_7(b) = 0, e_7(c) = u_2, e_7(d) = 0, e_7(e) = u_2, e_7(f) = 0$$
$$e_8(a) = 0, e_8(b) = 0, e_8(c) = u_3, e_8(d) = u_3, e_8(e) = 0, e_8(f) = 0$$
$$e_9(a) = 0, e_9(b) = 0, e_9(c) = u_2, e_9(d) = u_2, e_9(e) = u_2, e_9(f) = u_2$$
$$e_{10}(g) = u_1, e_{10}(h) = 0, e_{10}(i) = 0$$
$$e_{11}(g) = u_2, e_{11}(h) = 0, e_{11}(i) = 0$$
$$e_{12}(g) = u_3, e_{12}(h) = 0, e_{12}(i) = 0$$ 
$$e_{13}(g) = 0, e_{13}(h) = u_1, e_{13}(i) = 0$$
$$e_{14}(g) = 0, e_{14}(h) = u_2, e_{14}(i) = 0$$
$$e_{15}(g) = 0, e_{15}(h) = u_3, e_{15}(i) = 0$$
$$e_{16}(g) = 0, e_{16}(h) = 0, e_{16}(i) = u_1$$
$$e_{17}(g) = 0, e_{17}(h) = 0, e_{17}(i) = u_2$$
$$e_{18}(g) = 0, e_{18}(h) = 0, e_{18}(i) = u_3$$ 

\begin{remark} It is implicit here that $e_k(g)=e_k(h)=e_k(i)=0$ for $1\leq k\leq 9$ and $e_l(a)=e_l(b)=e_l(c)=e_l(d)=e_l(e)=e_l(f)=0$ for $10\leq l\leq 18$. 
\end{remark}

\subsection{Derivative of the action of \texorpdfstring{$T^2$}{} at \texorpdfstring{$\rho_0$}{}} 
Since $T^2$ sends $\rho_0$ to $A\rho_0A^{-1}$, its derivative at $[\rho_0]$ is the linear operator 
$$u\mapsto \Ad_{A^{-1}}(u\circ T^2).$$
Hence, our task is to write $\Ad_{A^{-1}}(e_i\circ T^2)$, $1\leq i\leq 18$, as a linear combination of the cocycles $\{e_j\}_{j=1}^{18}$ modulo the coboundaries $\partial_k$, $1\leq k\leq 3$. 

Note that 
$$u\circ T^2(a) = u(c), \quad u\circ T^2(b) = u(d), \quad u\circ T^2(c) = u(f),$$ 
$$u\circ T^2(d) = u(e), \quad u\circ T^2(e) = u(b), \quad u\circ T^2(f) = u(a),$$ 
and 
\begin{eqnarray*} 
u\circ T^2(g) &=& u(dga) = u(d) + \Ad_{\rho_0(d)}(u(g)) + \Ad_{\rho_0(dg)}(u(a)) \\ 
&=& u(d) + \Ad_{BAB^{-1}A^{-1}}(u(g) + u(a)) \\ 
&=& u(d) + \Ad_{AB^{-1}A}(u(g) + u(a)),
\end{eqnarray*}  
\begin{eqnarray*} 
u\circ T^2(h) &=& u(dhd^{-1}) = u(d) + \Ad_{\rho_0(d)}(u(h))-\Ad_{\rho_0(dhd^{-1})}(u(d)) \\ 
&=& \Ad_{BAB^{-1}A^{-1}}(u(h)) \\ 
&=& \Ad_{AB^{-1}A}(u(h)), 
\end{eqnarray*}  
\begin{eqnarray*} 
u\circ T^2(i) &=& u(did^{-1}) = u(d) + \Ad_{\rho_0(d)}(u(i))-\Ad_{\rho_0(did^{-1})}(u(d)) \\ 
&=& \Ad_{BAB^{-1}A^{-1}}(u(i)) \\ 
&=& \Ad_{AB^{-1}A}(u(i)).
\end{eqnarray*} 

Therefore, the matrix $d_{[\rho_0]}(T^2)^*(u) = \Ad_{A^{-1}}(u\circ T^2)$ is given by the formulas: 
\begin{align*}
&d_{[\rho_0]}(T^2)^*(e_{10}) = -e_{12}, \quad d_{[\rho_0]}(T^2)^*(e_{11}) = -e_{10}, \quad d_{[\rho_0]}(T^2)^*(e_{12}) = e_{11},\\
&d_{[\rho_0]}(T^2)^*(e_{13}) = -e_{15}, \quad d_{[\rho_0]}(T^2)^*(e_{14}) = -e_{13}, \quad d_{[\rho_0]}(T^2)^*(e_{15}) = e_{14},\\
&d_{[\rho_0]}(T^2)^*(e_{16}) = -e_{18}, \quad d_{[\rho_0]}(T^2)^*(e_{17}) = -e_{16}, \quad d_{[\rho_0]}(T^2)^*(e_{18}) = e_{17},\\
&d_{[\rho_0]}(T^2)^*(e_1) = -e_5-2e_{12}, \quad d_{[\rho_0]}(T^2)^*(e_2) = e_4+e_7-e_9+2e_{11},\\ &d_{[\rho_0]}(T^2)^*(e_3) = e_2 + 2e_{8} +\frac{1}{2}\partial_3\quad d_{[\rho_0]}(T^2)^*(e_4) = e_3 - e_6 + e_{10} - \frac{1}{2} \partial_{1},\\ &d_{[\rho_0]}(T^2)^*(e_5) = -e_7, \quad  d_{[\rho_0]}(T^2)^*(e_6) = e_8 + \frac{1}{2} \partial_3,\\ &d_{[\rho_0]}(T^2)^*(e_7) = e_1, \quad d_{[\rho_0]}(T^2)^*(e_8) = -e_{7} + e_{9} - e_{11} + \frac{1}{2} \partial_2, \\ &d_{[\rho_0]}(T^2)^*(e_9) = e_1 + e_6 + e_{10}.\end{align*} 

In other terms, we have $ d_{[\rho_0]}(T^2)^* =$ $$\left(\begin{array}{cccccccccccccccccc} 
0 & 0 & 0 & 0 & 0 & 0 & 1 & 0 & 1 & 0 & 0 & 0 & 0 & 0 & 0 & 0 & 0 & 0 \\ 
0 & 0 & 1 & 0 & 0 & 0 & 0 & 0 & 0 & 0 & 0 & 0 & 0 & 0 & 0 & 0 & 0 & 0 \\ 
0 & 0 & 0 & 1 & 0 & 0 & 0 & 0 & 0 & 0 & 0 & 0 & 0 & 0 & 0 & 0 & 0 & 0 \\ 
0 & 1 & 0 & 0 & 0 & 0 & 0 & 0 & 0 & 0 & 0 & 0 & 0 & 0 & 0 & 0 & 0 & 0 \\ 
-1 & 0 & 0 & 0 & 0 & 0 & 0 & 0 & 0 & 0 & 0 & 0 & 0 & 0 & 0 & 0 & 0 & 0 \\ 
0 & 0 & 0 & -1 & 0 & 0 & 0 & 0 & 1 & 0 & 0 & 0 & 0 & 0 & 0 & 0 & 0 & 0 \\ 
0 & 1 & 0 & 0 & -1 & 0 & 0 & -1 & 0 & 0 & 0 & 0 & 0 & 0 & 0 & 0 & 0 & 0 \\ 
0 & 0 & 2 & 0 & 0 & 1 & 0 & 0 & 0 & 0 & 0 & 0 & 0 & 0 & 0 & 0 & 0 & 0 \\ 
0 & -1 & 0 & 0 & 0 & 0 & 0 & 1 & 0 & 0 & 0 & 0 & 0 & 0 & 0 & 0 & 0 & 0 \\ 
0 & 0 & 0 & 1 & 0 & 0 & 0 & 0 & 1 & 0 & -1 & 0 & 0 & 0 & 0 & 0 & 0 & 0 \\ 
0 & 2 & 0 & 0 & 0 & 0 & 0 & -1 & 0 & 0 & 0 & 1 & 0 & 0 & 0 & 0 & 0 & 0 \\ 
-2 & 0 & 0 & 0 & 0 & 0 & 0 & 0 & 0 & -1 & 0 & 0 & 0 & 0 & 0 & 0 & 0 & 0 \\
0 & 0 & 0 & 0 & 0 & 0 & 0 & 0 & 0 & 0 & 0 & 0 & 0 & -1 & 0 & 0 & 0 & 0 \\
0 & 0 & 0 & 0 & 0 & 0 & 0 & 0 & 0 & 0 & 0 & 0 & 0 & 0 & 1 & 0 & 0 & 0 \\ 
0 & 0 & 0 & 0 & 0 & 0 & 0 & 0 & 0 & 0 & 0 & 0 & -1 & 0 & 0 & 0 & 0 & 0 \\ 
0 & 0 & 0 & 0 & 0 & 0 & 0 & 0 & 0 & 0 & 0 & 0 & 0 & 0 & 0 & 0 & -1 & 0 \\ 
0 & 0 & 0 & 0 & 0 & 0 & 0 & 0 & 0 & 0 & 0 & 0 & 0 & 0 & 0 & 0 & 0 & 1 \\ 
0 & 0 & 0 & 0 & 0 & 0 & 0 & 0 & 0 & 0 & 0 & 0 & 0 & 0 & 0 & -1 & 0 & 0 
  \end{array}\right)$$
  
\begin{remark} Note that $d_{[\rho_0]}(T^6)^* = \textrm{Id}+X$, where $X$ has rank three and $X^2 = 0$.  
\end{remark}

\subsection{Derivative of the action of \texorpdfstring{$S^2$}{} at \texorpdfstring{$\rho_0$}{}} 
Since $S^2$ sends $\rho_0$ to $A^2\rho_0A^{-2}$, its derivative at $[\rho_0]$ is the linear operator 
$$u\mapsto \Ad_{A^{-2}}(u\circ S^2).$$

By proceeding as in the previous section, one derives that $d_{[\rho_0]}(S^2)^*$ is 
$$ \left(\begin{array}{cccccccccccccccccc} 
0 & 1 & 0 & 0 & 0 & 0 & 0 & 0 & 0 & 0 & 0 & 0 & 0 & 0 & 1 & 0 & 0 & 0 \\ 
0 & 0 & 0 & -2 & 0 & 0 & 0 & 0 & 1 & 0 & 2 & 0 & 0 & 1 & 0 & 0 & 3 & 0 \\ 
0 & 0 & 0 & 0 & -1 & 0 & 0 & 0 & 0 & 0 & 0 & 2 & 0 & 0 & 1 & 0 & 0 & 2 \\ 
-1 & 0 & 2 & 0 & 0 & 2 & 0 & 0 & 0 & -2 & 0 & 0 & -2 & 0 & 0 & -3 & 0 & 0 \\ 
0 & 0 & 0 & -1 & 0 & 0 & 1 & 0 & 2 & 0 & 0 & 0 & 0 & 1 & 0 & 0 & 1 & 0 \\ 
0 & 0 & 0 & 0 & 6 & 0 & 0 & -1 & 0 & 0 & 0 & -4 & 0 & 0 & -4 & 0 & 0 & -6 \\ 
1 & 0 & 3 & 0 & 0 & 1 & 0 & 0 & 0 & -2 & 0 & 0 & 0 & 0 & 0 & -4 & 0 & 0 \\ 
0& 0& 0& -1& 0& 0& -2& 0& -1& 0& 0& 0& 0& 0& 0& 0& 2& 0 \\ 
-1& 0& -2& 0& 0& -1& 0& 0& 0& 2& 0& 0& 0& 0& 0& 3& 0& 0 \\ 
0& 0& 0& 0& 2& 0& 0& 0& 0& 0& 0& -1& 0& 0& 0& 0& 0& -2 \\ 
0& 0& 2& 0& 0& 1& 0& 0& 0& -1& 0& 0& 0& 0& 0& -2& 0& 0 \\ 
0& 0& 0& -1& 0& 0& 0& 0& 1& 0& 1& 0& 0& 0& 0& 0& 2& 0 \\ 
0& 0& 0& 0& 2& 0& 0& 0& 0& 0& 0& -2& 0& 0& -1& 0& 0& -2 \\ 
0& 0& -2& 0& 0& -1& 0& 0& 0& 0& 0& 0& 1& 0& 0& 2& 0& 0 \\ 
0& 0& 0& 1& 0& 0& 0& 0& -1& 0& 0& 0& 0& -1& 0& 0& -2& 0 \\ 
0& 0& 0& 0& 2& 0& 0& -1& 0& 0& 0& 0& 0& 0& -2& 0& 0& -1 \\ 
0& 0& 2& 0& 0& 2& 0& 0& 0& -2& 0& 0& -2& 0& 0& -3& 0& 0 \\ 
0& 0& 0& -1& 0& 0& 2& 0& 3& 0& 0& 0& 0& 2& 0& 0& 1& 0 
  \end{array}\right)$$
  
\begin{remark} Note that $d_{[\rho_0]}(S^6)^* = \textrm{Id}+Y$, where $Y$ has rank three and $Y^2 = 0$.  
\end{remark}

\subsection{Blocks and uniform expansion at \texorpdfstring{$\rho_0$}{} of the \texorpdfstring{$\Gamma_{\rho_0}$}{}-action} 
Similarly to the case of $\EW$, one has that the $\Gamma_{\rho_0}$ action respects the decomposition in Lemma \ref{lem:twist_hol_diff_O}. 
In terms of the vectors $e_k$, $k=1,\dots, 18$, the three blocks described above are  
$$U_1=\bigoplus\limits_{k\in\{1,3, 6, 10, 13, 16\}}\mathbb{R} e_k, \quad U_2=\bigoplus\limits_{k\in\{4, 7, 9, 11, 14, 17\}}\mathbb{R} e_k, \quad U_3=\bigoplus\limits_{k\in\{2, 5, 8, 12, 15, 18\}}\mathbb{R} e_k.$$ 
From our formulas for $d_{[\rho_0]}(T^2)^*$ and $d_{[\rho_0]}(S^2)^*$, it is not hard to check that these blocks are cyclically permuted: indeed, $d_{[\rho_0]}(T^2)^*$ maps $U_1\to U_3\to U_2\to U_1$ and $d_{[\rho_0]}(S^2)^*$ maps $U_1\to U_2\to U_3\to U_1$. Consequently, the Lyapunov spectra of $U_1$, $U_2$ and $U_3$ are the same and, thus, we can restrict (from now on) our attention to the finite index subgroup of $\langle d_{[\rho_0]}(T^2)^*, d_{[\rho_0]}(S^2)^* \rangle$ stabilizing $U_1$.

Denote by $g_1$ and $h_1$ the restrictions of $d_{[\rho_0]}(T^6 S^4 T^4 S^6)^*$ and $d_{[\rho_0]}(T^6 S^4 T^4 S^6 T^6 S^6)^*$ to $U_1$. We will show that $\langle g_1, h_1 \rangle$ is Zariski dense in $\mathsf{Sp}(6,\mathbb{R})$. For this sake, we use a Zariski density criterion based on the works \cite{PR} and \cite{MMY} as explained in \cite{KM}.

More concretely, we begin by noticing that $g_1$ and $h_1$ do not commute. 

Also, the characteristic polynomial of $g_1$ is 
$$P(x) = x^6-2x^5-125x^4-404x^3-125x^2-2x+1$$ 
and the characteristic polynomial of $h_1$ is 
$$\widetilde{P}(x) = x^6-30x^5-781x^4-2540x^3-781x^2-30x+1.$$ 
Note that the roots of $P$ are $\approx 13.51, -7.69, -3.47, -0.28, -0.13, 0.07$ and the roots of $\widetilde{P}$ are $\approx 47.55, -13.72, -3.49, -0.28, -0.07, 0.02$. In particular, the (hyperbolic) matrix $h_1$ has infinite order. 

Furthermore, the matrix $g_1$ is Galois-pinching in the sense of \cite{MMY}, i.e., $P$ has simple real roots of distinct moduli and its Galois group is the largest possible (namely, the hyperoctahedral group $S_3\ltimes(\mathbb{Z}/2\mathbb{Z})^3$). Indeed, consider the polynomial $Q(y) = y^3-8y^2-108y-160$ with $x^3 Q(x+\frac{1}{x}+2) = P(x)$. Its discriminant $\textrm{Disc}(Q)=2^8\cdot 11\cdot 809$ is a positive integer which is not a square and, hence, $Q$ has Galois group $S_3$. In this context, if the Galois group of $P$ is not maximal, then it should coincide with one of the four groups $S_3$, $H_{3,1}$, $H_{3,2}$ or $H_{3,3}$ described in \cite[Proposition 2.4.2]{KM}. Since $\Delta_{3,1} = Q(0)Q(4) = 2^9\cdot 5\cdot 41$ and $\Delta_{3,2}=\textrm{Disc}(Q)\cdot\Delta_{3,1}$ are not squares, the possibilities $H_{3,1}$ and $H_{3,2}$ are ruled out (cf. \cite[Lemma 2.4.3 and Remark 2.4.4]{KM}). On the other hand, $P(x)\equiv (x+14)(x+19)(x^4+18x^3+22x^2+18x+1)$ mod $53$, so that Dedekind's theorem ensure that some element of the Galois group of $P$ has a cycle decomposition $(4,1,1)$ and this is incompatible with the possibilities $S_3$ and $H_{3,3}$ (cf. \cite[Appendix A]{KM}). 

By \cite[Theorem 9.10]{PR}, the facts that $g_1$ is Galois-pinching, $h_1$ has infinite order and they do not commute imply that the Zariski closure of $\langle g_1, h_1\rangle$ is $\mathsf{Sp}(6,\mathbb{R})$ or a product of $\SL(2)$. However, this last possibility can not occur in our setting because $h_1$ doesn't preserve the $g_1$-invariant plane associated to the eigenvalues $\theta_1\approx 13.51$ and $\theta_1^{-1}$ (actually, if $g_1(v_1)=\theta_1 v_1$ and $g_1(v_2)=\theta_1^{-1} v_2$, then the matrix with columns $v_1, v_2, h_1(v_1), h_1(v_2)$ has rank four). In other terms, $\langle g_1,h_1\rangle$ is Zariski dense in $\mathsf{Sp}(6,\mathbb{R})$. 

In this way, we get the uniform expansion of the infinitesimal action of $\Gamma_{\rho_0}$ at $\rho_0$ (compare with \S\ref{ss.expansion-EW}) and the non-uniform hyperbolicity statement in Theorem \ref{main_thm:P}. 

\begin{remark}
It can be checked that $h_1$ is not Galois-pinching: in fact, the Galois group of $\widetilde{P}$ is $H_{3,1}$. 
\end{remark}

\begin{remark}
The above calculations were performed in {\it Mathematica}.
\end{remark}

\appendix

\section{The Fermat quartic curve and the Eierlegende Wollmilchsau }
\label{a.Fermat}
We show that the {\em Fermat Quartic Curve\/} $\Fermat$
defined by $A^4 + B^4 + C^4 = 0$ in homogeneous coordinates
is equivalent to the Riemann surface $\EW$ underlying
the {\em Eierlegende Wollmilchsau\/} translation surface
discussed in Herrlich--Schmith\"usen~\cite{HS} and
papers by Forni and Matheus~\cite{Intro} and many others.

Let 
\[
e^{\pi/8 i} 
= \frac{\sqrt{2 + \sqrt{2}}}{2} 
+ i \frac{\sqrt{2 - \sqrt{2}}}{2}   
\]
be a primitive $16$-th root of unity. 
Then 
we use $\sqrt[4]{2}$ to find a fourth root of $- 8 i$:

\[
\sqrt[4]{8} =  2^{3/4} \\
\sqrt[4]{-i} = \exp(i 3 \pi/8)
\]

The linear coordinate change
\begin{equation}\label{eq:LinearCoordinateChange}
U :=  (-i X + Z)/ \sqrt[4]{-8 i}, \qquad
V :=  (X -i  Z)/ \sqrt[4]{8 i}. \end{equation}
relates the quartic forms 
\[
U^4 + V^4 = - XZ (X^2 - Z^2). \]
Let $Y$ be another variable. 
Adding $Y^4$, 
\[
Y^4 + U^4 + V^4 = Y^4 - XZ (X^2 - Z^2),\]
and \eqref{eq:LinearCoordinateChange} identifies the {\em Fermat quartic curve\/} 
\[
\Fermat :=
\Bigg\{ 
\bbmatrix U \\ V \\ Y \endbbmatrix \in \P^3 \ \Bigg|\ 
U^4 + V^4 + Y^4 = 0 \Bigg\}
\]
with the algebraic curve $\EW$ 
(underlying the  {\em Eierlegende Wollmilchsau\/} square-tiled translation surface
discussed in \cite{Intro}, \S 7.1 and \cite{HS}):
\[
\EW := 
\Bigg\{ 
\bbmatrix X \\ Y \\ Z \endbbmatrix \in \P^3 \ \Bigg|\  
Y^4  =  XZ (X^2 - Z^2) \Bigg\}
\]
which under the affine chart
\[
(x,y) \longmapsto \bbmatrix x \\ y \\ 1 \endbbmatrix \]
identifies with the closure in $\P^2$ of the  affine plane curve defined by
\[
y^4 = x (x-1) (x-\lambda), \]
where the parameter $\lambda = -1$.

$\Fermat$ admits 96 automorphisms:
it is obviously symmetric in the variables, and is invariant
under the {\em diagonal transformations\/}
\[
\bbmatrix X \\ Y \\ Z \endbbmatrix \xmapsto{~\Delta^{(\xi,\eta,\zeta)}~}
\bbmatrix \xi X \\ \eta Y \\ \zeta Z \endbbmatrix, \]
where $\xi,\eta,\zeta$ are quartic roots of unity
$\{\pm 1,\pm i\} \cong \Z/4$. 
The group defined by $\xi=\eta=\zeta$ acts trivially,
so we can normalize these elements by requiring that $\eta=1$.
The corresponding group of collineations is the quotient, 
isomorphic to $\Z/4\oplus \Z/4$, 
and we and simply write 
\[
\Delta_{(\xi,\zeta)} =   \Delta^{(\xi,1,\zeta)} \] in the above notation.
Permuting the three coordinates gives an action of the symmetric
group $\mathfrak{S}_3$. 
The group $\Lambda$ generated by the 
$\Delta_{\xi,\zeta}$ and $\mathfrak{S}_3$ is

\begin{proposition} 
$\Lambda$ is the full group of automorphism of $\Fermat$.
\end{proposition}
\begin{proof} (thanks to Patrick Brosnan)
Since $\Fermat$ is a smooth projective quartic, it has genus $g = 3$. 
Let $d := \#(\Aut(\Fermat)/\Lambda)$; we show that $d=1$.

If $d \ge 2$, 
Hurwitz's theorem implies 
\[
192 = 2 \#(\Lambda) \le d \#(\Lambda) = \#\big(\Aut(\Fermat)\big) \le 84 (g-1) = 168.
\]
This contradiction proves $d = 1$, and $\Lambda = \Aut(\Fermat)$ as desired.
\end{proof}

Now we introduce the abelian differential 
\[
\Phi :=  Y^{-2} (Z\, dX - X\, dZ) \]
which extends $y^{-2} dx$ to obtain the 
{\em Eierlegende Wollmilchsau\/} square-tiled translation surface $(\EW,\Phi)$.
Under the canonical embedding of $\EW \hookrightarrow \P^2$, 
the holomorphic differential $\Phi$ corresponds to the hyperplane $ Y = 0$.
Since the automorphism group $\Aut(\EW)$ fixes this hyperplane,
it can be represented by direct sums of $2\times 2$ matrices (acting linearly on $(U,V)$) with
$1$ (in the $Y$ direction).
Herrlich and Schmith\"usen~\cite{HS} compute the automorphism group  $\Aut(\EW)$
and show that it has order $16$,
a double extension of the quaternion group 
\[
Q = \{\pm \bOne, \pm \bi, \pm \bj ,\pm \bk \}. \]
Explicitly, it has an extra generator $c$ of order four, 
which centralizes $Q$.

\section{Special \texorpdfstring{$\mathsf{SU}(2)$}{}- characters of \texorpdfstring{$\pi_1(\Plat,\circ)$}{}}\label{a.SU2O}

Let $\biTet:=\langle a,b\ |\ a^3=b^3=(ab)^2\rangle$ be the {\it binary tetrahedral group} of order 24, and $\biOct:=\langle a,b\ |\ a^3=b^4=(ab)^2\rangle$ be the {\it binary octahedral group} of order 48.  The group $\biTet$ is naturally a group of index 2 in $\biOct$.

Letting 
$1\mapsto \left(
\begin{array}{cc}
 1 & 0 \\
 0 & 1 \\
\end{array}
\right),\ i\mapsto\left(
\begin{array}{cc}
 i & 0 \\
 0 & -i \\
\end{array}
\right),\ j\mapsto\left(
\begin{array}{cc}
 0 & 1 \\
 -1 & 0 \\
\end{array}
\right),$ and $k\mapsto
\left(
\begin{array}{cc}
 0 & i \\
 i & 0 \\
\end{array}
\right)$ gives a copy of the quaternions $\Qu$ in $\SUtwo$.  Extending this by the order 3 element $$\frac{1}{2}\left(
\begin{array}{cc}
 -1-i & -1+i \\
 1+i & -1+i \\
\end{array}
\right)$$ gives a copy of $\biTet$ in $\SUtwo$ and then extending $\biTet$ by the order 8
\ element $\left(
\begin{array}{cc}
 \frac{1+i}{\sqrt{2}} & 0 \\
 0 & \frac{1-i}{\sqrt{2}} \\
\end{array}
\right)$ gives $\biOct$ in $\SUtwo$.

Precisely, let $t:=\frac{1}{2}\left(
\begin{array}{cc}
 1+i & 1-i \\
 -1-i & 1-i \\
\end{array}
\right)$ and $s:=\frac{1}{2}\left(
\begin{array}{cc}
 1+i & 1+i \\
 -1+i & 1-i \\
\end{array}
\right)$. Define $\biTet\to \SUtwo$ by $a\mapsto t,$ and $b\mapsto s.$  This homomorphism embeds $\biTet$ into $\SUtwo$ since $t^3=s^3=(ts)^2=-I$ where $I$ is the $2\times 2$ identity matrix.  We now identify $\biTet$ with its image in $\SUtwo$ and show its normalizer in $\SUtwo$ is isomorphic to $\biOct$.

\begin{lemma}\label{lem:normalizer}
Let $G=\SUtwo$. Then the normalizer $N_G(\biTet)=\biOct=N_G(\Qu)$. Let $\Gamma$ be a finitely generated group, and let $\rho_1,\rho_2\in \mathsf{Hom}(\Gamma, \biTet)\subset \mathsf{Hom}(\Gamma, G)$ be onto $\Qu\subset \biTet$. Then: $\rho_1$ is $G$-conjugate to $\rho_2$ if and only if $\rho_1$ is $\biOct$-conjugate to $\rho_2$.
\end{lemma}

\begin{proof}
Let $Z_G(\Qu)$ be the centralizer of the quaternions $\Qu$ in $G=\SUtwo$ and let $Z(G)=\{\pm I\}$ be the center of $G$.  We first show that $Z_G(\Qu)=Z(G)$; which implies $Z_G(\biTet)=Z(G)$ too.  Indeed, requiring that a generic matrix $\left(
\begin{array}{cc}
 a & b \\
 c & d \\
\end{array}
\right)$ commutes with $\left(
\begin{array}{cc}
 0 & 1 \\
 -1 & 0 \\
\end{array}
\right)$ implies that $a=d$ and $b=-c$.  Requiring it further commutes with $\left(
\begin{array}{cc}
 i & 0 \\
 0 & -i \\
\end{array}
\right)$
implies that $b=0$.  Thus, any such matrix that commutes with $\Qu\subset \biTet$ must be diagonal.  But the only diagonal matrices in $\SUtwo$ are $\pm I$.

Let $\varphi_{\biTet}:N_G(\biTet)\to \mathrm{Aut}(\biTet)$ by $\varphi_{\biTet}(g)=\iota_g$ where $\iota_g(x)=gxg^{-1}$.  The kernel of $\varphi_{\biTet}$ is $Z_G(\biTet)$ by definition.  

As the binary tetrahedral group is a 2-1 cover of the group of rigid motions of a tetrahedron $ABCD$, one can deduce that $\mathrm{Aut}(\biTet)\cong S_4$ where $S_4$ is the permutations of $\{A,B,C,D\}$.  The inner automorphisms, 
$\mathrm{Inn}(\biTet)\cong \biTet/\{\pm I\}$, is isomorphic to the alternating group $A_4$.

By direct computation, we can see that $\left(
\begin{array}{cc}
 \frac{1+i}{\sqrt{2}} & 0 \\
 0 & \frac{1-i}{\sqrt{2}} \\
\end{array}
\right)$ normalizes $\biTet$ and has order 8.  But under $\varphi$ this gives a $4$-cycle.  Thus, we have that $\varphi(N_G(\biTet))$ is a group of $S_4$ containing $A_4$ and a $4$-cycle.  Thus the image is all of $S_4$.

Therefore, the order of $N_G(\biTet)$ is $48$ and thus must be $\biOct$, as required.

Likewise, we have $\varphi_\Qu:N_G(\Qu)\to \mathrm{Aut}(\Qu)$ given by $\varphi_\Qu(g)=\iota_g$, as before.  Again, we have $\mathrm{Aut}(\Qu)\cong S_4$ and the kernel of $\varphi_\Qu$ is $Z(G)=\{\pm I\}$.  So as before we know $|N_G(\Qu)|\leq 48$.  We guess that $\biOct$ might be normalizer again, and check explicitly to see if the 48 elements in $\biOct\leq \SUtwo$ fix the group $\Qu\leq \SUtwo$ under conjugation.  Finding that they do we explicitly see that $\biOct\subset N_G(\Qu)$.  But given that $|N_G(\Qu)|\leq 48$ we conclude that $N_G(\Qu)=\biOct$ as required.

Now, since $\rho_1$ and $\rho_1$ are onto $\Qu$, they are conjugate in $G$ if and only if they are conjugate in $\biOct$ since we have shown $N_G(\Qu)=\biOct$.
\end{proof}

\begin{remark} 
If there exists a surjection $\Gamma\to F_2$ where $F_2$ is a free group of rank 2 (and there does exists such for hyperbolic surface groups), then any irreducible representation $\rho:\Gamma\to \Qu\subset \SUtwo$ will be onto $\Qu$ since any two non-central elements generate $\Qu$. 
\end{remark}

We know from earlier that the fundamental group of the Platypus $\Plat$ admits a presentation with 9 generators and two relations:

$$\pi_1(\Plat)=\langle a,b,c,d,e,f,g,h,i\ |\ d^{-1}fe^{-1}ab^{-1}c,\ ihgacfg^{-1}i^{-1}h^{-1}d^{-1}e^{-1}b^{-1}\rangle.$$

Let $\Gamma:=\langle p_1,p_2,p_3,p_4\ |\ p_1p_2p_3p_4,\ p_1^6,\ p_2^6,\ p_3^6,\ p_4^2\rangle$ be the orbifold fundamental group of the 4-holed sphere with prescribed torsion.

Described earlier, there is a homomorphism $\varphi:\pi_1(\Plat)\to \Gamma,$ defined by:
$$
\begin{cases}
a  &\mapsto p_1 p_2^{-1}    \\
b  &\mapsto   p_1^{-1}  p_2    \\ 
c  &\mapsto  p_1^{-1} p_2^{-1}   p_1^{2}  \\
d &\mapsto      p_1^{-3} p_2   p_1^{2}      \\
e  &\mapsto   p_1 p_2    p_1^{-2}          \\
f  &\mapsto  p_1^{3}  p_2^{-1} p_1^{-2}   \\ 
g   &\mapsto    p_1^{-2}    p_3^{-3}    p_1^{-1}      \\
h &\mapsto   p_1^{-2}  p_3 p_4  p_3^2    p_1^2           \\
i &\mapsto  p_1^{-2} p_3^{-1}  p_4  p_3^{-2}  p_1^2.
\end{cases} 
$$

We obtain a map $\varphi^*:\XX(\Gamma,\SUtwo)\to \XX(\pi_1(\Plat),\SUtwo)$ by $\varphi^*([\rho])=[\rho\circ \varphi]$.   We will also denote the corresponding map $\mathsf{Hom}(\Gamma,\SUtwo)\to \mathsf{Hom}(\pi_1(\Plat),\SUtwo)$ by $\varphi^*$.

Let $\pi:\mathsf{Hom}(\Gamma,\SUtwo)\to \XX(\Gamma, \SUtwo)$. We wish to classify the irreducible representations in $\varphi^*\left(\pi(\mathsf{Hom}(\Gamma, \biTet))\right)$ as these points will be fixed by the group generated by $S^2$ and $T^2$.

Let $F_4$ be a free group of rank 4.  Then $|\mathsf{Hom}(F_4,\biTet)|=(24)^4=331,776$.  Letting homeomorphisms in $\mathsf{Hom}(F_4,\biTet)$ be represented by tuples $(A,B,C,D)$, we have that $$\mathsf{Hom}(\Gamma,\biTet)=\{(A,B,C,D)\in \biTet^4\ |\ A^6=B^6=C^6=D^2=I\},$$ where $I$ is the identity matrix in $\SUtwo$.  Running a for-loop in {\it Mathematica} on the set $\mathsf{Hom}(F_4,\biTet)$, checking the conditions above, results in $|\mathsf{Hom}(\Gamma,\biTet)|=456$.

Next, for each of the 456 representations in $\mathsf{Hom}(\Gamma,\biTet)$, we use the formulae defining $\varphi$ to determine $\varphi^*(\mathsf{Hom}(\Gamma,\biTet))\subset \mathsf{Hom}(\pi_1(\Plat),\SUtwo)$.  Since the map $\varphi^*$ is not injective, we don't have 456 homeomorphisms in the image, we end up with 132 (again obtained via a for-loop run in {\it Mathematica}).  

Our next task is to determine which of these 132 homomorphisms are irreducible as these will correspond to smooth points in the character variety.

Running a for-loop in {\it Mathematica}, removing homomorphisms that are simultaneously diagonalizable, leaves us with 96 of the 132.  Lastly, again with a for-loop in {\it Mathematica}, we determine there are exactly 4 $\biOct$-conjugacy classes among the 96 irreducible homomorphisms. We also observe that each of these 4 classes is represented by a representation that maps onto $\Qu$. From Lemma \ref{lem:normalizer}, we have there are exactly 4 irreducible representations in $\varphi^*\left(\pi(\mathsf{Hom}(\Gamma, \biTet))\right)\subset \XX(\pi_1(\Plat),\SUtwo)$.  

Thus, since irreducible representations are smooth, we have classified the smooth points in the character variety $\XX(\pi_1(\Plat),\SUtwo)$ that are $\biTet$-valued and correspond to the branched cover of the pillow-case, with $(6,6,6,2)$-torsion at the four singularities, by $\Plat$.

Let $\mathbf{1}:=\left(
\begin{array}{cc}
 1 & 0 \\
 0 & 1 \\
\end{array}
\right)$, $\mathbf{i}:=\left(
\begin{array}{cc}
 i & 0 \\
 0 & -i \\
\end{array}
\right)$, $\mathbf{j}:=\left(
\begin{array}{cc}
 0 & i \\
 i & 0 \\
\end{array}
\right)$, and $\mathbf{k}:=\left(
\begin{array}{cc}
 0 & -1 \\
 1 & 0 \\
\end{array}
\right)$.

Here representatives of the four irreducible representations described above:
\begin{align*}
&(\mathbf{j},\mathbf{i},-\mathbf{k},-\mathbf{j},\mathbf{k},-\mathbf{i},-\mathbf{1},\mathbf{1},\mathbf{1}),\\
&(\mathbf{j},\mathbf{i},-\mathbf{k},-\mathbf{j},\mathbf{k},-\mathbf{i},\mathbf{1},\mathbf{1},\mathbf{1}),\\
&(-\mathbf{j},-\mathbf{i},\mathbf{k},\mathbf{j},-\mathbf{k},\mathbf{i},\mathbf{1},-\mathbf{1},-\mathbf{1}),\\
&(-\mathbf{j},-\mathbf{i},\mathbf{k},\mathbf{j},-\mathbf{k},\mathbf{i},-\mathbf{1},-\mathbf{1},-\mathbf{1}).
\end{align*}

\begin{remark}
Notice that the first two and the second two only differ in the seventh component matrix by a factor of $-I$ and that the first group of two is a scalar multiple (in every component) of the second group of two by $-I$.  Thus there is a Klein group action of $\mathbb{Z}_2\times \mathbb{Z}_2$ on this set of four points by central matrices (and so commutes with conjugation).  This accounts for the fact that they are locally (on tangent spaces) identical (since the differential does not see the Klein group action).  In other words, the cocycle basis described for the fourth representation listed above, in the main body of the text, gives an identical cocycle basis (and coboundary basis) for each of these four representations (for a sanity check we directly verified this computationally).  Moreover, the differential matrices of $T^2$ and $S^2$ are {\it exactly} the same for all four of these representations.
\end{remark}

Lastly, we can then check (again via computation in {\it Mathematica}) that all 4 of these points are fixed (up to conjugation by $\biTet$) by both $S^2$ and $T^2$. 

\providecommand{\bysame}{\leavevmode\hbox to3em{\hrulefill}\thinspace}
\providecommand{\MR}{\relax\ifhmode\unskip\space\fi MR }
\providecommand{\MRhref}[2]{%
  \href{http://www.ams.org/mathscinet-getitem?mr=#1}{#2}
}
\providecommand{\href}[2]{#2}

\listoffigures
\end{document}

%% file: image1.pdf_tex
\begingroup%
  \makeatletter%
  \providecommand\color[2][]{%
    \errmessage{(Inkscape) Color is used for the text in Inkscape, but the package 'color.sty' is not loaded}%
    \renewcommand\color[2][]{}%
  }%
  \providecommand\transparent[1]{%
    \errmessage{(Inkscape) Transparency is used (non-zero) for the text in Inkscape, but the package 'transparent.sty' is not loaded}%
    \renewcommand\transparent[1]{}%
  }%
  \providecommand\rotatebox[2]{#2}%
  \newcommand*\fsize{\dimexpr\f@size pt\relax}%
  \newcommand*\lineheight[1]{\fontsize{\fsize}{#1\fsize}\selectfont}%
  \ifx\svgwidth\undefined%
    \setlength{\unitlength}{552.99423098bp}%
    \ifx\svgscale\undefined%
      \relax%
    \else%
      \setlength{\unitlength}{\unitlength * \real{\svgscale}}%
    \fi%
  \else%
    \setlength{\unitlength}{\svgwidth}%
  \fi%
  \global\let\svgwidth\undefined%
  \global\let\svgscale\undefined%
  \makeatother%
  \begin{picture}(1,0.72437092)%
    \lineheight{1}%
    \setlength\tabcolsep{0pt}%
    \put(0.74265479,0.56905757){\color[rgb]{0,0,0}\makebox(0,0)[lt]{\lineheight{1.25}\smash{\begin{tabular}[t]{l}$\mathbb{T}^ 2$\end{tabular}}}}%
    \put(0.74265479,0.1197087){\color[rgb]{0,0,0}\makebox(0,0)[lt]{\lineheight{1.25}\smash{\begin{tabular}[t]{l}$\mathbb{S}^ 2_{(2,2,2,2)}$\end{tabular}}}}%
    \put(0,0){\includegraphics[width=\unitlength,page=1]{image1.pdf}}%
  \end{picture}%
\endgroup%

%% file: image2.pdf_tex
\begingroup%
  \makeatletter%
  \providecommand\color[2][]{%
    \errmessage{(Inkscape) Color is used for the text in Inkscape, but the package 'color.sty' is not loaded}%
    \renewcommand\color[2][]{}%
  }%
  \providecommand\transparent[1]{%
    \errmessage{(Inkscape) Transparency is used (non-zero) for the text in Inkscape, but the package 'transparent.sty' is not loaded}%
    \renewcommand\transparent[1]{}%
  }%
  \providecommand\rotatebox[2]{#2}%
  \newcommand*\fsize{\dimexpr\f@size pt\relax}%
  \newcommand*\lineheight[1]{\fontsize{\fsize}{#1\fsize}\selectfont}%
  \ifx\svgwidth\undefined%
    \setlength{\unitlength}{913.38939276bp}%
    \ifx\svgscale\undefined%
      \relax%
    \else%
      \setlength{\unitlength}{\unitlength * \real{\svgscale}}%
    \fi%
  \else%
    \setlength{\unitlength}{\svgwidth}%
  \fi%
  \global\let\svgwidth\undefined%
  \global\let\svgscale\undefined%
  \makeatother%
  \begin{picture}(1,0.95931296)%
    \lineheight{1}%
    \setlength\tabcolsep{0pt}%
    \put(0,0){\includegraphics[width=\unitlength,page=1]{image2.pdf}}%
    \put(0.39256351,0.8664897){\color[rgb]{0,0,0}\makebox(0,0)[lt]{\lineheight{1.25}\smash{\begin{tabular}[t]{l}$\Sigma_3$\end{tabular}}}}%
    \put(0.46483721,0.37535465){\color[rgb]{0,0,0}\makebox(0,0)[lt]{\lineheight{1.25}\smash{\begin{tabular}[t]{l}$\mathbb{T}^2_{(2,2,2,2)}$\end{tabular}}}}%
    \put(0.47144305,0.22583178){\color[rgb]{0,0,0}\makebox(0,0)[lt]{\lineheight{1.25}\smash{\begin{tabular}[t]{l}$\mathbb{S}^2_{(4,4,4,4)}$\end{tabular}}}}%
    \put(0,0){\includegraphics[width=\unitlength,page=2]{image2.pdf}}%
  \end{picture}%
\endgroup%

%% file: picture3.pdf_tex
\begingroup%
  \makeatletter%
  \providecommand\color[2][]{%
    \errmessage{(Inkscape) Color is used for the text in Inkscape, but the package 'color.sty' is not loaded}%
    \renewcommand\color[2][]{}%
  }%
  \providecommand\transparent[1]{%
    \errmessage{(Inkscape) Transparency is used (non-zero) for the text in Inkscape, but the package 'transparent.sty' is not loaded}%
    \renewcommand\transparent[1]{}%
  }%
  \providecommand\rotatebox[2]{#2}%
  \newcommand*\fsize{\dimexpr\f@size pt\relax}%
  \newcommand*\lineheight[1]{\fontsize{\fsize}{#1\fsize}\selectfont}%
  \ifx\svgwidth\undefined%
    \setlength{\unitlength}{353.6142184bp}%
    \ifx\svgscale\undefined%
      \relax%
    \else%
      \setlength{\unitlength}{\unitlength * \real{\svgscale}}%
    \fi%
  \else%
    \setlength{\unitlength}{\svgwidth}%
  \fi%
  \global\let\svgwidth\undefined%
  \global\let\svgscale\undefined%
  \makeatother%
  \begin{picture}(1,0.72852723)%
    \lineheight{1}%
    \setlength\tabcolsep{0pt}%
    \put(0,0){\includegraphics[width=\unitlength,page=1]{picture3.pdf}}%
    \put(0.06802484,0.0085446){\color[rgb]{0,0,0}\makebox(0,0)[lt]{\lineheight{1.25}\smash{\begin{tabular}[t]{l}$A_{1,1}$\end{tabular}}}}%
    \put(0.42728784,0.01162798){\color[rgb]{0,0,0}\makebox(0,0)[lt]{\lineheight{1.25}\smash{\begin{tabular}[t]{l}$A_{0,1}$\end{tabular}}}}%
    \put(0.42728784,0.70010973){\color[rgb]{0,0,0}\makebox(0,0)[lt]{\lineheight{1.25}\smash{\begin{tabular}[t]{l}$A_{0,1}$\end{tabular}}}}%
    \put(0.82557086,0.35512861){\color[rgb]{0,0,0}\makebox(0,0)[lt]{\lineheight{1.25}\smash{\begin{tabular}[t]{l}$A_{1,0}$\end{tabular}}}}%
    \put(-0.00208782,0.35512861){\color[rgb]{0,0,0}\makebox(0,0)[lt]{\lineheight{1.25}\smash{\begin{tabular}[t]{l}$A_{1,0}$\end{tabular}}}}%
    \put(0.47208782,0.37512861){\color[rgb]{0,0,0}\makebox(0,0)[lt]{\lineheight{1.25}\smash{\begin{tabular}[t]{l}$A_{i}$\end{tabular}}}}%
    \put(0.24877978,0.03667607){\color[rgb]{0,0,0}\makebox(0,0)[lt]{\lineheight{1.25}\smash{\begin{tabular}[t]{l}$\sigma_{i}$\end{tabular}}}}%
    \put(0.59387102,0.03667607){\color[rgb]{0,0,0}\makebox(0,0)[lt]{\lineheight{1.25}\smash{\begin{tabular}[t]{l}$\sigma'_{i}$\end{tabular}}}}%
    \put(0.58209257,0.67333657){\color[rgb]{0,0,0}\makebox(0,0)[lt]{\lineheight{1.25}\smash{\begin{tabular}[t]{l}$\sigma'_{i-1}$\end{tabular}}}}%
    \put(0.82217704,0.50824902){\color[rgb]{0,0,0}\makebox(0,0)[lt]{\lineheight{1.25}\smash{\begin{tabular}[t]{l}$\zeta'_{i+1}$\end{tabular}}}}%
    \put(0.82217704,0.20028302){\color[rgb]{0,0,0}\makebox(0,0)[lt]{\lineheight{1.25}\smash{\begin{tabular}[t]{l}$\zeta_{i-1}$\end{tabular}}}}%
    \put(0.00300216,0.20028296){\color[rgb]{0,0,0}\makebox(0,0)[lt]{\lineheight{1.25}\smash{\begin{tabular}[t]{l}$\zeta_{i}$\end{tabular}}}}%
    \put(0.23373569,0.67333657){\color[rgb]{0,0,0}\makebox(0,0)[lt]{\lineheight{1.25}\smash{\begin{tabular}[t]{l}$\sigma_{i+1}$\end{tabular}}}}%
    \put(0.00300216,0.50824902){\color[rgb]{0,0,0}\makebox(0,0)[lt]{\lineheight{1.25}\smash{\begin{tabular}[t]{l}$\zeta'_{i}$\end{tabular}}}}%
  \end{picture}%
\endgroup%

%% file: picture4.pdf_tex
\begingroup%
  \makeatletter%
  \providecommand\color[2][]{%
    \errmessage{(Inkscape) Color is used for the text in Inkscape, but the package 'color.sty' is not loaded}%
    \renewcommand\color[2][]{}%
  }%
  \providecommand\transparent[1]{%
    \errmessage{(Inkscape) Transparency is used (non-zero) for the text in Inkscape, but the package 'transparent.sty' is not loaded}%
    \renewcommand\transparent[1]{}%
  }%
  \providecommand\rotatebox[2]{#2}%
  \newcommand*\fsize{\dimexpr\f@size pt\relax}%
  \newcommand*\lineheight[1]{\fontsize{\fsize}{#1\fsize}\selectfont}%
  \ifx\svgwidth\undefined%
    \setlength{\unitlength}{700.86497389bp}%
    \ifx\svgscale\undefined%
      \relax%
    \else%
      \setlength{\unitlength}{\unitlength * \real{\svgscale}}%
    \fi%
  \else%
    \setlength{\unitlength}{\svgwidth}%
  \fi%
  \global\let\svgwidth\undefined%
  \global\let\svgscale\undefined%
  \makeatother%
  \begin{picture}(1,0.63277171)%
    \lineheight{1}%
    \setlength\tabcolsep{0pt}%
    \put(0,0){\includegraphics[width=\unitlength,page=1]{picture4.pdf}}%
    \put(0.25397904,0.62253048){\color[rgb]{0,0,0}\makebox(0,0)[lt]{\lineheight{1.25}\smash{\begin{tabular}[t]{l}$\sigma'_{2}$\end{tabular}}}}%
    \put(0.52555323,0.62253048){\color[rgb]{0,0,0}\makebox(0,0)[lt]{\lineheight{1.25}\smash{\begin{tabular}[t]{l}$\sigma'_{0}$\end{tabular}}}}%
    \put(0.79336385,0.62253048){\color[rgb]{0,0,0}\makebox(0,0)[lt]{\lineheight{1.25}\smash{\begin{tabular}[t]{l}$\sigma'_{1}$\end{tabular}}}}%
    \put(0.25397904,0.00307931){\color[rgb]{0,0,0}\makebox(0,0)[lt]{\lineheight{1.25}\smash{\begin{tabular}[t]{l}$\sigma'_{2}$\end{tabular}}}}%
    \put(0.52555323,0.00307931){\color[rgb]{0,0,0}\makebox(0,0)[lt]{\lineheight{1.25}\smash{\begin{tabular}[t]{l}$\sigma'_{1}$\end{tabular}}}}%
    \put(0.79336385,0.00307931){\color[rgb]{0,0,0}\makebox(0,0)[lt]{\lineheight{1.25}\smash{\begin{tabular}[t]{l}$\sigma'_{0}$\end{tabular}}}}%
    \put(0.11915558,0.62253048){\color[rgb]{0,0,0}\makebox(0,0)[lt]{\lineheight{1.25}\smash{\begin{tabular}[t]{l}$\sigma_{1}$\end{tabular}}}}%
    \put(0.39641441,0.62253048){\color[rgb]{0,0,0}\makebox(0,0)[lt]{\lineheight{1.25}\smash{\begin{tabular}[t]{l}$\sigma_{2}$\end{tabular}}}}%
    \put(0.66950732,0.62253048){\color[rgb]{0,0,0}\makebox(0,0)[lt]{\lineheight{1.25}\smash{\begin{tabular}[t]{l}$\sigma_{0}$\end{tabular}}}}%
    \put(0.11915558,0.00307931){\color[rgb]{0,0,0}\makebox(0,0)[lt]{\lineheight{1.25}\smash{\begin{tabular}[t]{l}$\sigma_{2}$\end{tabular}}}}%
    \put(0.39641441,0.00307931){\color[rgb]{0,0,0}\makebox(0,0)[lt]{\lineheight{1.25}\smash{\begin{tabular}[t]{l}$\sigma_{1}$\end{tabular}}}}%
    \put(0.66950732,0.00307931){\color[rgb]{0,0,0}\makebox(0,0)[lt]{\lineheight{1.25}\smash{\begin{tabular}[t]{l}$\sigma_{0}$\end{tabular}}}}%
    \put(0.18507761,0.35465782){\color[rgb]{0,0,0}\makebox(0,0)[lt]{\lineheight{1.25}\smash{\begin{tabular}[t]{l} $\vec{A_0}$\end{tabular}}}}%
    \put(0.74269582,0.26724123){\color[rgb]{0,0,0}\makebox(0,0)[lt]{\lineheight{1.25}\smash{\begin{tabular}[t]{l} $\vec{A_0}$\end{tabular}}}}%
    \put(0.45745374,0.35465782){\color[rgb]{0,0,0}\makebox(0,0)[lt]{\lineheight{1.25}\smash{\begin{tabular}[t]{l} $\vec{A_1}$\end{tabular}}}}%
    \put(0.45745374,0.26724123){\color[rgb]{0,0,0}\makebox(0,0)[lt]{\lineheight{1.25}\smash{\begin{tabular}[t]{l} $\vec{A_1}$\end{tabular}}}}%
    \put(0.74269582,0.35465782){\color[rgb]{0,0,0}\makebox(0,0)[lt]{\lineheight{1.25}\smash{\begin{tabular}[t]{l} $\vec{A_2}$\end{tabular}}}}%
    \put(0.18507761,0.26724123){\color[rgb]{0,0,0}\makebox(0,0)[lt]{\lineheight{1.25}\smash{\begin{tabular}[t]{l} $\vec{A_2}$\end{tabular}}}}%
    \put(0.90675032,0.13359676){\color[rgb]{0,0,0}\makebox(0,0)[lt]{\lineheight{1.25}\smash{\begin{tabular}[t]{l}$\zeta_{2}$\end{tabular}}}}%
    \put(-0.0009029,0.13359676){\color[rgb]{0,0,0}\makebox(0,0)[lt]{\lineheight{1.25}\smash{\begin{tabular}[t]{l}$\zeta_{2}$\end{tabular}}}}%
    \put(0.90675032,0.49477256){\color[rgb]{0,0,0}\makebox(0,0)[lt]{\lineheight{1.25}\smash{\begin{tabular}[t]{l}$\zeta'_{0}$\end{tabular}}}}%
    \put(-0.0009029,0.49477256){\color[rgb]{0,0,0}\makebox(0,0)[lt]{\lineheight{1.25}\smash{\begin{tabular}[t]{l}$\zeta'_{0}$\end{tabular}}}}%
  \end{picture}%
\endgroup%

%% file: picture5.pdf_tex
\begingroup%
  \makeatletter%
  \providecommand\color[2][]{%
    \errmessage{(Inkscape) Color is used for the text in Inkscape, but the package 'color.sty' is not loaded}%
    \renewcommand\color[2][]{}%
  }%
  \providecommand\transparent[1]{%
    \errmessage{(Inkscape) Transparency is used (non-zero) for the text in Inkscape, but the package 'transparent.sty' is not loaded}%
    \renewcommand\transparent[1]{}%
  }%
  \providecommand\rotatebox[2]{#2}%
  \newcommand*\fsize{\dimexpr\f@size pt\relax}%
  \newcommand*\lineheight[1]{\fontsize{\fsize}{#1\fsize}\selectfont}%
  \ifx\svgwidth\undefined%
    \setlength{\unitlength}{728.87591505bp}%
    \ifx\svgscale\undefined%
      \relax%
    \else%
      \setlength{\unitlength}{\unitlength * \real{\svgscale}}%
    \fi%
  \else%
    \setlength{\unitlength}{\svgwidth}%
  \fi%
  \global\let\svgwidth\undefined%
  \global\let\svgscale\undefined%
  \makeatother%
  \begin{picture}(1,0.47943572)%
    \lineheight{1}%
    \setlength\tabcolsep{0pt}%
    \put(0,0){\includegraphics[width=\unitlength,page=1]{picture5.pdf}}%
    \put(0.11851555,0.06615498){\color[rgb]{0,0,0}\makebox(0,0)[lt]{\lineheight{1.25}\smash{\begin{tabular}[t]{l}$\sigma_{i}$\end{tabular}}}}%
    \put(0.2859366,0.06615498){\color[rgb]{0,0,0}\makebox(0,0)[lt]{\lineheight{1.25}\smash{\begin{tabular}[t]{l}$\sigma'_{i}$\end{tabular}}}}%
    \put(0.28022229,0.37503087){\color[rgb]{0,0,0}\makebox(0,0)[lt]{\lineheight{1.25}\smash{\begin{tabular}[t]{l}$\sigma'_{i-1}$\end{tabular}}}}%
    \put(0.39669929,0.29493861){\color[rgb]{0,0,0}\makebox(0,0)[lt]{\lineheight{1.25}\smash{\begin{tabular}[t]{l}$\zeta'_{i+1}$\end{tabular}}}}%
    \put(0.39669929,0.14552888){\color[rgb]{0,0,0}\makebox(0,0)[lt]{\lineheight{1.25}\smash{\begin{tabular}[t]{l}$\zeta_{i-1}$\end{tabular}}}}%
    \put(-0.0007235,0.14552885){\color[rgb]{0,0,0}\makebox(0,0)[lt]{\lineheight{1.25}\smash{\begin{tabular}[t]{l}$\zeta_{i}$\end{tabular}}}}%
    \put(0.11121691,0.37503087){\color[rgb]{0,0,0}\makebox(0,0)[lt]{\lineheight{1.25}\smash{\begin{tabular}[t]{l}$\sigma_{i+1}$\end{tabular}}}}%
    \put(-0.0007235,0.29493861){\color[rgb]{0,0,0}\makebox(0,0)[lt]{\lineheight{1.25}\smash{\begin{tabular}[t]{l}$\zeta'_{i}$\end{tabular}}}}%
    \put(0,0){\includegraphics[width=\unitlength,page=2]{picture5.pdf}}%
    \put(0.67189823,0.17550153){\color[rgb]{0,0,0}\makebox(0,0)[lt]{\lineheight{1.25}\smash{\begin{tabular}[t]{l}$p_1$\end{tabular}}}}%
    \put(0.93578083,0.17550153){\color[rgb]{0,0,0}\makebox(0,0)[lt]{\lineheight{1.25}\smash{\begin{tabular}[t]{l}$p_2$\end{tabular}}}}%
    \put(0.67189823,0.00532127){\color[rgb]{0,0,0}\makebox(0,0)[lt]{\lineheight{1.25}\smash{\begin{tabular}[t]{l}$p_3$\end{tabular}}}}%
    \put(0.93578083,0.00532127){\color[rgb]{0,0,0}\makebox(0,0)[lt]{\lineheight{1.25}\smash{\begin{tabular}[t]{l}$p_4$\end{tabular}}}}%
  \end{picture}%
\endgroup%